\documentclass[12pt,thmsa]{amsart}

\usepackage[utf8]{inputenc}
\usepackage{amsmath}
\usepackage{amsthm}
\usepackage{graphicx}
\usepackage{amsfonts}
\usepackage{amssymb}
\usepackage{float}
\newtheorem{theorem}{Theorem}
\newtheorem{proposition}{Proposition}
\newtheorem{lemma}[proposition]{Lemma}
\newtheorem{corollary}[proposition]{Corollary}
\theoremstyle{definition}
\newtheorem{definition}[proposition]{Definition}
\theoremstyle{remark}
\newtheorem{remark}{Remark}

\newcommand{\C}{\mathbb C}
\newcommand{\D}{\mathbb D}
\newcommand{\N}{\mathbb N}

\newcommand{\Q}{\mathbb Q}
\newcommand{\R}{\mathbb R}
\newcommand{\Z}{\mathbb Z}
\newcommand{\ZZ}{\mathbb Z[i]}

\newcommand{\UU}{\mathbb U}
\newcommand{\ttt}{\theta}
\newcommand{\LL}{\Lambda}
\newcommand{\fff}{\rightarrow}

\newcommand{\CC}{\mathcal C}
\newcommand{\DD}{\mathcal D}

\newcommand{\SL}{\operatorname{SL}}

\newcommand{\GL}{\operatorname{GL}}

\newcommand{\dd}{\operatorname{d}}

\newcommand{\eps}{\varepsilon}
\newcommand{\cc}{\bold{C}}

\begin{document}

%%  "Title of the Paper"
\title{Gauss lattices and complex continued fractions}
\maketitle

%%  \author{\fnms{John} \snm{Smith}\thanksref{t2}\ead[label=e1]{smith@foo.com}\ead[label=e2,url]{www.foo.com}}
%%  \thankstext{t2}{The author is supported by ...}
%%  \address{line 1\\ line 2\\ \printead{e1}\\\printead{e2}}
\begin{center}
\author{\large Chevallier Nicolas}\\
\medskip
\address{nicolas.chevallier@uha.fr}
\end{center}
%\author{\fnms{???} \snm{???}\thanksref{t2}\ead[label=e2]{???}}
%\address{???\\\printead{e2}}
%\and
%\author{\fnms{???} \snm{???}%
%        \ead[label=e3]{???}%
%        \ead[label=u1,url]{???}}
%\address{???\\\printead{e3}\\\printead{u1}}
%\thankstext{t2}{The author is supported by ...}

%%  History:
%\received{\sday{3} \smonth{1} \syear{2019}}

%\begin{abstract}
%\end{abstract}
\begin{abstract}
	Our aim is to construct a complex continued fraction algorithm finding all the best Diophantine approximations to a complex number. Using the sequence of minimal vectors in a two dimensional lattice over the ring of Gaussian integers, we obtain  an algorithm defined on a submanifold of the space of unimodular two dimensional Gauss lattices. This submanifold is transverse to the diagonal flow. Thanks to the correspondence between minimal vectors and best Diophantine approximations, the algorithm we propose finds all the best approximations to a complex number by construction. A byproduct of the algorithm is the best constant  for the complex version of Dirichlet's theorem about approximations of complex numbers by quotients of two Gaussian integers.
\end{abstract}

%\tableofcontents

%%  The body
\section{Introduction}
Let us start with a very brief and partial account of the history of complex continued fractions (see \cite{Os}, \cite{OsSt} or \cite{Ro} for  detailed historical accounts). Since the pioneering works \cite{Mi}, \cite{HuA} and \cite{HuJ} of N. Michelangeli in 1887,  Adolf Hurwitz in 1888 and  Julius Hurwitz in   1895, complex continued fractions have been considered by many authors during the  20th century and at the beginning of the 21st century.  
Adolf Hurwitz considered continued fractions with partial quotients in a ``system'' $S\subset \C$  and the work of Julius Hurwitz used the $(1+i)\ZZ$ subring of the ring of Gaussian integer, see also \cite{Lu}. 
An important contribution of Adolf Hurwitz concerned the continued fractions associated with  the ring of Gaussian integers using the nearest Gaussian integer (we refer to it as Adolf Hurwitz continued fraction). In his 1888 article, Adolf Hurwitz showed the non-trivial fact that  the sequence of moduli of the denominators of such a continued fraction is increasing. 
In 1973,  R. Lakein \cite{Lak}  studied complex continued fractions associated with the rings of integers of the quadratic number fields $\Q[\sqrt{-1}]$, $\Q[\sqrt{-3}]$, $\Q[\sqrt{-7}]$ and $\Q[\sqrt{-11}]$ (the imaginary quadratic fields with Euclidean rings of integers). 
For instance, he proved that the convergents associated with the Adolf Hurwitz  continued fraction algorithm are best approximations for all complex numbers not in a countable family of lines and circles. 
About at the same time, in 1975, A. Schmidt proposed an algorithm based on the concept of Farey sets very different from the A. Hurwitz continued fraction algorithm, see \cite{Sch}. In 1985, A. Tanaka proposed a complex continued fractions algorithm (\cite{Ta}) which turned out to be a new version of the Julius Hurwitz continued fraction algorithm, see \cite{Os}. 
More recently, D. Hensley  produced  complex numbers, solutions of  irreducible quartic equations over $\Q[i]$, with a bounded, not ultimately periodic sequence  of ``Adolf Hurwitz'' partial quotients, see \cite{Hen}. In 2014,
S. G. Dani and A. Nogueira, \cite{DaNo}, proposed a general approach to complex continued fractions associated with the ring of Gaussian integers. Their approach has been taken up by other authors.  
In 2019,  H. Ei, S. Ito, H. Nakada and R. Natsui studied the construction of the
natural extension of the Hurwitz complex continued fraction map, see \cite{EiNa}. Beside they proved a ``Legendre's theorem'' for Adolf Hurwitz continued fractions.
Also, in 2018, the PHD thesis of G. G. Robert \cite{Ro} gave an almost complete overview of known results and many interesting new results. \smallskip

In this work we propose a lattice approach to complex continued fractions associated with the ring of Gaussian integers. In particular, we address the question of finding all the best approximations of a complex number.

The lattice approach goes back to C. Hermite (\cite{Her}) and G. Voronoï (\cite{Vor}). In a sequence of works, among which \cite{Lag1,Lag2,Lag3}, J. C. Lagarias studied the best simultaneous Diophantine approximations and clearly stated the connections with the shortest vectors in lattices  \cite{Lag1}, see also \cite{Che}.  \smallskip

Our starting point comes from the ordinary real continued fractions.  
In the space of dimension two unimodular lattices, $\SL(2,\R)/\SL(2,\Z)$, let us consider the subset of lattices  whose two minima with respect to the sup norm are equal. It is known that the first return map on this subset induced by the left action of the diagonal flow $g_t=\begin{pmatrix}
	e^t&0\\0&e^{-t}
\end{pmatrix}$, $t\in\R$, is a two-fold extension of the natural extension of the Gauss map $x\fff\{1/x\}$ (see  \cite{CheChe} or \cite{EiWa} where another version of the natural extension is given).  In the complex case, we shall  use the exact same idea where the space $\SL(2,\R)/\SL(2,\Z)$ is replaced by   the space $\SL(2,\C)/\SL(2,\ZZ)$ of unimodular lattices in $\C^2$. Like in the real case, we shall exploit two basic correspondences: 
\begin{itemize}	
	\item The correspondence between pairs of consecutive minimal vectors in a lattice and the intersection of the orbits of the flow $g_t$  with the transversal
	\[
	T=\{\LL\in \SL(2,\C)/\SL(2,\ZZ):\lambda_1(\LL)=\lambda_2(\LL)\}
	\]
	where the minima are associated with the sup norm in $\C^2$  (see Lemma \ref{lem:consecutivetransversal}, notice that $\lambda_i(\LL)$, $i=1,2$ are the complex minima of the lattice $\LL$, see definition \ref{def:minima} in the appendix). Actually, we shall use a slightly smaller transversal (see section \ref{sec:transversal}).
	\item 
	The correspondence between best approximations and minimal vectors (see Proposition \ref{prop:best}).
\end{itemize}
More precisely, for each lattice $\LL$ over the ring of Gaussian integers in $\C^2$, let us consider the set of {\sl minimal vectors}  in $\LL$, i.e., the nonzero vectors $u=(u_1,u_2)\in \LL$ such that for any nonzero $z=(z_1,z_2)\in \LL$, 
\[
|z_1|\leq |u_1| \text{ and } |z_2|\leq |u_2| \Rightarrow |z_1|= |u_1|
\text{ and } |z_2|= |u_2|.
\]
Let us order these minimal vectors according to the moduli of their second coordinate.  It is not difficult to prove that
when $u=(u_1,u_2)$ and $v=(v_1,v_2)$ are two consecutive minimal vectors in this sequence, then the interior of the cylinder
\[
C(u,v)=\{(z_1,z_2):|z_1|\leq\max(|u_1|,|v_1|),\,|z_2|\leq \max(|u_2|,|v_2|)|\}
\]
does not contain any nonzero element of $\LL$  (see Lemma \ref{lem:consecutive}).
Then, we can find  a real number $t$ such that the action of $g_t$ transforms the cylinder $C(u,v)$ into a cylinder of same width and height. 
For this value $t$, the action of $g_t$ on $u$ and $v$ gives two new vectors $u'$ and $v'$ with sup norms $|u'|_{\infty}=|v'|_{\infty}=\lambda_1(g_t\LL)=\lambda_2(g_t\LL)$.  
Thus, the new lattice $g_t\LL$ is in the transversal $T$.

For a  lattice of the shape
\[
\LL_{\ttt}=\begin{pmatrix}
	1& -\ttt\\
	0&1
\end{pmatrix}
\Z[i]^2,
\]  
the sequence of minimal vectors gives all the best approximation vectors of the complex number $\ttt$, see Proposition \ref{prop:best}. 
Thanks to the aforementioned  work of R. Lakein, we know  that  the convergents associated with $\ttt$ by Adolf Hurwitz's continued fraction expansion are best approximations for almost all $\ttt\in\C$. Thus, the sequence of convergents of Adolf Hurwitz's continued fraction expansion of $\ttt$ is given by a subsequence of the sequence of minimal vectors of the lattice $\LL_{\ttt}$. So we can consider the transversal together with the first return map like a complex continued fraction map in the space of lattices $\SL(2,\C)/\SL(2,\ZZ)$. 

An important difference with the real case is that two consecutive minimal vectors are no longer necessarily primitive. Our first significant result is 
\begin{theorem}\label{thm:index} If $u$ and $v$ are two consecutive minimal vectors of a  lattice over the ring of Gaussian integers $\LL$ in $\C^2$, then the sublattice $\Z[i]u+\Z[i]v$ is of index $1$ or $2$ in $\LL$.  Furthermore, when $\ZZ u+\ZZ v$ is of index two,
	\[
	\LL=\langle u,v\rangle_J\overset{def}=\{gu+hv:(g,h)\in\ZZ^2\cup J^2\}
	\]
	where $J=\tfrac1{1+i}\ZZ\setminus \ZZ$.
\end{theorem}
Observe that, in case of index two, the lattice $\LL$ is like a centered cubic lattice, where the index two ideal $2\Z$ in $\Z$ is replaced by the index two ideal $(1+i)\ZZ$ in $\ZZ$. 

Our second  result is about the geometry of numbers for two dimensional lattices over the ring of Gaussian integers. It is the counterpart in the complex case of the easy result: 

If $u=(u_1,u_2)$ and $v=(v_1,v_2)$ are two linearly independent vectors in $\R^2$ then the interior of the rectangle $R(u,v)=\{(x_1,x_2):|x_1|\leq\max(|u_1|,|v_1|),\,|x_2|\leq \max(|u_2|,|v_2|)|\}$ contains no nonzero vector of the lattice $\Z u+\Z v$ iff $u$,$v$ and $u\pm v$ are not in the interior of $R(u,v)$.   

\begin{theorem}\label{thm:geonumber1}
	Let $u=(u_1,u_2)$ and $v=(v_1,v_2)$ be two  vectors in $\C^2$ such that $|u_1|>0$, $|u_1|\geq |v_1|$,  $|v_2|>0$ and $|v_2|\geq |u_2|$. 
	\begin{enumerate}
		\item 
		Zero is the only element of $\ZZ u+\ZZ v$ in the interior of the cylinder
		\[
		C(u,v)=\{(z_1,z_2):|z_1|\leq |u_1|,\,|z_2|\leq |v_2|\}
		\] iff $gu+hv\notin \overset{o}C(u,v)$ for all nonzero $g$, $h\in\ZZ^2$ with $|g|\times|h|\leq \sqrt 2$. 
		\item Zero is the only element of 
		$
		\langle u,v\rangle_J
		$  in the interior of the cylinder
		$
		C(u,v)
		$ iff  $gu+hv\notin \overset{o}C(u,v)$ for all  $(g,h)\in J^2$ with $|g|=|h|= \tfrac1{\sqrt 2}$.
	\end{enumerate}
\end{theorem}  
The proof of this theorem depends only on elementary geometry, but is not as simple as in the real case. We use a computer to rule out many cases.
We shall also give a variant of this result with strict inequality and  a slightly more precise corollary, see section \ref{sec:geonumber}.

Next theorem explains how to compute inductively the sequence of minimal vectors of a lattice over  the ring of Gaussian integers in $\C^2$.
Let us equip $\C^2$  with the {\sl lexicographic preorder}
\[
(x_1,x_2)\prec(y_1,y_2) 
\]
iff $|x_2|<|y_2|$ or $|x_2|=|y_2|$ and $|x_1|\leq |y_1|$.

\begin{theorem}[Continued fraction algorithm]\label{thm:continuedfraction}
	Let $u=(u_1,u_2)$ and $v=(v_1,v_2)$ be  two consecutive minimal  vectors in a unimodular lattice $\LL$ with $|u_2|<|v_2|$. Let  $w_1=\tfrac{v_1}{u_1}$ and $w_2=\tfrac{u_2}{v_2}$. If $w_1\neq 0$ then there exists  $v'\in\LL$  a minimal vector such that  $v$ and $v'$ are two consecutive minimal vectors and
	\begin{itemize}
		\item If $\det_{\C}(u,v)=1$, then $v'$ is  any vector that is minimal for the preoder $\prec$ in the set
		\begin{align*}
			E_1=\left\{z=-au+gv: a\in\{1,1+i\},\, g\in\Z[i],\,|\tfrac{a}{w_1}-g|<1\right\}.
		\end{align*} 
		Moreover with $u'=v=(u'_1,v'_2w'_2)$ and $v'=-au+gv=(u'_1w'_1,v'_2)$, we have 
		\begin{align}
			w'_1=g-\frac{a}{w_1}, \hspace{1cm} w'_2=\frac{1}{g-aw_2}.	
		\end{align}
		\item If $\det_{\C}(u,v)=1+i$, then $v'$ is  any vector that is minimal for the preoder $\prec$  in the set
		\begin{align*}
			E_2=\left\{z=-\tfrac{1}{1+i}(u+v)+gv:  g\in\Z[i],\,|\tfrac1{(1+i)w_1}+\tfrac1{(1+i)}-g|<1\right\}.
		\end{align*}  
		Moreover with $u'=v=(u'_1,v'_2w'_2)$ and $v'=-au+gv=(u'_1w'_1,v'_2)$, we have 
		\begin{align}
			w'_1=g-\tfrac1{(1+i)w_1}-\tfrac1{(1+i)}, \hspace{1cm} w'_2=\frac{1}{g-\tfrac1{(1+i)}w_2-\tfrac1{(1+i)}}.	
		\end{align} 
	\end{itemize}
\end{theorem} 
The  set $E_1$ in the above theorem has  eight elements at most  and $E_2$ has  four elements at most because there are at most four Gaussian integers $g$ such that $|g-w|<1$ for a given complex number $w$. Therefore, the map 
\[
T_G:(w_1,w_2)\fff (w'_1,w'_2)
\] is easy to compute. 
This map is the core of the first return map in the transversal, see Theorem \ref{thm:firstreturn}. In fact, the minimal vectors $u'$ and $v'$ can be easily computed  because   $u'_1=u_1w_1$, $v'_2=v_2/w'_2$ and  as explained before, it is possible to bring  the lattice $\LL$ in the transversal using the flow $g_t$ and the two consecutive minimal vectors $u'$ and $v'$.

In the first case  of Theorem \ref{thm:continuedfraction}, the new consecutive minimal vectors $u',v'$ have index $1$ or $2$ (determinant $1$ or $1+i$) because by Theorem \ref{thm:index} two consecutive minimal vectors have index $1$ or $2$.
In the second case of Theorem \ref{thm:continuedfraction}, the new consecutive minimal vectors $u',v'$ have index $1$ because $u,v$ have index $2$ and it is not possible that two consecutive pairs of minimal vectors $u,v$ and $u'=v,v'$ have both index $2$, see Proposition \ref{prop:index2}. It is worth noticing that the proof of this latter proposition uses Theorem \ref{thm:geonumber1}.

The map $T_G$ might have some  links with the natural extension of the Adolf Hurwitz map studied in \cite{EiNa}. 
Indeed, $T_G$ could be the natural extension of the unknown algorithm that computes all the best approximations of complex numbers while the map defined by H. Ei, S. Ito, H. Nakada and R. Natsui is  the natural extension of the Hurwitz map which gives only a subsequence of the sequence of best approximations (see again \cite{Lak}).

In the first case of  Theorem \ref{thm:continuedfraction} with $a=1$, the condition $|\tfrac{a}{w_1}-g|<1$ is the condition considered by Dani and Nogueira to define an approximation sequence, see \cite{DaNo}. In Theorem \ref{thm:continuedfraction}, the second variable controls the choice among the possible Gaussian integers $g$.

It is not that easy to have an explicit description of the transversal or of the domain of definition of $T_G$. However, with a good choice of the parametrization, this domain becomes a finite union of products of subsets in the complex plane whose boundaries are arcs of circle, see Figure 3 in  subsection \ref{subsec:constraints}. The domain of definition can be found thanks to Theorem \ref{thm:geonumber1}, see section \ref{sec:geonumber} where a  description of the domain is given. 
We also give the invariant measure of the first return map of the flow $g_t$ in the transversal. The open transversal is parametrized with three parameters $\ttt,w_1,w_2$ where $\ttt\in [0,\pi/2]$ and $(w_1,w_2)$ is in an open set included in $\D^2=\{z\in\C:|z|<1\}^2$.
\begin{theorem}[Invariant measure]\label{thm:densityinvariantmeasure1}
	Using the parametrization  of the transversal  (see section \ref{sec:transversal}), the measure $\nu$ induced by the 
	Haar measure  in $\SL(2,\C)/\SL(2,\ZZ)$ and the flow $g_{t}$, has the
	density
	\[
	h(\theta,w_{1},w_{2})=\frac{32}{|1-w_{1}w_{2}|^{4}}
	\]
	with respect to the Lebesgue measure of $[0,\pi/2]\times\mathbb{D}^{2}$.	
\end{theorem} 
The constant $32$ depends on our choice of the normalization of the Haar measure.

A byproduct of our work is the exact value of the one dimensional complex Dirichlet constant. To the best of our knowledge this constant was unknown.
\begin{theorem}[Complex Dirichlet constant]\label{thm:Dirichlet}
	For every complex number $z$ and for every  real number $Q\geq 1$,  there exist   Gaussian integers $p$ and $q$ such that
	\begin{align*}
		\left\{
		\begin{array}[{l}c]{l}
			0<|q|< Q,\\
			|qz-p|\leq \frac{\sqrt 2}{ 3-\sqrt 3}\times\frac{1}{Q},
		\end{array}	
		\right.
	\end{align*}
	where	$\frac{\sqrt 2}{ 3-\sqrt 3}=\frac{1}{\sqrt{6-3\sqrt 3}}=1.115355\dots$ .
	Furthermore the set of complex numbers $z$ for which the constant  $\frac{\sqrt 2}{ 3-\sqrt 3}$ can be improved, is of zero Lebesgue measure.
\end{theorem}
In fact, the optimality of the constant is slightly stronger.\\
{\sl
	{\bf Theorem 5 bis.} 
	For almost all $\ttt\in\C$, all $C<\frac{\sqrt 2}{ 3-\sqrt 3}$ and all $T\geq 1$, there exists $Q\geq T$ such that the system 
	\begin{align*}
		\left\{
		\begin{array}[{l}c]{l}
			0<|q|< Q,\\
			|qz-p|\leq C\times\frac{1}{Q},
		\end{array}	
		\right.
	\end{align*}
	has no solution with $p,q\in \ZZ$.
}\medskip

The essential ingredient of the proof of these latter theorems is Corollary \ref{cor:ConstraintsLarge} of Theorem \ref{thm:geonumber1} about the geometry of numbers (or the explicit description of the transversal). With this description we can show that the best Dirichlet constant is bounded above by $\frac{\sqrt 2}{ 3-\sqrt 3}$. To see that this constant is the best possible constant for almost all $\ttt\in\C$ we use an additional tool, the ergodicity of the diagonal flow $g_t$.\smallskip

It should be noticed that  the complex version of Hurwitz best constant $\tfrac1{\sqrt 5}$ is  known. In 1925, Lester Ford  \cite{Fo} proved that for all irrational complex numbers $z$ there exist infinitely many Gaussian integers $p$ and $q\neq 0$ such that $|z-\tfrac{p}{q}|<\tfrac1{\sqrt 3 |q|^2}$.  The constant $\tfrac1{\sqrt 3}$ is the best possible. Ford's proof did not use continued fractions and in 1975  R. Lakein gave a new proof of this result using complex continued fractions (see \cite{Lak2}). \medskip

The paper is organized as follows. We begin by some preliminaries on lattices over the ring of Gaussian integers, minimal vectors, the sequence of minimal vectors associated with a lattice and the relation between minimal vectors and best Diophantine approximations. 
Next, we prove the theorem about the index of consecutive  minimal vectors. In the next section, we prove the geometry of numbers' result, a more explicit version of this result (see Corollary \ref{cor:ConstraintsLarge}) and an  example showing that two linearly independent minimal vectors can both be successors of the same minimal vector. Thanks to Theorem \ref{thm:geonumber1}, we prove that two consecutive pairs of consecutive minimal vectors cannot have both index $2$, see Proposition \ref{prop:index2}.

Next, we define the transversal and a  parametrization of the transversal, then  we give explicit formulas for the  first return map in the transversal, see Theorem \ref{thm:continuedfraction} and \ref{thm:firstreturn}.
Then, we prove Theorem \ref{thm:densityinvariantmeasure1} about the density of the measure induced by the flow. Finally, we prove Dirichlet's theorem.
We finish the paper by two more small sections and an appendix. In the first of these sections we explain how the Gauss reduction algorithm of basis in two dimensional lattices can be used to find two consecutive minimal vectors.
The second section is devoted to a few open questions. 
The appendix is devoted to some basic facts about lattices over the ring of Gaussian integers. 
\smallskip

{\bf Acknowledgements.} The author  thanks Yann Bugeaud and Yitwah Cheung for their helpful comments and the referee for the very careful reading and for pointing out  missing pairs in the ``critical set'' defined in Lemma \ref{lem:computer1}.

\section{Preliminaries}
\subsection{Notations} We collect the notations that we shall use.
\begin{itemize}
	\item $|z|$ is the modulus of the complex number $z$ and $\arg z\in [0,2\pi[$ its argument.
	\item If $E$ is a subset in $\C$, $\overline E$ and $\overset{o}E$ denote its closure and its interior.
	Although we are working with complex numbers there should not be any confusion between closure and conjugate. Most of the time ``the bar" will be used for the closure.

	\item $\mathbb{D}$ denote the open unit disk in $\mathbb{C}$.
	$D(a,r)$ denote the closed disk of center $a\in\mathbb{C}$ and radius $r$ and $\D(a,r)$ the open disk of center $a$ and radius $r$.
	
	\item $\mathbf C(a,r)$ denote the circle of center $a\in\mathbb{C}$ and radius $r$.
	
	\item $|(z_1,z_2)|_{\infty}=\max(|z_1|,|z_2|)$ is the sup norm on $\C^2$ and $B_{\infty}(x,r)$ is the closed ball of radius $r$ and center $x$ in $\C^2$ associated with the sup norm.

	\item Let $a$ and $b$ be two non-negative real numbers and $u=(u_{1},u_{2})$ and $v=(v_{1},v_{2})$ be vectors in $\C^2$. We define the cylinders
	\begin{align*}
		C(a,b)&=\{(x,y)\in\C^2:|x|\leq|a|,|y|\leq|b|\},\\
		C(u)&=C(|u_1|,|u_2|),\\
		C(u,v)&=C(\max(|u_{1}|,|v_1|),\max(|u_2|,|v_{2}|)),\\
		C_1(a)&=\{(z_1,z_2)\in\C^2:|z_1|\leq a\},\\
		C_2(a)&=\{(z_1,z_2)\in\C^2:|z_2|\leq a\}. 
	\end{align*} 
	\item When $C(u,v)$ has nonempty interior,  $C(u,v)$ is the unit ball of a norm  $|.|_{u,v}$ defined on $\C^2$. Observe that for any $x=(x_1,x_2)\in \C^2$, 
	\[
	|x|_{u,v}=\max(\tfrac{|x_1|}{\max(|u_{1}|,|v_1|)},\tfrac{|x_2|}{\max(|u_2|,|v_{2}|)}).
	\]
	
	\item $\mathbb{U}_{n}=\{z\in\mathbb{C}:z^{n}=1\}$ is the group $n$-th roots of unity in $\mathbb{C}$.
	\item $\mathbb{D}_{8}$ is  the group of isometries acting on $\mathbb{C}$ generated by the
	multiplications by elements in $\mathbb{U}_{4}$ and by  conjugation.
	
	\item 
	$
	(x_1,x_2)\prec(y_1,y_2) 
	$
	iff $|x_2|<|y_2|$ or $|x_2|=|y_2|$ and $|x_1|\leq |y_1|$ is the lexicographic preorder on $\C^2$.
	
	\item When $A$ is a subset of $\C$ or $\C^2$, $A^*=A\setminus\{0\}$.
	
	\item $\ZZ=\Z+i\Z$, $I=(1+i)\mathbb{Z}[i]$ and  $J=\frac1{1+i}(\mathbb{Z}[i]\setminus I)$.
	\item For $u,v\in\C^2$, $\langle u,v\rangle_J\overset{def}=\{gu+hv:(g,h)\in\ZZ^2\cup J^2\}$
	
	\item We shall use also the following sets
	\begin{align*}
		\mathcal{C}= & \{z\in\mathbb{C}: |z|<1,\,\arg z\in[0,\tfrac{\pi}4 ]\}\\
		\mathcal{D}=  &  \{w\in\mathbb{C}:|z|<1,\, \operatorname{d}(w_{2},1)>
		1,\,\operatorname{d}(w_{2},1-i)> 1 \},\\
		\mathcal{T}=  &  \{w\in\mathbb{C}:|z|<1,\, \operatorname{d}(w_{2}%
		,1)>\sqrt2,\,\operatorname{d}(w_{2},-i)> \sqrt2 \},\\
		F=  &  \{(1,1),(1,-i),(1,1-i),(1,1+i),(1+i,1)\}.\\
		S=& [-\tfrac{1}{2},\tfrac{1}{2}[+[-\tfrac{1}{2},\tfrac{1}{2}[i
	\end{align*}
	\item For $\ttt\in\C$,
	\[
	M_{\ttt}=\begin{pmatrix}
		1&-\ttt \\
		0&1
	\end{pmatrix},
	\hspace{1cm} \LL_{\ttt}=M_{\ttt}\ZZ^2.
	\]
	\item When  $A$ is a commutative ring with unit $1_A$, $\SL(2,A)$ is the set of $2\times 2$ matrices with entries in $A$ and determinant $1_A$.
	
	\item $\lambda_1(\LL,\|.\|,\C)$ and $\lambda_2(\LL,\|.\|,\C)$ are the two complex minima of a Gauss lattice $\LL$ in $\C^2$ associated with the norm $\|.\|$, see definition \ref{def:minima}.
	\item The space of unimodular lattices in $\C^2$
	\[
	\Omega_1=\SL(2,\C)/\SL(2,\ZZ).
	\]
	\item The transversal $T$ is defined in subsection \ref{subsec:opentrans} and $T'$, $T_1$, $T_2$ are defined in subsection \ref{subsec:fulltrans}.
	\item The negligible set $\mathcal N$ is defined in subsection \ref{subsec:fulltrans}.
	\item The parametrizations $\Psi_k(\ttt,w_1,w_2)$ are defined in Proposition \ref{prop:para} in subsection \ref{subsec: param}.
	\item The sets $W_1$ and $W_2$ are defined in subsection \ref{subsec:detopentrans}.
	\item The sets $W'_1$ and $W'_2$, the map $T_G$ and the coefficients $a_k(w_1,w_2)$ are defined in subsection \ref{subsec:firstreturn}.
\end{itemize}

\subsection{The set of unimodular Gauss lattices in $\mathbb{C}^{2}$}

\begin{definition}
	Let $E$ be a finite dimensional $\mathbb{C}$-vector space. A subset
	$\Lambda$ in $E$ is a Gauss lattice if it is a $\mathbb{Z}[i]$-submodule of $E$, if it is a discrete subset of
	$E$ and if it generates the vector space $E$.
\end{definition}

Let $\Omega_{1}$ be the set of Gauss lattices $\Lambda$ in $\mathbb{C}^{2}$
that admits a basis $(u,v)$ with determinant in $\mathbb{U}_{4}=\{\pm1,\pm
i\}$. By definition, $\Lambda=M\mathbb{Z}[i]^{2}$ where $M$ is the matrix with
columns $u$ and $v$. Changing $u$ to $\pm u$ or to $\pm iu$, we can assume
that $M\in\operatorname{SL}(2,\mathbb{C})$. Next proposition is clear.

\begin{proposition}
	The map
	\[
	M\operatorname{SL}(2,\mathbb{Z}[i])\in\operatorname{SL}(2,\mathbb{C}%
	)/\operatorname{SL}(2,\mathbb{Z}[i])\rightarrow M\mathbb{Z}[i]^{2}\in
	\Omega_{1}
	\]
	is well defined and is bijective.
\end{proposition}

Thanks to the proposition, we can identify $\Omega_{1}$ and $\operatorname{SL}%
(2,\mathbb{C})/\operatorname{SL}(2,\mathbb{Z}[i])$ and use results from
ergodic theory. For $t\in\mathbb{R}$, consider the matrices
\[
g_{t}=%
\begin{pmatrix}
	e^{t} & 0\\
	0 & e^{-t}%
\end{pmatrix}
.
\]
The flow $(g_{t})_{t\in\R}$ acts on $\Omega_{1}$ by left multiplication :
\[
g_t\LL=\{g_tx:x\in\LL\}=g_tM \ZZ^2\cong g_tM\SL(2,\ZZ).
\]

\subsection{Minimal vectors}
The notion of minimal vector goes back to Voronoï, see \cite{Vor}. He used minimal vectors to find units in cubic fields. The Voronoï's algorithm has been generalized by Buchmann to find units in some quartic and quintic fields, see \cite{Bu1,Bu2}
\begin{definition}
	Let $\Lambda$ be a Gauss lattice in $\mathbb{C}^{2}$.
	
	\begin{itemize}
		\item A nonzero vector $u=(u_{1},u_{2})\in\Lambda$ is a \textit{minimal}
		vector in $\Lambda$ if for every nonzero $v\in\Lambda$, $v\in C(u)=\{(z_1,z_2):|z_1|\leq|u_1|,|z_2|\leq |u_2|\}\Rightarrow
		|v_{1}|=|u_{1}|$ and $|v_{2}|=|u_{2}|$.
		
		\item Two minimal vectors $u=(u_{1},u_{2})$ and $v=(v_{1},v_{2})$ are
		\textit{equivalent} if $C(u)=C(v)$.
		
		\item Two minimal vectors $u=(u_{1},u_{2})$ and $v=(v_{1},v_{2})$  are 
		\textit{consecutive} iff $|u_2|<|v_2|$ and there is no minimal vector $w=(w_1,w_2)$ such $|u_{2}|<|w_2|<|v_{2}|$.
	\end{itemize}
\end{definition}

\begin{remark}
	Following Buchmann (\cite{Bu1,Bu2}), we could have define the minimal vectors using the 	preoder 
	$u \ll v$ iff $|u_1|\leq |v_1|$ and $|u_2|\leq|v_2|$ for $u,v$ be in $\C^2$. With this preorder, the minimal vectors of a Gauss lattice $\LL$ in $\C^2$ are the minimal elements in $(\LL\setminus\{0\} ,\ll)$. Observe  that the lexicographic order $\prec$ is also used by Buchmann in the same papers. 
\end{remark}

\begin{remark}
	If $u=(u_1,u_2)$ and $v=(v_1,v_2)$ are two minimal vectors in a lattice $\LL\subset\C^2$ and if $|v_2|>|u_2|$ then by definition, $|u_1|>|v_1|$. Therefore, there exist complex numbers $w_1$ and $w_2$ unique such that $u=(u_1,v_2w_2)$ and $v=(u_1w_1,v_2)$. Moreover $|w_1|,|w_2|<1$.
\end{remark}	

We collect a few easy lemmas about minimal vectors.

\begin{lemma}\label{lem:consecutive}
	Two minimal vectors $u=(u_{1},u_{2})$ and $v=(v_{1},v_{2})$ in a Gauss lattice  $\LL\subset\C^2$ are
	consecutive iff $|u_{2}|<|v_{2}|$ and the only lattice point in the
	interior of $C(u,v)$ is zero.
\end{lemma}

\begin{proof}
	Let  $u=(u_{1},u_{2})$ and $v=(v_{1},v_{2})$ be two minimal vectors with $|u_2|<|v_2|$. If the set $\overset{o}{C}(u,v)\cap \LL\setminus\{0\} $ is nonempty, then it is finite and there is a $w=(w_1,w_2)$ minimal in this set for the lexicographic preorder $\prec$. On the one hand, $w$ is minimal in $\LL$. On the other hand, $|w_1|<|u_1|$ and $|w_2|<|v_2|$ and since $u$ is a minimal vector we have $|w_2|>|u_2|$. Hence $u$ and $v$ are not consecutive.
	
	Conversely, if $u$ and $v$ are not consecutive there is a minimal vector $w$ with $|u_2|<|w_2|<|v_2|$. Since $w$ is minimal $|u_1|>|w_1|$, hence $w\in  \overset{o}{C}(u,v)\cap \LL$. 
\end{proof}

Next lemma is clear.

\begin{lemma}
	Let $\Lambda$ be a Gauss Lattice in $\mathbb{C}^{2}$ and let $u$ be a minimal
	vector in $\Lambda$.
	
	\begin{itemize}
		\item All minimal vectors $v\in\Lambda$ such that $u$ and $v$ are consecutive,
		are  equivalent.
		
		\item If $u^{\prime}$ and $v$ are minimal vectors such that $u$ is equivalent
		to $u^{\prime}$, and $u$ and $v$ are consecutive, then $u^{\prime}$ and $v$
		are consecutive.
	\end{itemize}
\end{lemma}

Next lemma is useful to construct minimal vector in lattice.
\begin{lemma}\label{lem:lexico}
	Let $\LL$ be a Gauss lattice in $\C^2$ and let $r$ be a positive real number. Let $C$ be the infinite cylinder $C_1(r)=\{(z_1,z_2):|z_1|\leq r\}$ or its interior.
	\begin{itemize} 
		\item The set $C\cap\LL\setminus\{0\}$ is nonempty and admits a minimal element for the lexicographic order.
		\item If $u$ is a minimal element for the lexicographic order in the set $C\cap\LL\setminus\{0\} $ then $u$ is minimal in $\LL$.
	\end{itemize}
\end{lemma}
\begin{proof}
	Since $r>0$, by Minkowski convex body theorem $C\cap\LL\setminus\{0\} $ is nonempty. Let $C_2(\rho)=\{(z_1,z_2):|z_2|\leq \rho \}$. If $v=(v_1,v_2)$ is in $C\cap\LL\setminus\{0\} $, then $C\cap\LL\setminus\{0\} \cap C_2(|v_2|)$ is finite and nonempty and so $C\cap\LL\setminus\{0\} \cap C_2(|v_2|)$ must contain a minimal element $u$ for the lexicographic preorder. 	This element $u$ is also minimal in $C\cap\LL\setminus\{0\} $ for the lexicographic preorder.
	
	If $w=(w_1,w_2)\in C(u)\cap \LL\setminus\{0\} $ then $w\prec u$ and $w\in C$. Since $u$ is minimal for the lexicographic order we also have $u\prec w$ which implies $|u_2|=|w_2|$ and  $|w_1|=|v_1|$, hence $u$ is minimal in $\LL$
\end{proof}

\subsection{The sequence of minimal vectors}
Given a Gauss lattice $\Lambda$ in $\mathbb{C}^{2}$, the set of minimal vectors
can be arranged in a sequence $(X_{n}(\Lambda))_{n\in I_{\LL}}=(z_{1,n},z_{2,n})_{n\in
	I_{\LL}}$ where $I_{\LL}$ is an interval in $\mathbb{Z}$ such that the sequence
$(|z_{2,n}|)_{n\in I_{\LL}}$ is increasing and each minimal vector is equivalent to a
minimal vector of the sequence. This sequence might be finite, infinite one
sided or two sided. Two minimal vectors are consecutive if and only if they are
equivalent to two consecutive terms of the sequence $(X_{n}(\Lambda))_{n\in
	I_{\LL}}$. For all $n\in I_{\LL}$, let denote
$r_{n}(\Lambda)=|z_{1,n}|$ and $q_{n}(\Lambda)=|z_{2,n}|$. 
The three following results are standard in the frame work of best Diophantine approximations and continued fractions.  The second inequality of the first item gives an upper bound of the Dirichlet complex constant. The lemma will not be used in the sequel.

\begin{lemma}\label{lem:pigeon}
	Let $\Lambda$ be a lattice in $\mathbb{C}^{2}$ and let $(X_{n}(\Lambda))_{n\in
		I_{\LL}}$ be the sequence of minimal vectors of $\Lambda$.
	
	\begin{enumerate}
		\item If $n$ and $n+1\in I_{\LL}$, then $\tfrac{1}{2}|\det_{\mathbb{C}}%
		(\Lambda)|\leq q_{n+1}(\Lambda)r_{n}(\Lambda)\leq\tfrac{4}{\pi}|\det
		_{\mathbb{C}}(\Lambda)|$.
		
		\item If $n$ and $n+14\in I_{\LL}$, then $q_{n+14}(\Lambda)\geq Cq_{n}(\Lambda)$
		where $C=\tfrac12(1+\cos(\tfrac{2\pi}{7}))>1.1234$
		
		\item If $n$ and $n+56\in I_{\LL}$, then $r_{n+70}(\Lambda)\leq\tfrac12
		r_{n}(\Lambda)$.
	\end{enumerate}
\end{lemma}

\begin{proof}
	1. Making use of Minkowski convex body Theorem with the cylinder $C(X_{n}
	(\Lambda),X_{n+1}(\Lambda))$ and the lattice $\Lambda$, we obtain that $(\pi
	q_{n+1}(\Lambda)r_{n}(\Lambda))^{2}\leq16|\det_{\mathbb{R}}(\Lambda)|$, thus
	$q_{n+1}(\Lambda)r_{n}(\Lambda)\leq\tfrac{4}{\pi}|\det_{\mathbb{C}}(\Lambda
	)|$. Since the minimal vectors $X_{n}(\Lambda)=(z_{1,n},z_{2,n})$ and
	$X_{n+1}(\Lambda)=(z_{1,n+1},z_{2,n+1})$ are linearly independent, 
	$|\det_{\mathbb{R}}(X_{n}(\Lambda),X_{n+1}(\Lambda))|$ is a positive integer
	multiple of $|\det_{\mathbb{R}}(\Lambda)|$. It follows that $|\operatorname{det}_{\mathbb{C}}(X_{n}(\Lambda),X_{n+1}(\Lambda
	))|\geq|\operatorname{det}_{\mathbb{C}}(\Lambda)|$ and then
	\[
	2q_{n+1}(\Lambda)r_{n}(\Lambda)\geq|z_{1n}z_{2,n+1}|+|z_{2n}z_{1,n+1}%
	|\geq|\operatorname{det}_{\mathbb{C}}(X_{n}(\Lambda),X_{n+1}(\Lambda
	))|\geq|\operatorname{det}_{\mathbb{C}}(\Lambda)|.
	\]
	2. This is a standard application of the pigeonhole principle. Given
	$r>0$ and $C'<C$, $7$ closed disks of radius  $\tfrac12 r$ are enough to
	cover a disk of radius $r$ and $8$ open disks of radius $\tfrac12 r$ are
	enough to cover a closed disk of radius $C'r$. The first covering
	result is very well known and easy, the second is due to G. Fejes Toth,
	\cite{Th}. It follows that $7\times8=56$ translates of the semi-open box
	$B_{1}=D(0,\tfrac12 r_{n}(\Lambda))\times\overset{o}D(0,\tfrac12 q_{n}(\Lambda))$
	can cover the box $B_{2}=C(r_{n}(\Lambda),C^{\prime}q_{n}(\Lambda))$ for any
	$C^{\prime}<C$. Now if $q_{n+14}(\Lambda)<Cq_{n}(\Lambda)$ then all the
	$4\times15=60$ points of the set $\mathbb{U}_{4}\{X_{n}(\Lambda),\dots
	,X_{n+14}(\Lambda)\}$ are in the box $B_{2}=C(r_{n}(\Lambda),C^{\prime}%
	q_{n}(\Lambda))$ with $C^{\prime}=\tfrac{q_{n+14}(\Lambda)}{q_{n}(\Lambda)}$,
	so at least two of them are in the same translate of the box $B_{1}$. It
	follows that their difference is in the box $2B_{1}$ which contradicts that
	$X_{n}(\Lambda)$ is a minimal vector. \newline3. We use twice the pigeonhole
	principle. We can split $\mathbb{C}$ in height angular sector $C_{1}%
	,\dots,C_{8}$ such that if $z$ and $z^{\prime}$ are in the same angular sector
	then $|z-z^{\prime}|\leq\max(|z|,|z^{\prime}|)$. Consider the $57$ minimal vectors $X_{n}%
	(\Lambda)=(z_{1n},z_{2n}),\dots,X_{n+56}(\Lambda)=(z_{1,n+56},z_{2,n+56})$. There is a sector $C_{i}$ that contains at least
	 seven of the $z_{1j}$, say for the $j\in J$. Since $r_{j}(\Lambda)\leq
	r=r_{n}(\Lambda)$ for $j\in J$ and $\operatorname{card} J\geq7$, there exists
	$k\neq j$ in $J$ such that $|z_{1k}-z_{1j}|\leq\tfrac12 r$. Therefore, the
	vector $X=X_{k}(\Lambda)-X_{j}(\Lambda)=(x_{1},x_{2})$ is such that
	$|x_{1}|\leq\tfrac12 r$ and $|x_{2}|=|z_{2k}-z_{2j}|\leq2\max(|z_{2k}%
	|,|z_{2j}|)\leq 2q_{n+56}(\Lambda)$. The cylinder $C(X)$ contains a minimal vector
	$X_{i}$ which is one of $X_{n}(\Lambda),\dots,X_{n+56+14}(\Lambda)$ so we are done.
\end{proof}

\subsection{Minimal vectors and Diophantine approximations}
\begin{definition}
	Let $\ttt$ be a complex number. A pair $(p,q)\in\Z[i]$ is a best approximation vector of $\ttt$ if $|q|>0$ and for all $(a,b)\in\Z[i]^2$,
	\begin{align*}
		\left\{
		\begin{array}[c]{l}
			0<|b|<|q|\Rightarrow |p-q\ttt|<|a-b\ttt|\\
			0<|b|\leq|q|\Rightarrow |p-q\ttt|\leq|a-b\ttt|
		\end{array}
		\right..
	\end{align*}
\end{definition}

\begin{proposition}\label{prop:best}
	Let $\ttt$ be a complex number and consider  the lattice $\LL_{\ttt}$ defined by
	\[
	\LL_{\ttt}=\begin{pmatrix}
		1&-\ttt \\
		0&1
	\end{pmatrix}\Z[i]^2=M_{\ttt}\Z[i]^2.
	\]
	Then $X=\begin{pmatrix}
		x\\y
	\end{pmatrix}=M_{\ttt}\begin{pmatrix}
		p\\q
	\end{pmatrix}\in\LL_{\ttt}
	$ is a minimal vector with $y\neq 0$ iff $(p,q)$ is a best Diophantine approximation vector of $\ttt$.
\end{proposition}
In the multidimensional real  setting, Lagarias proved  that a shortest vector of the lattice $g_t\LL_{\ttt}$ is associated with a best Diophantine  approximation of $\ttt$, see \cite{Lag1}. 
His result was stated for the Euclidean norm instead of the sup norm. That is why some best approximations are not associated with a shortest vector even in one-dimensional case.
\begin{proof}
	Suppose that $X=\begin{pmatrix}
		x\\y
	\end{pmatrix}$ is a minimal vector with $y\neq 0$. If $a$ and $b$ are Gaussian integers with $0<|b|<|y=q|$, then $Y=\begin{pmatrix}
		a-b\ttt \\
		b
	\end{pmatrix}\notin C(X)$ which implies $|a-b\ttt|>|p-q\ttt|$. If $|b|=|q|$ and if $Y\in C(X)$ then $|a-b\theta|=|p-q\theta|$.
	
	Conversely, if $(p,q)$ is a best Diophantine approximation vector of $\ttt$, then for any $(a,b)\in\Z[i]^2$, $Y=\begin{pmatrix}
		a-b\ttt \\
		b
	\end{pmatrix}\in C(X)$ implies
	\[\left\{\begin{array}{ll}
		|a-b\ttt|\leq |p-q\ttt|\\ |b|\leq|q|
	\end{array}\right..
	\]
	If $b\neq 0$ this in turn implies $|a-b\ttt|=|p-q\ttt|$ and $|b|=|q|$ by definition of best approximation vectors. If $b=0$ and $a\neq 0$ then $|a|\geq 1>\tfrac{\sqrt 2}{2}\geq |p-q\ttt|$, hence $Y\notin C(X)$.
\end{proof}

\section{Proof of Theorem \ref{thm:index}, index of lattices spanned by two consecutive minimal vectors}\label{sec:conscecutive}

Let $I$ be the ideal in $\mathbb{Z}[i]$ generated by $1+i$, i.e.
$I=(1+i)\mathbb{Z}[i]$ and let $J=\frac1{1+i}(\mathbb{Z}[i]\setminus I)$.

Theorem \ref{thm:index} is a consequence of the following proposition.

\begin{proposition}
	\label{prop:index} Let $\Lambda$ be a Gauss lattice in $\mathbb{C}^{2}$.
	Suppose that $u=(u_{1},u_{2})$ and $v=(v_{1},v_{2})$ are two linearly independent
	minimal  vectors in $\Lambda$ and such that $\overset{o}C(u,v)\cap \LL=\{0\}$. Call $L$ the lattice spanned by $u$ and $v$. Then
	
	\begin{enumerate}
		\item $\tfrac14 \det_{\mathbb{R}}(\Lambda)\leq|u_{1}|^{2}|v_{2}|^{2}\leq
		\tfrac{16}{\pi^{2}}\det_{\mathbb{R}}(\Lambda)$,
		
		\item $L$ has index $1$ or $2$: $[\Lambda:L]= \frac{|\det_{\mathbb{R}}%
			(L)|}{|\det_{\mathbb{R}}(\Lambda)|}=1$ or $2$.
		
		\item If $L$ has index $2$, then
		\[
		\Lambda=\{au+bv:(a,b)\in\mathbb{Z}[i]^{2}\cup J^{2}\}
		\]
		and $(U=u,V=\frac{1}{1+i}(u+v))$ and $ (U'=\frac{1}{1+i}(u+v),V'=v)$ are two bases of $\Lambda$.
	\end{enumerate}
	When $u$ and $v$ are two consecutive minimal vectors, we shall say that $[L:\LL]$ is the \sl{ index } of the two consecutive minimal vectors $u$ and $v$.
\end{proposition}

\begin{proof}
	Since $u$ and $v$ are  minimal vectors, we can suppose $|u_{2}|\leq|v_{2}|$ and
	$|v_{1}|\leq|u_{1}|$ w.l.o.g..
	By Minkowski convex body theorem,
	\[
	\operatorname{Vol}(C(u,v))=\pi^{2}|u_{1}|^{2}|v_{2}|^{2}\leq2^{4}%
	|\operatorname{det} _{\mathbb{R}}(\Lambda)|=2^{4}|\operatorname{det}%
	_{\mathbb{C}}(\Lambda)|^{2}.
	\]
	Now $|\det_{\mathbb{C}}(\Lambda)|\leq|\det
	_{\mathbb{C}}(L)|\leq2|u_{1}||v_{2}|$, hence
	\[
	|\operatorname{det}_{\mathbb{R}}(L)|\leq4 |u_{1}|^{2}|v_{2}|^{2}%
	=4\frac{\operatorname{Vol}(C(u,v))}{\pi^{2}} \leq\frac{64|\det_{\mathbb{R}%
		}(\Lambda)|}{\pi^{2}}.
	\]
	Therefore,
	\[
	\frac{|\det_{\mathbb{R}}(L)|}{|\det_{\mathbb{R}}(\Lambda)|}\leq\frac{64}%
	{\pi^{2}}=6.48\dots
	\]
	Therefore, $[\Lambda:L]\leq6$. Since this index is the square of the modulus of
	a Gaussian integer, it is the sum of two squares and cannot be $3$ or $6$.
	
	By Theorem \ref{thm:base} about basis in $\mathbb{Z}[i]$-modules, there
	exist a basis $U,V$ of $\Lambda$ and Gaussian integers $a,b$ and $c$ such that
	\begin{align*}
		\left\{
		\begin{array}
			[c]{l}%
			u=aU\\
			v=bU+cV.
		\end{array}
		\right.
	\end{align*}
	We have $V=-\tfrac{b}{c}U+\tfrac{1}{c}v$. Since $u$  is primitive in
	$\Lambda$, $a$ must be a unit in $\ZZ$. By
	changing $U$ to $a^{-1}U$, we can suppose $a=1$. 
	
	Since $[\Lambda:L]=|c|^{2}$, the only possible values for $|c|^{2}$ are
	$1,2,4$ or $5$. We have to exclude the values $4$ and $5$.
	
	Suppose that $|c|=2$. Again by changing $V$ to $zV$ where $z$ is a unit, we can
	suppose $c=2$ w.l.o.g. There exists a Gaussian integer $g$ such that $|g-\tfrac
	{b}{c}|\leq\tfrac1{\sqrt2}$. Since $|cg-b|\leq\sqrt2$, $|cg-b|=0,1$ or
	$\sqrt2$. If $cg-b=0$ then $V+gU=\tfrac{1}{c}v\in\Lambda$, but this is not
	possible for $v$ is primitive. If $|cg-b|=1$, consider the vector $w=V+gU=\tfrac{cg-b}%
	{c}u+\tfrac{1}{c}v\in\Lambda$. Since $u$ and $v$ are minimal, either $|u_1|>|v_1|$ and $|v_2|>|u_2|$ or $|u_1|=|v_1|$ and $|v_2|=|u_2|$. In the first case, by convexity, $w$ would be in the interior of
	the cylinder $C(u,v)$ which is not possible by assumption. In the second case, the linear independence implies $(cg-b)u\neq v$, so that one of the coordinates of $(cg-b)u$ and of $v$ are note equal, and therefore the corresponding coordinate of $w$ would be strictly smaller which contradicts the minimality of $u$ and $v$. If $|cg-b|=\sqrt2$, then
	the inverse $z$ of $\tfrac{cg-b}{c}$ is a Gaussian integer and the vector
	$w^{\prime}=zw-u=\tfrac{z}{c}v$ is in $\Lambda$. But this is impossible for
	$|\tfrac{z}{c}|<1$ and $v$ is primitive.
	
	Suppose that $|c|= \sqrt5$. There is $8$ possible values for $c$. By changing $V$ to $zV$
	where $z$ is a unit, or by considering the image of $\Lambda$ by the map
	$(z_{1},z_{2})\rightarrow(\overline{z_{1}},\overline{z_{2}})$, we can suppose
	that $c=2-i$. We can also suppose that $|b|\leq\tfrac1{\sqrt2}|c|$ by changing
	$V$ to $V+gU$ where $g$ is a Gaussian integer such that $|\tfrac{b}{c}%
	-g|\leq\tfrac1{\sqrt2}$. So $|b|\leq\tfrac{\sqrt 5}{\sqrt 2}$. Now $|b|^{2}$ is an
	integer, hence $|b|^{2}=0,1$ or $2$. The case $b=0$ is not possible for $v$ is
	minimal. If $|b|=1$, then $|\tfrac{b}{c}|+|\tfrac{1}{c}|<1$ and
	$V=-\tfrac{b}{c}u+\tfrac{1}{c}v$ would be in the interior of $C(u,v)$.
	
	It remains to consider the cases $b=1+i,1-i,-1-i$ and $-1+i$. Since $b=z(1+i)$ with $z\in\UU_{4}$, the vector
	\[
	w=V+ziu=-z(1+i)\frac{2+i}{5}u+\frac{2+i}{5}v+ziu=z\frac{-1+2i}{5}u+\frac{2+i}{5}v
	\]
	is in $\Lambda$ and in the interior of $C(u,v)$ for the sum of the moduli of
	the coefficients of $u$ and $v$ is $<1$. So $|c|=\sqrt 5$ is not possible and we conclude that $|c|=1$ or $\sqrt2$.
	
	If $|c|=1$, $L=\Lambda$.
	
	Suppose that $|c|=\sqrt2$. We have
	\[
	\left\{
	\begin{array}
		[c]{l}%
		u=U\\
		v=bU+cV
	\end{array}
	\right.
	\]
	and by changing $V$ to $zV$ for some $z\in\mathbb{U}_{4}$, we can suppose that
	$c=1+i$. There is a Gaussian integer $g$ such that $b=g(1+i)$ or $g(1+i)+1$. Changing $V$ to $V+gU$, we can suppose that $b=0$ or $1$. Again $b\neq 0$ since $v$ is primitive, hence $b=1$. Solving in $U,V$, we obtain
	\[
	\left\{
	\begin{array}
		[c]{l}%
		U=u\\
		V=\frac{1}{c}(-u+v)
	\end{array}
	\right.
	\]
	and for all $g,h\in\mathbb{Z}[i]$
	\[
	gU+hV=\frac{cg-h}{c}u+\frac{h}{c}v.
	\]
	On the
	other hand, $c\in I$, hence either $cg-h$ and $h$ are both in $I$ or $cg-h$ and $h$ are both in
	$\mathbb{Z}[i]\setminus I$, which implies that
	\[
	\Lambda=\{gU+hV:(g,h)\in\mathbb{Z}[i]^{2}\}\subset\{g'u+h'v:(g',h')\in\mathbb{Z}%
	[i]^{2}\cup J^{2}\}.
	\]
	The reverse inclusion also holds because if $(g',h')=\tfrac{1}{c}(p,q)$ with $p,q\in\ZZ\setminus I$, then $g'u+h'v=\tfrac{1}{c}(p+q)U+qV\in\LL$.
\end{proof}

\section{Geometry of numbers, proof of Theorem \ref{thm:geonumber1}}\label{sec:geonumber}
Our aim is to prove Theorem \ref{thm:geonumber1}. In fact we shall prove the following two theorems, the first  is just a reformulation of Theorem \ref{thm:geonumber1} using the norm $|.|_{u,v}$ instead of the cylinder $C(u,v)$. The norm is defined by  $|x|_{u,v}=\max(\tfrac{|x_1|}{\max(|u_{1}|,|v_1|)},\tfrac{|x_2|}{\max(|u_2|,|v_{2}|)})$ for  $x=(x_1,x_2)\in \C^2$.

\begin{theorem}[Theorem 2a]\label{thm:geonumber2}
	Let $u=(u_1,v_2w_2)$ and $v=(u_1w_1,v_2)$ be two  vectors in $\C^2$ with $|u_1|,|v_2|>0$ and  $|w_1|,|w_2|\leq 1$. 
	\begin{enumerate}
		\item 
		
		If $|gu+hv|_{u,v}\geq 1$ for all nonzero $g$, $h \in\ZZ$ with $|g|\times|h|\leq \sqrt 2$, then $|z|_{u,v}\geq 1$ for all nonzero $z\in\ZZ u+\ZZ v$. 
		\item 
		If  $|gu+hv|_{u,v}\geq 1$ for $(g,h)\in J^2$ with $|g|=|h|=\tfrac1{\sqrt 2}$, then $|z|_{u,v}\geq 1$ for all nonzero $z\in\langle u,v\rangle_J$.
	\end{enumerate}
\end{theorem}

The next theorem deals with strict inequality and is useful to determine the open transversal. 

\begin{theorem}[Theorem 2b]\label{thm:geonumber3}
	Let $u=(u_1,v_2w_2)$ and $v=(u_1w_1,v_2)$ be two  vectors in $\C^2$ with  $|u_1|,|v_2|>0$ and $|w_1|,|w_2|<1$. 
	\begin{enumerate}
		\item 
		
		If $|gu+hv|_{u,v}>1$ for all nonzero $g,h \in\ZZ$ with $|g|\times|h|\leq \sqrt 2$ then $|z|_{u,v}> 1$ for all nonzero $z\in(\ZZ u+\ZZ v)\setminus  \UU_4 u\cup \UU_4 v$. 
		\item 
		If $|gu+hv|_{u,v}> 1$ for the four vectors $(g,h)=(\tfrac1{1+i},\tfrac{\alpha}{1+ i})$,  $\alpha\in\UU_{4}$, then $|z|_{u,v}> 1$ for all nonzero $z\in\langle u,v\rangle_J \setminus  \UU_4 u\cup \UU_4 v$.
	\end{enumerate}
\end{theorem}
The proof of these two theorems are very similar and based on many case distinctions. The first case distinction is made on the location of $ w_1 $ in the unit disk. 

Let $\mathcal{C}= \{z\in\mathbb{C}: |z|<1,\,\arg z\in[0,\tfrac{\pi}4 ]\}$. The first case distinction is 
$w_1\in\bar \CC$ (the closure of $\CC$) or in $i\bar \CC$ or in $-\bar \CC$ or in $-i\bar \CC$ or in the conjugates of one of  these sets. Thanks to the following subsection, Symmetric of a lattice, these eight cases reduce to the single case $w_1\in\bar \CC$.

The same reduction will also be helpful for computing the Dirichlet constant in the Theorem \ref{thm:Dirichlet}. 
\subsection{Symmetric of a lattice, reduction to the case $w_1\in\bar \CC$ }
Let denote $\mathbb{U}_{n}=\{z\in\mathbb{C}%
:z^{n}=1\}$ the group $n$-th roots of unity in $\mathbb{C}$ and let denote $\mathbb{D}%
_{8}$ the group of isometries acting on $\mathbb{C}$ generated by the
multiplications by elements in $\mathbb{U}_{4}$ and by  conjugation.

\begin{proposition}\label{prop:sym}
	Let $u=(u_{1},v_{2}w_{2})$ and $v=(u_{1}w_{1},v_{2})$ be
	in $\mathbb{C}^{2}$. Assume that $|w_{1}|,|w_{2}|\leq 1$ and $|u_{1}|,|v_{2}|>0$. Let $\varphi$ be in $\D_8$. Consider $u'=(u'_1,v'_2\tfrac{1}{\varphi(1)^2}
	\varphi(w_2))$ and $v'=(u'_1\varphi(w_1),v'_2)$ where  $|u'_{1}|,|v'_{2}|>0$. Then
	\begin{enumerate}
		\item 	
		For all nonzero complex numbers $a$ and $b$,
		\[
		|au-bv|_{u,v}=|\varphi(1)\varphi(a)u'-\varphi(b)v'|_{u',v'}.
		\]
		\item When $|w_1|,|w_2|<1$, the vectors $u$ and $v$ are consecutive minimal vectors in $\ZZ u+\ZZ v$ (resp. in $\langle u,v\rangle_J$) iff  $u'$ and $v'$  are consecutive minimal vectors in $\ZZ u'+\ZZ v'$ (resp. in $\langle u',v'\rangle_J$)
	\end{enumerate}
\end{proposition}
Let us  explain how the first assertion in the proposition allows us to reduces the proofs of Theorems \ref{thm:geonumber2} and \ref{thm:geonumber3} to the case  $w_1\in\overline{\CC}$. When $\varphi\in \D_8$,
the three maps $\varphi$, $\psi:z\in\C \fff \psi(z)=\varphi(1)\varphi(z)$
and $\varphi':z\in\C \fff \varphi'(z)=\tfrac{\varphi(z)}{\varphi(1)^2}$ are isometries and bijection on the ring of Gaussian integers and on $J$. Using part 1 of the proposition, we see that, if for some vectors $u,v$ and a subset  $F$ of $\mathcal R$ where $\mathcal R= (\ZZ\setminus\{0\})^2$ or $(\ZZ\setminus\{0\})^2\cup J^2$, one has
\[
\forall (g,h)\in F, |gu-hv|_{u,v}\geq 1\Rightarrow \forall (g,h)\in \mathcal R^2 \text{ with } fg\neq 0,|gu-hv|_{u,v}\geq 1,
\]
then one has the same implication with $u',v'$ and 
\[
F'=\{(\psi(g),\varphi(h)):(g,h)\in F\}
\]
instead of $F$. Since the images of $\overline{\CC}$ by the maps $\varphi \in\D_8$ cover the closed unit disk, we have only to deal with $w_1\in\overline{\CC}$.

Before proving the proposition, we need a simple formula.

\begin{lemma}
	\label{lem:sym2} For all  $\varphi\in\mathbb{D}_{8}$ and
	 all $x,y\in\C$
	\[
	\varphi(xy)=\frac{1}{\varphi(1)}\varphi(x)\varphi(y)
	\]	
\end{lemma}

\begin{proof}
	The formula is obvious since the maps $\varphi\in\D_8$ are of the shape $\varphi(z)=\alpha z$ or $\alpha \bar z$ with $\alpha\in\UU_{4}$.
\end{proof} 

\begin{proof}[Proof of the proposition]
	1. For all $a,b\in\C$ and all $\varphi\in\D_8$, we have
	\begin{align*}
		au-bv&=(u_1(a-bw_1),v_2(aw_2-b))
	\end{align*}
	and using the above lemma, we obtain
	\begin{align*}	\varphi(1)\varphi(a)u'-\varphi(b)v'&=(u'_1(\varphi(1)\varphi(a)-\varphi(b)\varphi(w_1)),v'_2(\varphi(1)\varphi(a)\frac{1}{\varphi(1)^2}\varphi(w_2)-\varphi(b)))\\
		&=(u'_1(\varphi(1)\varphi(a)-\varphi(1)\varphi(bw_1)),v'_2(\varphi(aw_2)-\varphi(b)))\\
		&=(u'_1\varphi(1)\varphi(a-bw_1),v'_2\varphi(aw_2-b)).	
	\end{align*}
	Therefore,
	\begin{align*}
		|au-bv|_{u,v}      =|\psi(a)u^{\prime}-\varphi(b)v^{\prime}|_{u',v'}.
	\end{align*}
	
	2. The vector $u$ is minimal iff for all nonzero $au-bv\in\ZZ u+\ZZ v$ (resp. $\in\langle u,v\rangle_J$)
	\[
	\left\{
	\begin{array}[c]{l}
		|bw_{1}-a|\leq 1 \\
		|aw_{2}-b|\leq |w_2|
	\end{array}
	\right.
	\Rightarrow 
	\left\{
	\begin{array}[c]{l}
		|bw_{1}-a|= 1 \\|aw_{2}-b|= |w_2|
	\end{array}
	\right.
	\]
	and $u'$ is minimal iff for all nonzero $\varphi(1)\varphi(a)u'-\varphi(b)v'\in \ZZ u'+\ZZ v'$ (resp. $\in\langle u',v'\rangle_J$)
	\[
	\left\{
	\begin{array}[c]{l}
		|\varphi(1)\varphi(bw_{1}-a)|\leq 1 \\|\varphi(aw_{2}-b)|\leq |\tfrac{1}{\varphi(1)^2}
		\varphi(w_2)|
	\end{array}\right.
	\Rightarrow
	\left\{
	\begin{array}[c]{l}
		|\varphi(1)\varphi(bw_{1}-a)|= 1 \\|\varphi(aw_{2}-b)|= |\tfrac{1}{\varphi(1)^2}
		\varphi(w_2)|
	\end{array}\right.
	\]
	Therefore $u$ is minimal iff $u'$ is minimal.
	We see that $v$ is minimal iff $v'$ is minimal as well. Furthermore, by Lemma \ref{lem:consecutive}, $u$
	and $v$ are consecutive iff  $|au-bv|_{u,v}\geq 1$ for all nonzero $au-bv\in\ZZ u+\ZZ v$ (resp. $\in\langle u,v\rangle_J$). The formula $|au-bv|_{u,v}      =|\psi(a)u^{\prime}-\varphi(b)v^{\prime}|_{u',v'}$ implies that $u$ and $v$ are consecutive iff $u'$ and $v'$ are.
\end{proof}

\subsection{Proof of Theorem \ref{thm:geonumber2} and \ref{thm:geonumber3} when $w_1\in \overline{\CC}$}
We shall need the following sets
\begin{align*}
	\DD=  &  \{z\in\C:|z|<1,\, \dd(z,1)>
	1,\,\dd(z,1-i)> 1 \},\\
	\mathcal{T}=  &  \{z\in\C:|z|<1,\, \dd(z%
	,1)>\sqrt2,\,\dd(z,-i)> \sqrt2 \},\\
	F=  &  \{(1,1),(1,-i),(1,1-i),(1,1+i),(1+i,1)\}.\\
\end{align*}

Theorem \ref{thm:geonumber2} and \ref{thm:geonumber3} are  obvious consequences of the following proposition where we assume  $w_{1}\in\overline{\mathcal{C}}$. 

\begin{proposition}\label{prop:geonumber}
	\label{prop:conditions} Let $u=(u_{1},v_{2}w_{2})$ and $v=(u_{1}w_{1},v_{2})$ in $\mathbb{C}^{2}$ be such that $w_{1}\in\overline{\mathcal{C}}$  and $|w_{1}|,|w_{2}|\leq 1$ (resp. $<1$) and $|u_{1}|$, $|v_{2}|>0$.
	
	\begin{enumerate}
		\item[0.] If $w_1=0$ then $|gu-hv|_{u,v}\geq 1$ for all nonzero Gaussian integers $g,h$, and $|u-hv|_{u,v}=1$ for at least one $h\in\UU_4$, and $|\tfrac{1}{1+i}u-\tfrac{b}{1+i}v|<1$ for at least one $b\in\UU_4$.
		\item[1.] Suppose that $w_1\neq 0$ and that $|gu-hv|_{u,v}\geq 1 $ (resp. $>1$) for all  $(g,h)\in F$. Then  $|gu-hv|_{u,v}\geq 1 $ (resp. $>1$)
		for all nonzero $g,h$ in $\mathbb{Z}[i]$. If moreover, $w_1\neq 1$, then $w_{2}\in\overline{\DD}$.		
		\item[2.] Suppose that $w_1\neq 0$ and that $|gu-hv|_{u,v} \geq 1 $ (resp. $>1$) for all  $(g,h)\in\{(\tfrac1{1+i},\tfrac{\alpha}{1+i}):\alpha\in\UU_{4}\}$.
		Then 
		$|gu-hv|_{u,v}\geq 1 $ (resp. $>1$) for all nonzero $g,h$ both in $\Z[i]$ or both in $J$. If moreover, $w_1\neq 1$ and  $w_2\neq -1$, then $w_1\in \overline{\CC}\setminus \D(-i,\sqrt 2)$ and $w_2\in\overline{ \mathcal T}$. 
	\end{enumerate}
\end{proposition}

The following simple formula will be useful.

\begin{lemma}[Distance formula]
	\label{lem:formule} Let $u=(u_{1},v_{2}w_{2})$ and $v=(u_{1}w_{1},v_{2})$ be
	in $\mathbb{C}^{2}$. Assume that $|w_{1}|,|w_{2}|\leq 1$ and $|u_{1}|,|v_{2}|>0$.
	Then for all nonzero complex numbers $a$ and $b$,
	\begin{align*}
		|au-bv|_{u,v}=\max(|b|\dd(w_{1},\tfrac{a}{b}),|a|\dd(w_{2},\tfrac{b}{a})).
	\end{align*}	
\end{lemma}

\begin{proof}%
	Since $|w_1|,|w_2|\leq 1$, for any $x=(x_1,x_2)\in \C^2$, $|x|_{u,v}=\max(\tfrac{|x_1|}{|u_1|},\tfrac{|x_2|}{|v_2|})$. Therefore
	\begin{align*}
		|au-bv|_{u,v}  &  =\max(\tfrac1{|u_1|}|au_{1}-bu_{1}w_{1}|,\tfrac1{|v_2|}|av_{2}w_{2}-bv_{2}|)\\
		&  =\max(|b||w_{1}-\tfrac{a}{b}|,|a||w_{2}-\tfrac{b}{a}|).
	\end{align*}	
\end{proof}

\begin{proof}[Proof of the proposition]
	The proof needs only elementary geometry but is rather long, the strategy works as follows. We assume that $w_1\in \overline\CC$, $|w_2|\leq 1$ and 
	$|gu-hv|_{u,v}\geq 1 $ for all nonzero $(g,h)\in F$ or for all $(g,h)\in J^2$  with $|g|=|h|=\tfrac1{\sqrt 2}$. We want to show that $|gu-hv|_{u,v}\geq 1 $ for all nonzero  Gauss  integers or for all $(g,h)\in \Z[i]^2\cup J^2$.
	\begin{enumerate}
		\item We first get rid of the four particular cases $w_1=0$,  $w_1=1$, $w_2=-i$ and $w_2=-1$.		
		\item We show that $|gu-hv|_{u,v}\geq 1$ for $(g,h)=(1,1)$ and $(1,1-i)$ implies that $w_2\in\overline\DD$ 		
		\item We show that $|gu-hv|_{u,v}\geq 1$ for $(g,h)=(\tfrac1{1+i},\tfrac1{1+i})$ and $(\tfrac{i}{1+i},\tfrac{1}{1+i})$ implies that $w_2\in \overline{\mathcal T}$.
		\item Let $a$ be a positive real number. We show that, if $g$ and $h$ are two  nonzero complex numbers such that $|g|,|h|\geq \tfrac1a$ and,  $\tfrac{|g|}{|h|}$ or $\tfrac{|g|}{|h|}> 1+a$, then $|gu-hv|_{u,v}> 1$. 
		\item We show that if $|gu-hv|_{u,v
		}\geq 1$ for  $(g,h)\in\{(1,1),(i,1),(1,1-i)\}$, then $|gu-hv|_{u,v
		}> 1$ for all complex numbers $g$ and $h$ such that $|g|$ and $|h|\geq 3$ (Lemma \ref{lem:33}).
		\item Thanks to points (4) and (5), we shall see that we are reduced to deal with the pairs $(g,h)$ with $|g|$ and $|h|\leq 6$. Since $g$ and $h$ are in $\Z[i]$ or in $J$, we are left with finitely many pairs $(g,h)$. Then we are able to conclude the proof with a computer. 
	\end{enumerate}
	1) The four particular cases $w_1=0$, $w_1=1$ $w_2=-i$ and $w_2=-1$. 
	
	{\it Case $w_1=0$.} For all nonzero $a,b\in\C$, 
	\[
	|au-bv|_{u,v}=\max(|a|,|a||w_2-\tfrac{b}{a}|).
	\]
	Therefore, $|au-bv|_{u,v}\geq 1$ for all nonzero $a\in \Z[i]$. Furthermore, since the four closed disks $D(b,1)$, $b\in\UU_{4}$, cover the closed disk $D(0,1)$, we have $|1\times u-bv|_{u,v}=1$ for at least one $b\in\UU_4$, and $|\tfrac{1}{1+i}u-\tfrac{b}{1+i}v|_{u,v}<1$ for at least one $b\in\UU_4$.\smallskip
	
	In the three other cases we don't have to consider the strict inequalities because $|w_1|$ or $|w_2|=1$.\smallskip
	
	{\it Case $w_1=1$.} For all nonzero $a,b\in\C$, 
	\[
	|au-bv|_{u,v}=\max(|a-b|,|a||w_2-\tfrac{b}{a}|).
	\]
	Therfore, $|au-bv|_{u,v}\geq 1$ for all $a\neq b$ both in $\Z[i]$ or both in $J$. If $a=b$, then $|au-bv|_{u,v}=|a||u-v|_{u,v}$, hence $|au-av|_{u,v}\geq 1$ for all nonzero $a\in \Z[i]$ iff $|u-v|_{u,v}\geq 1$, and $|au-av|_{u,v}\geq 1$ for all  $a\in J$ iff $|\tfrac{1}{1+i}u-\tfrac{1}{1+i}v|_{u,v}\geq 1$. So the proposition holds when $w_1=1$.\smallskip

	{\it Cases $w_2=\alpha=-1$ or $-i$.} For all nonzero $a,b\in\C$, 
	\[
	|au-bv|_{u,v}=\max(|b||w_1-\tfrac{a}{b}|,|a\alpha-b|).
	\]
	Therefore, $|au-bv|_{u,v}\geq 1$ for all $b\neq \alpha a$ both in $\Z[i]$ or both in $J$. If $b=\alpha a$, then $|au-bv|_{u,v}=|a||1\times u-\alpha v|_{u,v}$, hence $|au-\alpha av|_{u,v}\geq 1$ for all nonzero $a\in \Z[i]$ iff $|1\times u-\alpha v|_{u,v}\geq 1$ which always holds when $\alpha=-1$ because $|1\times u-\alpha v|_{u,v}=|w_1+1|$ and  $w_1\in\overline{\CC}$. Likewise, $|au-\alpha av|_{u,v}\geq 1$ for all  $a\in J$ iff $|\tfrac{1}{1+i} u-\tfrac{\alpha}{1+i}v|_{u,v}\geq  1$. So the proposition holds when $w_2=\alpha$.\smallskip

	We now suppose that $w_1\neq 0$, $w_1\neq 1$, $w_2\neq -i$ and $w_2\neq -1$.\smallskip

	2) 
	By the distance formula (Lemma \ref{lem:formule}), for all nonzero $g$ and $h$, 
	\[
	|gu-hv|_{u,v}=\max(|h|\operatorname{d}(w_{1},\tfrac{g}{h}), |g|\operatorname{d}%
	(w_{2},\tfrac{h}{g})).
	\]
	Now since $w_{1}\in\overline{\mathcal{C}}\setminus\{0,1\}$, we have
	\[
	|1|\operatorname{d}(w_{1},\tfrac{1}{1})<1 \text{ and } |1-i|\operatorname{d}%
	(w_{1},\tfrac{1}{1-i})<1.
	\]
	Therefore, if $|gu-hv|_{u,v}\geq 1$ (resp. $>1$) for $(g,h)=(1,1)$ and $(1,1-i)$,  then
	
	\[
	|1|\operatorname{d}(w_{2},\tfrac{1}{1})\geq 1 \text{ and } |1|\operatorname{d}(w_{2}%
	,\tfrac{1-i}{1})\geq 1,
	\]
	(resp. $>1$) which in turn implies $w_2\in\overline\DD$ (resp. $w_2\in\DD$).
	
	3) Since $w_1\in \overline{\mathcal{C}}\setminus\{0,1\}$, we have
	\[
	|\tfrac{1}{1+i}|\dd(w_1,1=\tfrac{\tfrac{1}{1+i}}{\tfrac{1}{1+i}})<1 \text{ and } |\tfrac{1}{1+i}|\dd(w_1,i=\tfrac{\tfrac{1}{1+i}}{\tfrac{-i}{1+i}})<1.
	\]
	Therefore, if $|gu-hv|_{u,v}\geq 1$  (resp. $>1$) for $(g,h)=(\tfrac{1}{1+i},\tfrac{1}{1+i})$ and $(g,h)=(\tfrac{1}{1+i},\tfrac{-i}{1+i})$, then 
	\[
	\dd(w_2,1)\geq \sqrt 2 \text{ and } \dd(w_2,-i)\geq \sqrt 2
	\]
	(resp. $>\sqrt 2$) which in turn implies $w_2\in\overline{ \mathcal T}$ (resp. $w_2\in\mathcal T$).
	
	4) Let $a$ be a positive real number. Let $g,h$ be two nonzero complex numbers with $|g|,|h|\geq \tfrac1a$.
	Since $|w_{1}|$ and $|w_{2}|\leq 1$, if $\tfrac{|g|}{|h|}$ or $\tfrac
	{|h|}{|g|}> 1+a$, then by the distance formula (Lemma \ref{lem:formule}),
	\begin{align*}
	|gu-hv|_{u,v}&=\max(|h|\dd(w_1,\tfrac{g}{h}),|g|\dd(w_2,\tfrac{h}{g}))\\
	&\geq \max(|h|(\tfrac{|g|}{|h|}-1),|g|(\tfrac{|h|}{|g|}-1))> 1.
	\end{align*}
	
	5)
	\begin{lemma}\label{lem:33} Suppose $w_1\in \overline{\mathcal{C}}\setminus\{0,1\}$ and $w_2\in D(0,1)$. If $|gu-hv|_{u,v}\geq 1$ for  $(g,h)\in\{(1,1),(i,1),(1,1-i)\}$ then $|gu-hv|_{u,v}> 1$ for all complex numbers $g$ and $h$ such that $|g|$ and $|h|\geq 3$.
	\end{lemma}
	\begin{proof}[Proof of Lemma \ref{lem:33}]
		We
		proceed by contradiction and assume that $|gu-hv|_{u,v}\leq1$ for some complex numbers $g$ and $h$ with $|g|$ and $|h|\geq 3$. Set
		$z=\tfrac{g}{h}$ and $z^{\prime}=\tfrac{1}{z}$. By the distance formula (Lemma \ref{lem:formule}),
		\[
		\operatorname{d}(w_{1},z) \text{ and } \operatorname{d}(w_{2},z^{\prime}%
		)\leq\tfrac13
		\]
		hence $|z|,|z^{\prime}|\leq \tfrac43$. It follows that $|z|,|z^{\prime}|\geq 
		\tfrac34$ and then that $|w_{1}|,|w_{2}|\geq  \tfrac34-\tfrac13=\tfrac5{12}$.
		
		Also observe that since $w_{1}\in\overline{\mathcal{C}}$, $\Re z$ and $\Im z$ are
		$\geq-\tfrac13$ which implies that the inverse $z^{\prime}$ of $z$ is neither
		in the open disk $\D(-\tfrac32,\tfrac32)$ nor in the open disk $\D(\tfrac
		32i,\tfrac32)$.
		
		We divide the proof in two cases:
		
		\begin{enumerate}
			\item $\operatorname{d}(w_{1},i)\geq1$,
			
			\item $\operatorname{d}(w_{1},i)<1$.
		\end{enumerate}
		
		The first case uses the following intermediate lemma.
		
		\begin{lemma}
			Let $w\in\C$ be such that $\Re w,\,\Im w\geq0$,
			$\tfrac5{12}\leq |w|\leq1$,   and $\operatorname{d}(w,i)\geq1$ then
			$\operatorname{d}(w,1)<\tfrac23$.
		\end{lemma}
		
		\begin{proof}[Proof of the intermediate lemma]
			We want to show that the function $f(z)=|z-1|^{2}%
			-\tfrac49$ is $< 0$ when $|z|\leq 1$, $\Re z\geq 0$, $\Im z\geq 0$  and $z$ is neither in the interiors of 
			$D(i,1)$ nor in the interior of $D(0,\tfrac5{12})$. It is easy to see that the maximum of $f$ on this region is
			reached at a point which belongs to the circle $C$ of radius $1$ and  center
			$i$. The circle $C$ has polar equation $r=2\sin\theta$. Since for
			$z=re^{i\theta}\in C$,
			\[
			f(re^{i\theta})=\tfrac59+r^{2}-2r\cos\theta=\tfrac59+4\sin^{2}\theta
			-2\sin2\theta=g(\theta),
			\]
			it is enough to show that $g(\theta)< 0$ for $\theta\in[\arcsin\tfrac
			5{24},\tfrac{\pi}6]$. Now $g^{\prime}(\theta)=4\sin2\theta-4\cos2\theta$ is
			negative if $\theta<\tfrac{\pi}8$ and non-negative otherwise, hence it is
			enough to check that $g$ is $< 0$ at the extremities of the interval. Since
			\begin{align*}
				g(\tfrac{\pi}6)  &  =\tfrac59+1-\sqrt3< 0,\\
				g(\arcsin\tfrac5{24})  &  =\tfrac59+(\tfrac5{12})^{2}- 4\tfrac5{24}%
				\sqrt{1-(\tfrac5{24})^{2}} \leq-0.08\\
			\end{align*}
			we are done.
		\end{proof}
		
		\textit{End of  proof of Lemma \ref{lem:33}.}
		
		Case 1: $\dd(w_1,i)\geq 1$. By the above lemma, $|z-1|\leq|z-w_{1}|+|w_{1}-1|< \tfrac13+\tfrac23=1$, hence its
		inverse $z^{\prime}$ has a real part  $>1/2$. We also already know that $z^{\prime}$ is not
		in the open disk $\D(\tfrac32i,\tfrac32)$.
		
		Since $w_{1}$ is in $\overline{\mathcal{C}}$ and $w_1\neq 0$ , $\operatorname{d}(w_{1},1)<1$. By
		assumption $|u-v|_{u,v}\geq 1$, therefore by the distance formula (Lemma \ref{lem:formule}) with
		$a=1$ and $b=1$, we obtain $\operatorname{d}(w_{2},1)\geq 1$. Since $w_{1}\in
		\overline{\mathcal C}\setminus\{0,1\}\subset \D(\tfrac{1+i}{2},\tfrac1{\sqrt2})$, $\operatorname{d}%
		(w_{2},1-i)\geq 1$ again by Lemma \ref{lem:formule} with $a=1$ and $b=1-i$.
		
		Finally $z^{\prime}$ and $w_{2}$ satisfy the inequalities
		\[
		z'\notin\D(\tfrac32i,\tfrac32),\,\Re z^{\prime}>\frac12,\,\text{ and }\operatorname{d}(w_{2},1)\geq1,\,\operatorname{d}%
		(w_{2},1-i)\geq 1  
		\]
		contradicting $\operatorname{d}(z^{\prime},w_{2})\leq\frac13$ because $\tfrac{\sqrt3}2-\tfrac12> \tfrac13$ (see Figure 1).		
		\begin{figure}
		\includegraphics[width=10cm]{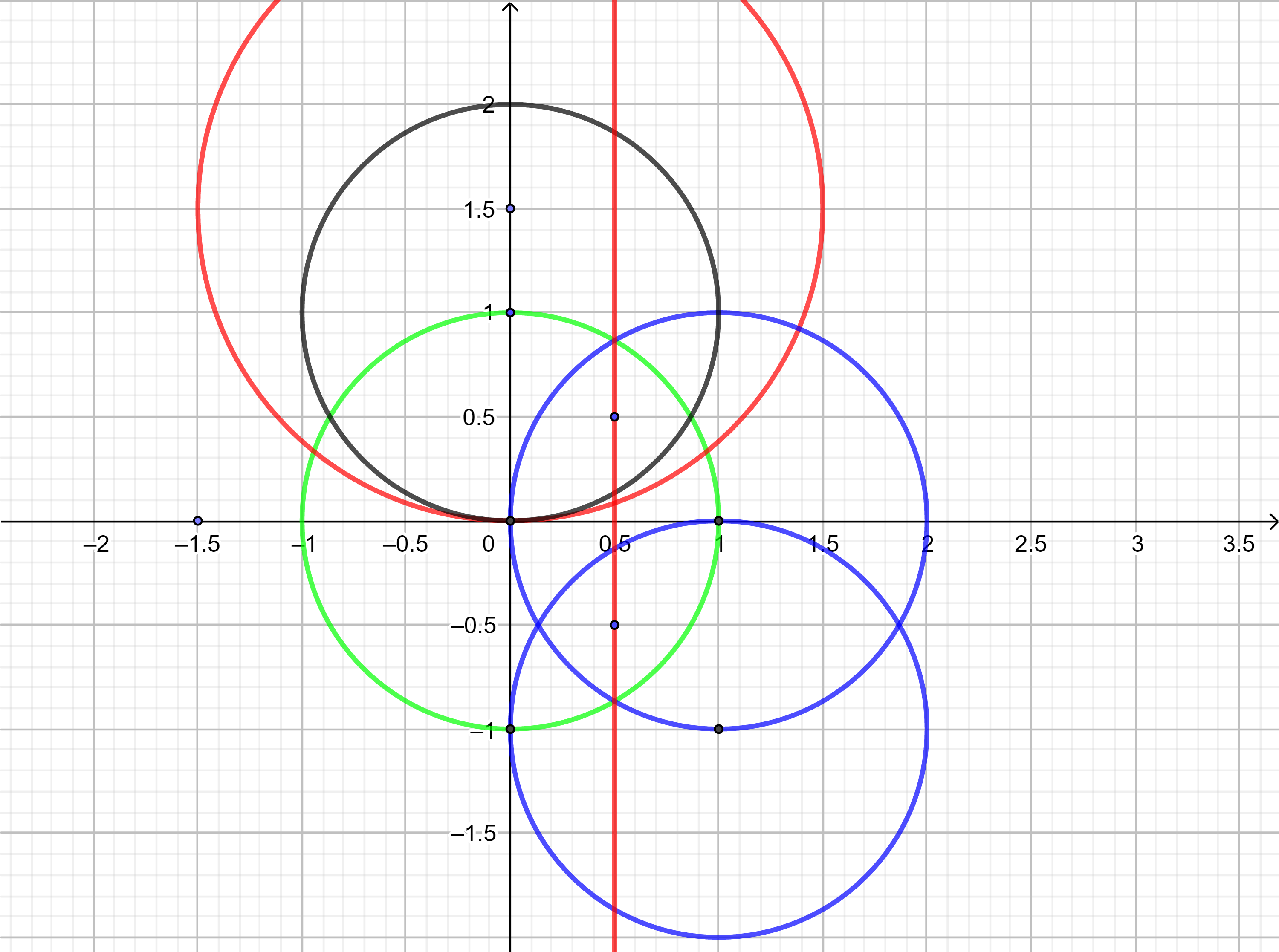}
		\caption{Proof of Lemma \ref{lem:33}, case 1}
		\end{figure}
		
		Case 2: $\dd(w_1,i)< 1$. We already know that $|z^{\prime}|\geq \tfrac34$ and that $z^{\prime}$
		is neither in the open disk $\D(-\tfrac32,\tfrac32)$ nor in the open disk
		$\D(\tfrac32i,\tfrac32)$
		
		As in case (1), making use of lemma \ref{lem:formule} with $a=1$ and $b=1$ we see that
		$\operatorname{d}(w_{2},1)\geq1$. Since $\dd(w_1,i)<1$ (case(2)), again with $a=i$ and $b=1$ we see that
		$\operatorname{d}(w_{2},-i)\geq 1$.
		
		It follows that $\operatorname{d}(w_{2},z^{\prime})>\tfrac13$ (see Figure 2), a contradiction.
		\begin{figure}
		\includegraphics[width=10cm]{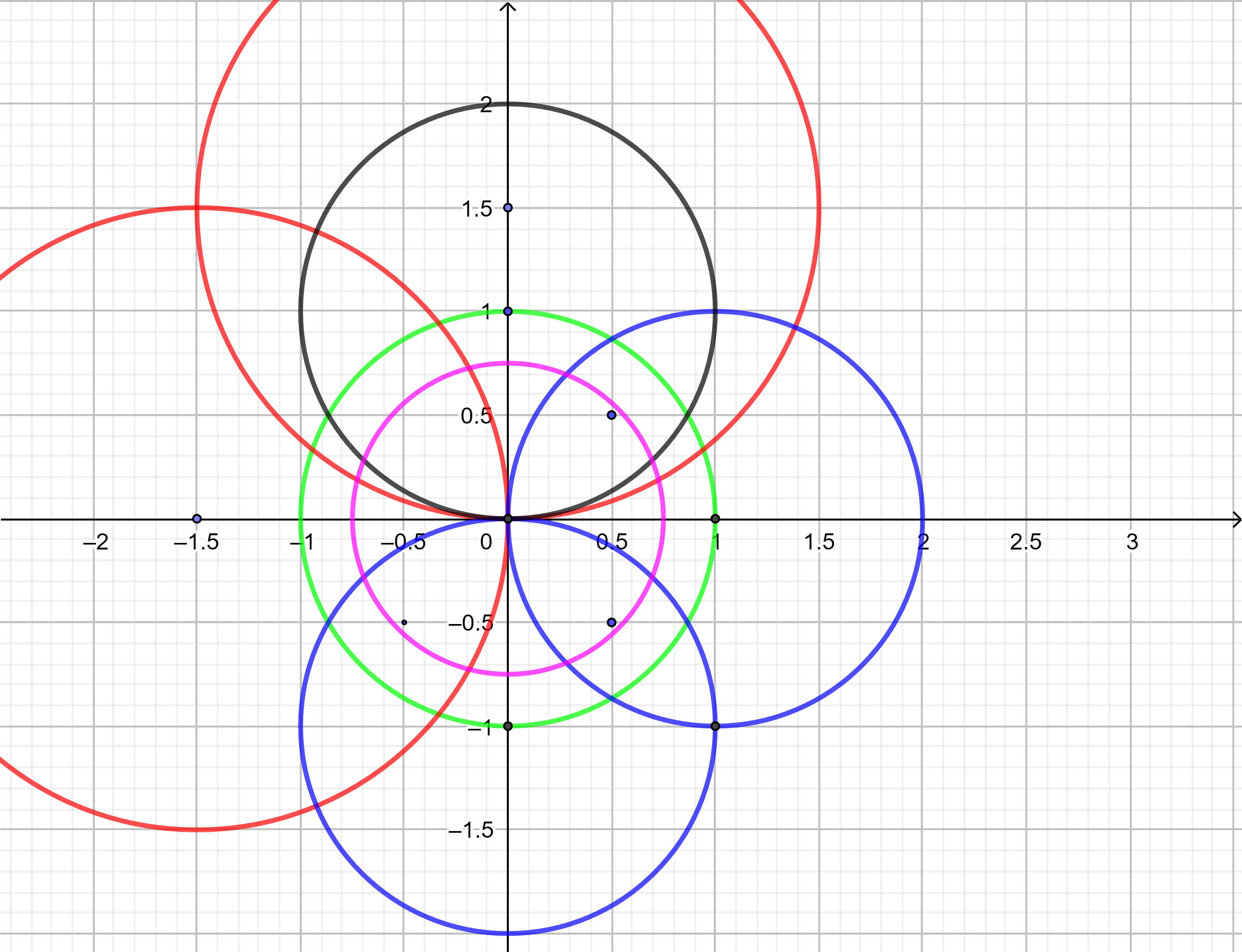}
		\caption{Proof of Lemma \ref{lem:33}, case 2}
		\end{figure}
	\end{proof}
	6)
	It remains to study the case $|g|$ and $|h|\leq6$. Indeed, if $g$ and $h$
	are nonzero complex numbers such that $|g|$ or $|h|>6$, say $|g|>6$, and taking into account that $g$ and $h$ are both Gaussian integers or both in $J$, we have
	\begin{itemize}
		\item either, $|h|\geq 1$ and  
		\begin{itemize}
			\item either, $\tfrac{|g|}{|h|}>2$, and point 4 with $a=1$ implies $|gu-hv|_{u,v}> 1$,
			\item or, $\tfrac{|g|}{|h|}\leq 2$ and $|h|\geq3$, and point 5 implies
			$|gu-hv|_{u,v}> 1$.
		\end{itemize}
		\item or, $|h|=\tfrac1{\sqrt 2}$ and $\tfrac{|g|}{|h|}>1+\sqrt 2$, and  4) with $a=\sqrt 2$ implies $|gu-hv|_{u,v}>1$. 
	\end{itemize}
	
	Observe that up to now we have only used the hypothesis: $w_1\in \overline{\mathcal{C}}\setminus\{0,1\}$, $w_2\in D(0,1)$ and $|gu-hv|_{\infty
	}\geq 1$ for  $(g,h)\in\{(1,1),(i,1),(1,1-i)\}$.
	
	We now use a computer to prove two lemmas.
	\begin{lemma}[First set of critical pairs]\label{lem:computer1}
		For all nonzero Gaussian integers $g$ and $h$ with $|g|$ and $|h|\leq6$, if
		the pair $(g,h)$ is not in $\Z[i]G_1$ where 
		\begin{align*}
			G_1=&\{(1,-i),(1,1),(1,i),(1,-1),(1,1+i),(1,1-i),(1,-1+i),(1,-1-i),\\
			&(1,2i),(1,-2i),(1,-2),(1+i,1),(1+i,i),(1+i,2-i),(2,1),(2,1-2i),\\
			&(2-i,-2i),(2+i,2-2i)
			\}
		\end{align*}
		then
		\[
		|h|\operatorname{d}(\tfrac{g}{h},\CC) \text{ or } |g|\operatorname{d}%
		(\tfrac{h}{g},\DD)> 1.
		\]
		$G_1$ is called the first set of critical pairs.
	\end{lemma}
	\begin{proof}[Proof using a computer]
		\begin{itemize}
		\item Let $L_1$ be the set of pairs of nonzero Gaussian integers with  moduli $\leq 6$. The set $L_1$ is finite with less than $(6+1+6)^4=28501$ elements and can be generated using a simple computer code (we use Python code).
		\item One can write two functions that calculate for any complex number $z$, the two distances $\dd(z,\CC)$ and $\dd(z,\DD)$. See Appendix Section \ref{sec:distance} where it is  explained how to calculate $\dd(z,\DD)$. This calculation can be performed with standard floating point arithmetic. The distance to $\CC$ can be calculated the same way. 
		\item Using these two functions one can obtain the set $L'_1$ of pairs $(g,h)\in L_1$ such that
		\[		
		|h|\operatorname{d}(\CC,\tfrac{g}{h})\leq 1+\eps \text{ and } |g|\operatorname{d}%
		(\DD,\tfrac{h}{g})\leq 1+\eps,
		\]
		with $\eps=0.001$, a numerical safety margin. The set $L'_1$ certainly contains all the pairs such that $|h|\operatorname{d}(\CC,\tfrac{g}{h})\leq 1$ and $|g|\operatorname{d}
		(\DD,\tfrac{h}{g})\leq 1$. 
		\item Finally extract from $L'_1$, a minimal subset $G'_1$ such that for each pair $(a,b)\in L'_1$ there exist $z\in\ZZ$ and $(g,h)\in G'_1$ such that $(a,b)=z(g,h)$. For this step observe that if $(a,b)\in L'_1$ then, there exists a primitive pair $(g,h)\in\ZZ^2$ which is in the line $\C(a,b)$ and which  is also in $L'_1$ because $|au-bv|_{u,v}\geq|gu-bv|_{u,v}$. 
		\item The pairs added in $G'_1$ due to the numerical margin are validated using calculation by hand. This lead to the set $G_1$ (actually, with the margin $\eps=0.001$, $G_1=G'_1$).
		\end{itemize}
	Suppose now that $(a,b)$ is a pair of nonzero Gaussian integers such that $|h|,|g|\leq 6$ and $|au-bv|_{u,v}\leq 1$. There exists a primitive pair $(g,h)\in\ZZ^2$ such that $(a,b)=z(g,h)$ with $|z|\geq 1$. Since $|z|\geq 1$, $1\geq|au-bv|_{u,v}\geq|gu-bv|_{u,v}$. Therefore $(g,h)\in L'_1$. Since $(g,h)$ is primitive, on the one hand, one of the pairs $\alpha(g,h)$, $\alpha\in\UU_{4}$ must be in $G'_1$, and on the other hand, $z\in\ZZ$. It follows that $(a,b)\in\ZZ G_1$.
	\end{proof}
\begin{remark}
	Without the  safety margin $\eps$ in the above proof, some pairs may be missing from the set $G_1$  as the referee pointed out.
\end{remark}	
	\begin{lemma}[Second set of critical pairs]\label{lem:computer2}
		For all $(g,h)\in J^{2}$ with $|g|$ and $|h|\leq 6$, and $(g,h)\notin
		G_2=\{(\frac{a}{1+i},\frac{b}{1+i}):a,b\in\mathbb{U}_{4}\}$, we have
		\[
		|h|\operatorname{d}(\tfrac{g}{h}, \overline{\CC}\setminus \D(-i,\sqrt2)) \text{ or }
		|g|\operatorname{d}(\tfrac{h}{g},\mathcal{T})>1.
		\]	
	\end{lemma}
\begin{proof}[Proof using a computer]
	\begin{itemize}
	\item Let $L_2$ be the set of pairs of nonzero elements in $J$ with  moduli $\leq 6$. The set $L_2$ is finite with less than $(8+1+8)^4=83521$ and can be generated a using simple computer code (we use Python code).
	\item One can write a function that calculates for any complex number $z$, the distance $\dd(z,\mathcal T)$ from $z$ to $\mathcal T$. See Appendix Section \ref{sec:distance} where it is explained how to calculate $\dd(z,\DD)$.  The distance to $\mathcal T$ can be calculated the same way. 
	\item Since $\overline{\CC}\setminus \D(-i,\sqrt2))\subset (-i)\overline{\mathcal T}$, for each nonzero $g$, $h$,
	\[
	\dd(\tfrac{g}{h},\overline{\CC}\setminus \D(-i,\sqrt2)))\leq \dd(i\tfrac{g}{h},\mathcal T).
	\]
	\item Using the function $\dd(z,\mathcal T)$, one can obtain the set $L'_2$ of pairs $(g,h)\in L_2$ such that
	\[		
	|h|\dd(i\tfrac{g}{h},\mathcal T)\leq 1+\eps \text{ and } |g|\dd
	(\tfrac{h}{g},\mathcal T)\leq 1+\eps,
	\]
	with $\eps=0.001$, a numerical safety margin. We obtain
	\[
	L'_2=\{(\tfrac{a}{1+i},\tfrac{b}{1+i}):a,b\in\UU_{4}\}.
	\]
\end{itemize}		
\end{proof}
	
	{\bf End of proof  of Part 1 in Proposition \ref{prop:conditions}.}
	Recall that we suppose $w_1\neq 0,1$ and $w_2\neq -i,-1$. By 2), we already know that $w_2\in \overline{\DD}$.
	 It remains to prove that, if  
	\[
	|gu-hv|_{u,v}\geq 1 \text{ (resp. $>1$) for all }(g,h)\in F=    \{(1,1),(1,-i),(1,1-i),(1,1+i),(1+i,1)\},
	\]
	then $|gu-hv|_{u,v}\geq 1 $ (resp. $>1$) 
	for all nonzero $g,h$ in $\mathbb{Z}[i]$.	\smallskip 
	
	 By Lemma \ref{lem:computer1},  if $|gu-hv|_{u,v}\geq 1 $ (resp. $>1$)  for all $(g,h)\in G_1$, then $|gu-hv|_{u,v}\geq 1 $ (resp. $>1$) for all pairs $(g,h)$ of nonzero Gaussian integers. We prove that we can remove the pairs in $G_1\setminus F$ and get the same conclusion.

	When $(g,h)\in\{(1,i),(1,-1),(1,-1+i),(1,-1-i), (1,2i),(1,-2)\}$, we have
	\begin{align*}
		&|h|\dd(\tfrac{g}{h},\CC)=1,\\
		&\forall w_1\in\overline{\CC}\setminus\{0\},\,|h|\dd(\tfrac{g}{h},w_1)>1.
	\end{align*}		
	So by the distance formula, these six pairs can be removed from $G_1$ when dealing with the large inequality or the strict inequality.

	When $(g,h)\in\{(1+i,i),(2,1)\}$, we have
	\begin{align*}
		&|h|\dd(\tfrac{g}{h},\CC)=1,\\
		&\forall w_1\in\overline{\CC}\setminus\{1\},\,|h|\dd(\tfrac{g}{h},w_1)>1.
	\end{align*}
So these two pairs can be removed.
	
	When $(g,h)\in\{(1,-2i),(1+i,2-i),(2,1-2i),(2+i,2-2i)\}$, we have
	\begin{align*}
		&|g|\dd(\tfrac{h}{g},\DD)=1,\\
		&\forall w_2\in\overline{\DD}\setminus\{-i\},\,|g|\dd(\tfrac{h}{g},w_2)>1.
	\end{align*}	 
So  these four pairs can be removed from $G_1$.

	Finally consider the pair $(g,h)=(-2+i,2i)$.  Since the disk $D(\tfrac{g
	}{h},\tfrac{1}{|h|})$ is included in the disk $D(\tfrac{g'}{h'},\tfrac{1}{|h'|})$ where $(g',h')=(1,1+i)$ and
	since the portion of the disk $D(\tfrac{h}{g},\tfrac{1}{|g|})$
	lying in the unit disk is included in the disk $D(\tfrac{h'}{g'},\tfrac{1}{|g'|})$, the
	inequality $|g'u-h'v|_{u,v}\geq 1$ (resp. $>1$) implies the inequality $|g
	u-hv|_{u,v}\geq 1$ (resp. $>1$) which means that we can remove the pair
	$(-2+i,2i)$.
	
	It follows that for $w_{1}\in\overline{\mathcal{C}}\setminus\{0,1\}$ and $w_2\in D(0,1)\setminus\{-i\}$, if for all $(g,h)\in F$ 
	$|gu-hv|_{u,v}\geq 1,$ (resp. $>1$) then $|gu-hv|_{u,v}\geq 1$ (resp. $>1$) for all nonzero Gaussian integers $g$ and $h$. \bigskip

	{\bf End of proof of Part 2 in Proposition \ref{prop:conditions}.} 
	By 3), we already know  $w_{2}\in\overline{\mathcal{T}}$. It remains to prove that if  $|gu-hv|_{u,v} \geq 1 $ (resp. $>1$) for all $(g,h)\in\{(\frac{a}{1+i},\frac{b}{1+i}):a,b\in\mathbb{U}_{4}\}$ then $w_{1}\in\overline{\CC}\setminus \D(-i,\sqrt2)$  and
	$|gu-hv|_{u,v}\geq 1 $ (resp. $>1$) for all nonzero $g,h$ both in $\Z[i]$ or both in $J$

	Since $w_2\in\overline{\mathcal{T}}$ and since $w_2\neq -1$, $\dd(w_2,i)<\sqrt 2$. Now by assumption, $|\tfrac1{1+i}u-\tfrac{i}{1+i}v|_{u,v}\geq 1$, hence $\dd(w_1,-i)\geq \sqrt 2$ by the distance formula. This means that $w_{1}\in\mathcal{C}\setminus \D(-i,\sqrt2)$.

	Since $w_{1}\in\mathcal{C}\setminus \D(-i,\sqrt2)$ and $w_2\in\overline{\mathcal T}$, we can use Lemma \ref{lem:computer2}, we obtain that  $|gu-hv|_{u,v}\geq 1$ (resp. $>1$) for all $g,h$ both in $J$. 
	
	It remains to see what is happening when $g$ and $h$ are Gaussian integers.
	
	Since $\mathcal T\subset \DD$, using the first set of critical pairs and the proof of Part 1, we see that  only the  pairs $(g,h)$ in $F$ must be examined. For each of these pairs, we have $|g|\dd(\tfrac{h}{g},\mathcal T)> 1$ except for $(g,h)=(1,1+i)$, and for this latter pair $\dd(w_2,\tfrac{h}{g})>1$ for all $w_2\in\mathcal T\setminus\{i\}$, so we are done.
\end{proof}

\subsection{Constraints on the pairs $(w_1,w_2)$  when  $w_{1}\in\mathcal{C}$}
\label{subsec:constraints}
Recall the sets we need 
\begin{align*}
	\mathcal{C}=  &  \{z\in\mathbb{C}: |z|<1,\,\arg z\in[0,\tfrac{\pi}4 ]\},\\
	\mathcal{D}=  &  \{z\in\mathbb{C}:|z|<1,\, \operatorname{d}(z,1)>
	1,\,\operatorname{d}(z,1-i)> 1 \},\\
	\mathcal{T}=  &  \{z\in\mathbb{C}:|z|<1,\, \operatorname{d}(z%
	,1)>\sqrt2,\,\operatorname{d}(z,-i)> \sqrt2 \},\\
	F=  &  \{(1,1),(1,-i),(1,1-i),(1,1+i),(1+i,1)\}.\\
\end{align*}

\begin{figure}[H]
	\label{fig:3}
	\begin{center}
		\includegraphics[width=10cm]{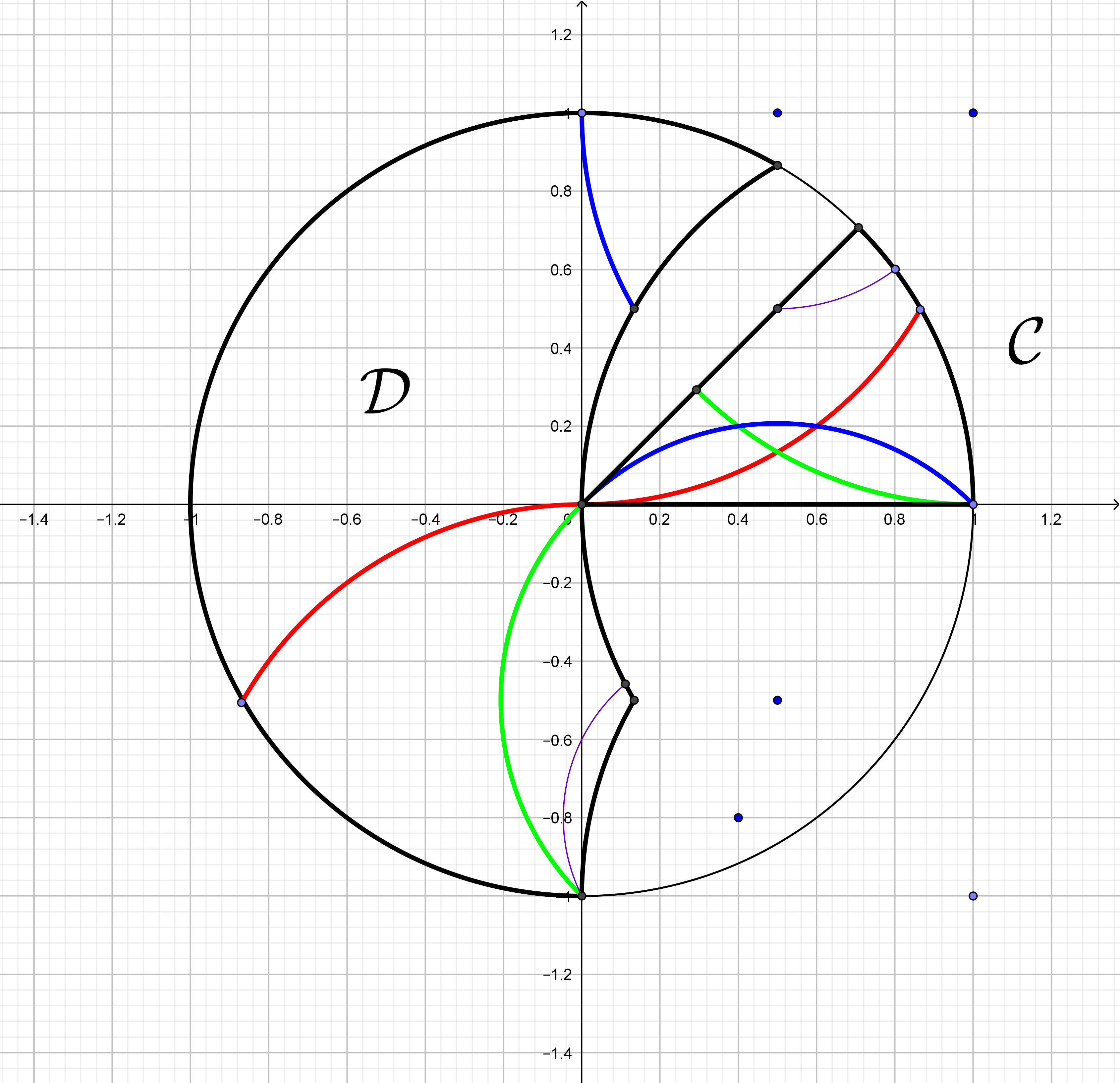}
		\caption{The constraints on consecutive minimal vectors.}
	\end{center}
\end{figure}

Consider the following pairs of open disks in $\mathbb{C}$
\begin{align*}
	&  Blue_{1}=\D(\tfrac{1-i}{2},\tfrac1{\sqrt2}),\,Blue_{2}=\D(1+i,1)\\
	&  Red_{1}=\D(i,1),\, \,Red_{2}=\D(-i,1)\\	
	&  Green_{1}=\D(1+i,1),\,Green_{2}=\D(\tfrac{1-i}{2},\tfrac1{\sqrt2}).
\end{align*}

\begin{corollary}\label{cor:ConstraintsLarge}
	\label{thm:T} Let  $u=(u_{1},v_{2}w_{2})$ and
	$v=(u_{1}w_{1},v_{2})$ be two vectors in $\C^2$ with $|u_{1}|,|v_{2}|>0$, $|w_1|,|w_2|<1$, and  $w_{1}\in\mathcal{C}\setminus\{0\}$. 
	
	\begin{itemize}
		\item No nonzero vector in $\ZZ u+\ZZ v$ are in $\overset{o}C(u,v)$ iff $w_{2}\in\overline{\mathcal{D}}$ and one of the four conditions
		
		\begin{enumerate}
			\item $w_{2}\in Green_{2}$ and $w_{1}\notin Red_{1}\cup Green_{1}$,
			
			\item $w_{2}\in Red_{2}\setminus Green_{2}$ and $w_{1}\notin Red_{1}$,
			
			\item $w_{2}\notin Red_{2}\cup Blue_{2}$,
			
			\item $w_{2}\in Blue_{2}$ and $w_{1}\notin Blue_{1}$,
		\end{enumerate}
		
		holds.
		
		\item No nonzero vector in $\langle u,v\rangle_J$ are in $\overset{o}C(u,v)$ iff $(w_{1},w_{2})\in(\mathcal{C}\setminus \D(-i,\sqrt2))\times
		\overline{\mathcal{T}}$.
	\end{itemize}
\end{corollary}
\begin{remark}
	We do not consider the particular cases $w_1=0$ or $|w_i|=1$ because we will not use them in the sequel. But clearly, it is not difficult to find the constraints on $(w_1,w_2)$ in these cases using Proposition \ref{prop:conditions}. In fact when $w_1=0$ then  $(\ZZ u+\ZZ v)\cap\overset{o}C(u,v)=\{0\}$ whatever the value of $w_2\in\DD$, and $\langle u,v\rangle_J\cap \overset{o}C(u,v)\neq \{0\}$   whatever the value of $w_2\in\DD$.
\end{remark}
\begin{proof}
	By Proposition \ref{prop:conditions}, no nonzero vector in $\ZZ u+\ZZ v$ are in $\overset{o}C(u,v)$
	iff for all $(g,h)$ in
	\[
	F=\{(1,1),(1,-i),(1,1-i),(1,1+i),(1+i,1)\},
	\]
	$|gu-hv|_{u,v}\geq 1$. By the distance formula (Lemma \ref{lem:formule}), it means that for all
	$(g,h)\in F$,
	\[
	\operatorname{d}(w_{1},\tfrac{g}{h})\geq\tfrac1{|h|} \text{ or } \operatorname{d}%
	(w_{2},\tfrac{h}{g})\geq\tfrac1{|g|}.
	\]
	Taking $(g,h)=(1,1)$ and $(g,h)=(1,1-i)$, we obtain that $w_{2}\notin \D(1,1)$
	and $w_{2}\notin \D(1-i,1)$ which implies $w_{2}\in\overline{\mathcal{D}}$.
	
	Taking $(g,h)=(1,-i)$, we obtain $w_{1}\notin \D(i,1)=Red_{1}$ or $w_{2}\notin
	\D(-i,1)=Red_{2}$.
	
	Taking $(g,h)=(1,1+i)$, we obtain $w_{1}\notin \D(\tfrac{(1-i)}{2}
	,\tfrac1{\sqrt2})=Blue_{1}$ or $w_{2}\notin \D(i+i,1)=Blue_{2}$.
	
	Taking $(g,h)=(1+i,1)$, we obtain $w_{1}\notin \D(1+i,1)=Green_{1}$ or
	$w_{2}\notin \D(\tfrac{1-i}{\sqrt2},\tfrac1{\sqrt2})=Green_{2}$.
	
	Taking the position of the disks $Red_{2}$, $Blue_{2}$ and $Green_{2}$ in
	$\overline{\mathcal{D}}$ into account, we obtain that $|gu-hv|_{u,v}\geq 1$ for all $(g,h)\in F$ iff one of the conditions (1) or (2) or (3) or (4) holds.
	
	Again by Proposition \ref{prop:geonumber}, no nonzero vector in $\langle u,v\rangle_J$ are in $\overset{o}C(u,v)$
	iff for all  $(g,h)\in\{(\tfrac1{1+i},\tfrac1{1+i}),(\tfrac1{1+i},\tfrac{i}{1+i})\}$. The distance formula shows that it is equivalent to	
	$(w_{1},w_{2})\in(\mathcal{C}\setminus D(-i,\sqrt2))\times\overline{\mathcal{T}}$.	
\end{proof}

\subsection{An example of a lattice with linearly independent equivalent minimal vectors}\label{subsec:example2}
We give an example of two vectors $u$ and $v$ such that
\begin{itemize}
	\item $u$ and $v$ are consecutive minimal vectors of $\LL=\ZZ u+ \ZZ v$,
	\item $u-(1-i)v$ is a minimal vector equivalent to $v$,
	\item $(1+i)u-v$ is a minimal vector equivalent to $u$.
\end{itemize}	 
Let  $s\in]\tfrac43\pi,\tfrac32 \pi[$, $t\in]\tfrac56 \pi,\pi[$, $w_1=1+i+e^{is}$, $w_2=1-i+e^{it}$, $u=r(1,e^{i\alpha}w_2)$ and $v=r(w_1,e^{i\alpha})$ where $r>0$ and $\alpha\in\R$ are such that
\[
\operatorname{det}_{\C}(u,v)=r^2e^{i\alpha}(1-w_1w_2)=1.
\]
Consider the lattice $\LL=\ZZ u+\ZZ v$. We have $w_1\in \CC\cap\partial Green_1\setminus \overline{Red_1}$  and $w_2\in \overline{ \DD}\cap\mathbf C(1-i,1)$, so that thanks to Corollary \ref{cor:ConstraintsLarge},
\[
\overset{o}C(u,v)\cap\LL=\{0\}.
\]
However, since $w_1\in \partial Green_1$ and  $w_2\in \mathbf C(1-i,1)$, $\partial C(u,v)$ contains not only the subsets of $\LL$, $\UU_4u$ and $\UU_4v$ but also  the subsets of $\LL$
\[
\UU_4((1+i)u-v) \text{ and } \UU_4(u-(1-i)v).
\] 
By Lemma \ref{lem:computer1} about the first set of critical pairs, the only others $(g,h)$ such that $|gu-hv|_{u,v}=1$ are  in $\UU_4 G_1$. Furthermore since $w_1\neq 0,1$ and $w_2\neq -1,-i$, we see as in the end of the proof of Proposition \ref{prop:geonumber} part 1, that if $(g,h)\in\UU_4(G_1\setminus F)$, then $|gu-hv|_{u,v}>1$.  By checking the values of $|gu-hv|_{u,v}$ for the pairs $(g,h)\in F$, we see that
\[
\partial C(u,v)\cap \LL=\UU_4u\cup\UU_4v\cup\UU_4((1+i)u-v)\cup\UU_4(u-(1-i)v).
\]
Now
\begin{align*}
	u-(1-i)v&=r(1-(1-i)w_1,e^{i\alpha}(w_2-(1-i)))\\
	&=r(1-(1-i)(1+i+e^{is}), e^{i\alpha}(1-i+e^{it}-(1-i)))\\
	&=r(-1-(1-i)e^{is}, e^{i\alpha}e^{it})\\
	&=r(x_1,x_2)\\
	(1+i)u-v&=r(1+i-w_1,e^{i\alpha}((1+i)w_2-1))\\
	&=r(1+i-(1+i+e^{is}), e^{i\alpha}((1+i)(1-i+e^{it})-1))\\
	&=r(-e^{is}, e^{i\alpha}((1+i)e^{it}+1)\\
	&=r(y_1,y_2)
\end{align*}
and one check that
\begin{align*}
	&|x_1|^2=3+2(\cos s+\sin s)=|w_1|^2 \text{ and } |x_2|=1\\
	&|y_1|^2=1 \text{ and } |y_2|=3+2(\cos t-\sin t)=|w_2|\\
\end{align*}
Hence  $v$ and $u-(1-i)v$ are two equivalent minimal vectors and $u$ and $(1+i)u-v$ as well.

\section{No consecutive pairs of index 2}
In Proposition \ref{prop:index} we have seen that two consecutive minimal vectors has index one or two. We show now that the case of index two cannot occur twice consecutively.
\begin{proposition}\label{prop:index2}
	Suppose  that $u=(u_{1},v_{2}w_{2})$ and
	$v=(u_{1}w_{1},v_{2})$ are two consecutive minimal vectors of index $2$ in a lattice $\LL$ in $\C^2$. Let $w$ be a minimal vector such that $v$ and $w$ are two consecutive minimal vectors. Then
	$\mathbb{Z}[i]v+\mathbb{Z}[i]w=\Lambda$.
\end{proposition}
\begin{remark}
	In fact, with the assumptions of the proposition, it is possible to prove that $\mathbb{Z}[i]u+\mathbb{Z}[i]w=\Lambda$ also holds. The proof of this latter fact goes as the proof of the proposition but is slightly more difficult. We shall not do it. 
\end{remark}

The proof of the proposition will use two lemmas. The first lemma is well known and its proof  is a straightforward calculation we omit.
\begin{lemma}\label{lem:circle}
	Let $w$ and $z$ be two complex numbers and let $k<1$ be
	a nonnegative real number. Then
	\begin{align*}
		|z-w|<k|w|\Leftrightarrow \dd(w,\frac{z}{1-k^2})<\frac{k|z|}{1-k^2},\\
		|z-w|=k|w|\Leftrightarrow \dd(w,\frac{z}{1-k^2})=\frac{k|z|}{1-k^2}.
	\end{align*}
\end{lemma}

\begin{lemma}
	Suppose  that $u=(u_{1},v_{2}w_{2})$ and
	$v=(u_{1}w_{1},v_{2})$ are two consecutive minimal vectors of index $2$ in a lattice $\LL$. Suppose that $w$ is a minimal vector such that  $v$ and $w$ are two consecutive minimal vectors
	and that $v$ and $w$ has index $2$ in $\Lambda$. Then
	$w=a(1+i)v+\alpha u$ where $a$ and $\alpha$ are Gaussian integers such that $0<|a|\leq
	2$ and $|\alpha|=1$.
\end{lemma}

\begin{proof}
	By Proposition \ref{prop:index}, $(U=\tfrac{1}{1+i}(u+v),V=v)$ is a basis of $\LL$ and  $(V=v,W=\tfrac{1}{1+i}(v+w))$ as well. Therefore, 
	\[
	\left\{
	\begin{array}
		[c]{l}%
		u=-V+(1+i)U\\
		v=V\\
		w=-V+(1+i)W\\
		W=bV-\alpha U
	\end{array}
	\right. 
	\]
	where  $b\in\mathbb{Z}[i]$ and $\alpha\in\mathbb{U}_{4}$ because the determinant of the coordinates of the vectors $V$ and $W$ in the basis $(U,V)$ is a unit of $\ZZ$. It follows that
	\begin{align*}
		w  &  =-v+(1+i)(bv-\alpha\tfrac{1}{1+i
		}(u+ v))\\
		&=(-1-\alpha+(1+i)b)v-\alpha u\\
		&  =a(1+i)v-\alpha u
	\end{align*}
	where $a\in\mathbb{Z}[i]$ because $-1-\alpha\in (1+i)\ZZ$. In coordinates this
	gives
	\[
	w=(u_{1}(a(1+i)w_{1}-\alpha),v_{2}(a(1+i)-\alpha w_{2}))
	\]
	and since $w$ follows $v$,
	\[
	|a(1+i)w_{1}-\alpha|<|w_{1}|.
	\]
	Therefore, $|w_{1}|>|a(1+i)||w_{1}|-|\alpha| $. Making use of Corollary \ref{cor:ConstraintsLarge} and of Proposition \ref{prop:sym}, we know that $w_{1}$ is in $\varphi(\mathcal{C}\setminus
	D(-i,\sqrt2))$ for some $\varphi\in\mathbb{D}_{8}$, hence $|w_{1}|>\frac
	{\sqrt3 -1}{\sqrt2}$. The two last inequalities imply that $|a|<\frac
	{1}{|1+i|}(1+\frac{1}{|w_{1}|})<\frac{1}{\sqrt2}+\frac{1}{\sqrt3-1}<\sqrt5$.
	It follows that $|a|\leq 2$. Finally, $a$ cannot be $0$ because $w$ and
	$u$ are not proportional.
\end{proof}

\begin{proof}
	[Proof of the proposition] We proceed by contradiction and suppose that $\mathbb{Z}%
	[i]v+\mathbb{Z}[i]w$ has index two. By the above lemma we have
	$w=(1+i)av-\alpha u$ where $a$ and $\alpha$ are Gaussian integers with
	$|\alpha|=1$ and $0<|a|\leq 2$. We have
	\[
	w=(u_{1}((1+i)aw_{1}-\alpha),v_{2}((1+i)a-\alpha w_{2})).
	\]

	Since the minimal vector $w$ follows $v$ we have $|(1+i)aw_{1}-\alpha
	|<|w_{1}|$ which is equivalent to
	\[
	|w_{1}-z|^{2}<|z|^{2}|w_{1}|^{2}
	\]
	where $z=c+id=\frac{\alpha}{(1+i)a}$. With the above lemma, we see that
	the latter inequality is equivalent to
	\[
	\operatorname{d}(w_{1},\frac{z}{1-|z|^{2}})<\frac{|z|^{2}}{1-|z|^{2}%
	}.
	\]

	Given $\beta\in\mathbb{U}_{4}$, consider the complex numbers $x=-\tfrac{1}{1+i}+(1+i)a$ and $y=-\tfrac{\beta}{1+i}+\alpha $; they are both in $J$. So that by Proposition \ref{prop:index}, the vector
	\[
	w^{\prime}=w-xv+yu=\tfrac{1}{1+i}(v-\beta u)=(w^{\prime}_{1},w^{\prime}%
	_{2})
	\]
	is in $\LL$.
	If we can choose $\beta$ in order that
	\[
	\left\{
	\begin{array}
		[c]{l}%
		|w^{\prime}_{1}|\leq |u_1w_{1}|\\
		|w^{\prime}_{2}|\leq |v_{2}((1+i)a-\alpha w_{2})|
	\end{array}
	\right.
	\]
	with one strict inequality at least, we have a contradiction with the fact that $v$ and  $w$ are two consecutive minimal vectors. The strategy is now to prove that either we can choose $\beta$ or
	that the inequality $\operatorname{d}(w_{1},\tfrac{z}{1-|z|^{2}})^{2}%
	<\tfrac{|z|^{4}}{(1-|z|^2)^{2}}$ does not hold.
	
	Using the symmetries and Proposition \ref{prop:sym} as in section \ref{sec:geonumber}, we can suppose that $w_{1}%
	\in\mathcal{C}$. By Corollary \ref{cor:ConstraintsLarge}, 
	\[
	w_{1}\in\mathcal{C}\setminus \D(-i,\sqrt2) \text{ and } w_{2}\in\mathbb{D}%
	\setminus(\D(1,\sqrt2)\cup \D(-i,\sqrt2)).
	\]
	With $t=\tfrac{1}{z}=\frac{(1+i)a}{\alpha}$, the above inequalities about
	$w^{\prime}_{1}$ and $w^{\prime}_{2}$ are equivalent to
	\[
	\left\{
	\begin{array}
		[c]{l}%
		\tfrac1{\sqrt2}|w_{1}-\beta|\leq|w_{1}|\\
		\tfrac1{\sqrt2}|w_{2}-\bar{\beta}|\leq|t- w_{2}|
	\end{array}
	\right.  .
	\]
	A short calculation shows that the latter  inequalities are equivalent to
	\[
	\left\{
	\begin{array}
		[c]{l}%
		|w_{1}+\beta|^{2}\geq 2\\
		|w_{2}-(2t-\bar{\beta})|^{2}\geq 2|t- \bar{\beta}|^{2}%
	\end{array}
	\right.  .
	\]
	Since $w_1\in\CC\setminus\D(-i,\sqrt 2)$, the first inequality holds when $\beta=1$ or $i$.
	
	Suppose first that $|a|=1$. We have $|t|^{2}=2$ hence $t=\pm1\pm i$.
	
	If $t=1+i$, choose $\beta=1$. We have $t-\bar{\beta}=i$ and $2t-\bar{\beta
	}=1+2i$, hence the second inequality is equivalent to $|w_{2}-(1+2i)|^{2}>2$
	which holds because $\Re w_{2}<0$ and $\Im w_{2}<1$.
	
	If $t=1-i$, choose $\beta=1$. We have $t-\bar{\beta}=-i$ and $2t-\bar{\beta
	}=1-2i$, hence the second inequality is equivalent to $|w_{2}-(1-2i)|^{2}>2$
	which holds because $\Im w_{2}>0$.
	
	If $t=-1+i$, then $z=-\tfrac{1+i}{2}$. Therefore, $\operatorname{d}%
	(w_{1},\tfrac{z}{1-|z|^{2}})^{2}>\tfrac{|z|^{4}}{(1-|z|^{2})^{2}}=1$ a contradiction.
	
	If $t=-1-i$, then $z=\tfrac{-1+i}{2}$. Therefore, $\operatorname{d}%
	(w_{1},\tfrac{z}{1-|z|^{2}})^{2}>\tfrac{|z|^{4}}{(1-|z|^{2})^{2}}=1$ a contradiction.
	
	Suppose that $|a|=\sqrt2$. We have $|t|=2$ hence $t=\pm2$ or $\pm2i$.
	Therefore, $\tfrac{z}{1-|z|^{2}}=\tfrac43 z$, and the information $w_{1}%
	\in\mathcal{C}\setminus D(-i,\sqrt2)$ implies that if $t=2i$ or $-2$ or $-2i$
	then  $\operatorname{d}(w_{1},\tfrac43 z)^{2}>\tfrac{|z|^{4}}{(1-|z|^{2})^{2}%
	}=\tfrac{1}{9}$. 
	
	If $t=2$, choose $\beta=1$. We have $t-\bar{\beta}=1$ and
	$2t-\bar{\beta}=3$, hence the second inequality becomes $|w_{2}-3|^{2}>2$
	which holds because $\Re w_{2}<0$.
	
	Suppose that $|a|=2$. We have $|t|=2\sqrt 2$ hence $t=\pm2(1 + i)$ or $\pm2i(1 + i)$.
	Therefore, $\tfrac{z}{1-|z|^{2}}=\tfrac87 z$, and the information $w_{1}%
	\in\mathcal{C}\setminus D(-i,\sqrt2)$ implies that if $t=2(1+i)$ or $-2(1+i)$ or $-2i(1+i)$ then $\operatorname{d}(w_{1},\tfrac87 z)^{2}>\tfrac{|z|^{4}}{(1-|z|^{2})^{2}%
	}=\tfrac{1}{49}$.
	
	If $t=2(1-i)$, choose $\beta=1$. We have $t-\bar{\beta}=1-2i$ and
	$2t-\bar{\beta}=3-4i$, hence the second inequality becomes $|w_{2}-(3-4i)|^{2}>2\times 5$
	which holds because $|\Re(w_{2}-(3-4i))|\geq 2$ and $|\Im(w_{2}-(3-4i))|\geq 3$.	
\end{proof}

\section{Definitions and parametrization of the Transversals}\label{sec:transversal}

\subsection{The open transversal}\label{subsec:opentrans}

Let $\mathbb{U}_{4}$ be the group of units in $\mathbb{Z}[i]$. The \textit{open transversal} $T$ is
the set of Gauss lattices $\Lambda$ in $\mathbb{C}^{2}$ such that
$\det_{\mathbb{C}}\Lambda\in\UU_4$ and such that there exist two vectors
$u=(u_{1},u_{2})$ and $v=(v_{1},v_{2})$ in $\Lambda$ such that 

\begin{enumerate}
	\item $|u_{2}|,|v_{1}|<|u_{1}|=|v_{2}|=r$,
	
	\item the only nonzero vectors of $\Lambda$ in the ball $B_{\infty}(0,r)$ are in
	$\UU_{4} u\cup\UU_{4} v$.
\end{enumerate}
Observe that the two vectors $u$ and $v$ are  minimal vectors in $\LL$ and that by Lemma \ref{lem:consecutive} they are consecutive.
The vectors $u$ and $v$ are the vectors \textit{associated} with $\Lambda$.
They are unique up to multiplicative factors in $\mathbb{U}_{4}$:
\begin{lemma}
	\label{lem:transversal1} Let $\Lambda$ be a lattice in $T$ and $u,v$ vectors
	in $\Lambda$ satisfying (1) and (2) in the above definition. If $u^{\prime}$
	and $v^{\prime}$ are two vectors in $\Lambda$ such that (1) and (2) hold then
	$u^{\prime}\in\mathbb{U}_{4} u$ and $v^{\prime}\in\mathbb{U}_{4} v$.
\end{lemma}

\begin{proof}
	Set $r=|u|_{\infty}$ and $r^{\prime}=|u^{\prime}|_{\infty}$. The balls
	$B_{\infty}(0,r)$ and $B_{\infty}(0,r^{\prime})$ are nested, therefore by (2) they are
	equal. So that by (1), $|u'_1|=|u_1|=|v'_2|=|v_2|$. Hence  $u'$ and  $v'\in B_{\infty}(0,r)$. Again by (2), this imply that $u^{\prime}\in\mathbb{U}_{4} u$ and $v^{\prime}%
	\in\mathbb{U}_{4} v$.
\end{proof}

\subsection{The full transversal}\label{subsec:fulltrans}
The \textit{full transversal} $T'$ is
the set of Gauss lattices $\Lambda$ in $\mathbb{C}^{2}$ such that
$\det_{\mathbb{C}}\Lambda\in\UU_4$ and such that there exist two minimal vectors
$u=(u_{1},u_{2})$ and $v=(v_{1},v_{2})$ in $\Lambda$ such that 

\begin{enumerate}
	\item[(1')] $|u_{2}|,|v_{1}|<|u_{1}|=|v_{2}|=r$,
	
	\item[(2')] the only nonzero vector of $\Lambda$ in the open ball $\overset{o}B_{\infty}(0,r)$ is $0$.
\end{enumerate}
\medskip
Clearly 
\[
T\subset T'\subset \left\{\LL\in\SL(2,\C)/\SL(2,\ZZ):\lambda_1(\LL,|.|_{\infty},\C)=\lambda_2(\LL,|.|_{\infty},\C)\right\} .
\]

The vectors $u$ and $v$ are the vectors \textit{associated} with $\Lambda$.
They are no longer  unique up to multiplicative factors in $\mathbb{U}_{4}$, see the example subsection \ref{subsec:example2}. By Lemma \ref{lem:consecutive}, they are consecutive.

By Proposition
\ref{prop:index}, the lattice $L=\mathbb{Z}[i]u+\mathbb{Z}[i]v$ has index $1$
or $2$ in $\Lambda$. Therefore, the transversal $T$ (resp. T') is the union of two pieces
$T_{1}$ and $T_{2}$ (resp. $T'_1$ and $T'_2$) according to the index of $L$. The above lemma implies
that $T_{1}$ and $T_{2}$ are disjoint but as the example of subsection \ref{subsec:example2} shows
\[
T'_1\cap T'_2\neq \emptyset.
\]
However, $T'_1\cap T'_2$ is a small set. It is a consequence of the following Lemma.\\

Let $\mathcal N$ be the set of unimodular lattices $\LL\subset \C^2$ such that either there exists a nonzero vector $(u_1,u_2)\in \LL$ with $u_1u_2=0$ or there exist two linearly independent vectors $u=(u_1,u_2)$ and $v=(v_1,v_2)$ in $\LL$ such that $|u_1|=|v_1|$ or $|u_2|=|v_2|$.

\begin{lemma}\label{lem:negligible} The following properties hold 
	\begin{enumerate}
		\item
		$\mathcal N$ contains the set
		\[
		\left\{\LL\in\SL(2,\C)/\SL(2,\ZZ):\lambda_1(\LL,|.|_{\infty},\C)=\lambda_2(\LL,|.|_{\infty},\C)\right\}\setminus T.
		\]
		\item $\mathcal N$ is stable under the action of the flow $g_t$, $t\in\R$.
		\item 
		$\mathcal N$ has zero Haar measure.
	\end{enumerate}
\end{lemma}
\begin{remark}  $T'\setminus T\subset \mathcal N$.
\end{remark}
\begin{proof}
	1. Let $\LL$ be a unimodular lattice not in the open transersal $T$ such that $r=\lambda_1(\LL,|.|_{\infty},\C)=\lambda_2(\LL,|.|_{\infty},\C)$. The equality of the two minima implies that there exist two linearly independent vectors $u=(u_1,u_2)$ and $v=(v_1,v_2)$ such that $|u|_{\infty}=|v|_{\infty}=r$ and $|u_1|\geq |v_1|$. 
	If $\LL$ were not in $\mathcal N$, we would have $r\geq|u_1|>|v_1|$ and therefore $r=|v_2|$. Since $|v_2|\neq |u_2|$, we would have $|v_2|> |u_2|$. Since $\LL$ is not in $T$, this implies that there exists a nonzero vector $w=(w_1,w_2)\in\LL\cap B_{\infty}(0,r)$ not in $\UU_{4} u\cup\UU_{4}v$. For this vector we have either $|w_1|=r=|u_1|$ or $|w_2|=r=|v_2|$, a contradiction.
	
	2. Clear.
	
	3.	It suffices to prove that the set $\mathcal M$ of $ 2 \times 2 $ matrices $ M $, with inputs in $ \C $ such that either there exists $X\in\ZZ^*$ such that the product of the coordinates of $MX$ is zero or there are two linearly independent vectors $ X $ and $ Y $ in $ \ZZ $ such that either the moduli of the first coordinates of $ MX $ and $ MY $ are equal, or the moduli of the second coordinates of $ MX $ and $ MY $ are equal. By definition, the set  $\mathcal M$ is the union of two sets $\mathcal M_1$ and $\mathcal M_2$. The first set is countable union of hyperplanes in $M_2(\C)$ and thus is of zero Lebesgue.  Let us deal now with $\mathcal M_2$.
	
	Since $\ZZ$ is countable, considering each row of the matrix $M$, we are reduced to prove that given two linearly independent vectors $x=(x_1,x_2)$ and $y=(y_1,y_2)\in\C^2$ the set of $(a,b)\in\C^2$ such that   $P(a,b)=|ax_1+bx_2|^2-|ay_1+by_2|^2=0$, is of zero Lebesgue measure in $\C^2$.  
	
	Now  $P$ can be considered as a real polynomial of four variable. Since $P(y_2,-y_1)=|\det_{\C}(x,y)|\neq 0$, the polynomial $P$ is not zero and the set of $(a,b)$ such that $P(a,b)=0$ has measure zero.
\end{proof}

\subsection{Properties of the open transversal}

\begin{lemma}
	\begin{itemize}
		\item The open transversal $T$ is a real submanifold of $\operatorname{SL}(2,\mathbb{C}%
		)/\operatorname{SL}(2,\mathbb{Z}[i])$.
		\item The flow $(g_{t})_{t\in\mathbb{R}}$ is transverse to $T$.
	\end{itemize}
\end{lemma}
\begin{proof}
	Let $\Lambda_{0}$ be in $T$ and let $u_{0}$ and $v_{0}$
	be the two vectors associated with $\LL_0$ by the definition of $T$. By Proposition \ref{prop:index} either $(u_0,v_0)$ form a basis of $\LL_0$ and we can suppose $\det(u_0,v_0)=1$ w.l.o.g. (case of index $1$) or $U_0=u_0$ and $V_0=\tfrac{1}{1+i}(u_0+v_0)$ form a basis of $\LL_0$ and we can suppose $\det(u_0,v_0)=(1+i)$ w.l.o.g. (case of index $2$).
	We can find a small enough positive real number $\varepsilon$ such that for
	any $(u,v)$ in the open set
	\[
	W=B_{\mathbb{C}^{2}}(u_0,\varepsilon)\times B_{\mathbb{C}^{2}%
	}(v_0,\varepsilon),
	\]
	\begin{itemize}
		\item the matrix $M=M(u,v)$ the columns of which are $u=(u_1,u_2)$ and $v=(v_1,v_2)$
		is in $\GL (2,\mathbb C)$ and the sets $WP$, $P\in\SL (2,\mathbb  Z[i])$ are disjoint,
		\item the vectors in $\UU_4 u$ and $\UU_4v$ are the only nonzero vectors of
		the lattice $\Lambda=M\mathbb Z[i]^2$ in the cylinder $C(u,v)$ in case index $1$, or of the lattice $\Z[i]u+\Z[i]\frac{1}{1+i}(u+v)$ in case of index $2$,
		\item $|u_1|>|u_2|$ and $|v_1|<|v_2|$,
		\item for all $M\in W$, $|\det M-\det(u_0,v_0)|\leq \tfrac{1}{10}$. 
		
	\end{itemize}
	Consider the map%
	\begin{align*}
		\phi  &  :W\rightarrow\C\times \R\\
		&  :M=(u,v)\rightarrow(\phi_{1}(M)=\frac{\det M}{\det(u_0,v_0)},\phi_{2}(M)=\left\vert
		u_1\right\vert^{2}-\left\vert v_2\right\vert^{2}).
	\end{align*}
	In case of index $1$, a lattice $\Lambda=M\mathbb{Z}[i]^{2}$ with $M\in W$, is  in $T$ iff
	$\phi(M)=(1,0)$. In case of index $2$, a lattice $\LL=\Z[i]u+\Z[i]\frac{1}{1+i}(u+v)$ is in $T$ iff $\phi(M)=(1,0)$. Hence, to prove that $T$ is a submanifold, it is enough to show that the differential
	$D\phi(M)$ is onto at every point $M$ in $W$. The
	differential of $\phi_{1}$ is $\C$-linear and is given by
	\[
	D\phi_{1}(M).(x,y)=\frac{1}{\det(u_0,v_0)}(x_1v_2+u_1y_2-x_2v_1-u_2y_1)
	\]
	where $x=(x_1,x_2)$ and $y=(y_1,y_2)$. 
	The differential of $\phi_{2}$ is 
	given by
	\[
	D\phi_{2}(M).(x,y)=u_1\bar x_1+\bar u_1x_1-v_2\bar y_2-\bar v_2y_2.
	\]
	Call $\gamma_M$ the $\C$-linear map defined by $\gamma_M(x,y)=\bar u_1x_1-\bar v_2y_2$. On the one hand, $D\phi_2(M)=\gamma_M+\bar {\gamma}_M$. On the other hand, $\gamma_M$ and $D\phi_2(M)$ are $\C$-linearly independent because $u_1\neq 0$ (or $v_2\neq 0$). Therefore, the three $\R$-linear maps $\gamma_M$, $\Re D\phi_1(M)$ and $\Im D\phi_1(M)$ are $\R$-linearly independent. It follows that $D\phi(M)$ is onto 
	which implies that $T$ is a submanifold of $\Omega_1$.
	
	To show that the flow is transverse to $T$, we have to check that for a matrix
	$M=M(u,v)$ in $W$ such that $\phi(M)=(1,0)$ we have $D\phi(M).((u_1,-u_2),(v_1,-v_2))\neq 0$. Now,  $D\phi_{2}%
	(M).((u_1,-u_2),(v_1,-v_2))=2|u_1|^2+2|v_2|^2$, hence $D\phi(M).((u_1,-u_2),(v_1,-v_2))$ is not zero.
\end{proof}

\section{First return map, proof of Theorem \ref{thm:continuedfraction}}

\begin{lemma}\label{lem:consecutivetransversal}
	Let $\LL$ be a unimodular lattice that is in the full transversal $T'$ and let $u=(u_1,u_2)$ and $v=(v_1,v_2)$ be two consecutive minimal vectors associated with $\LL$. Let $t=\inf\{s>0:g_s(\LL)\in T'\}$. Then $t<+\infty$ if and only if $v_1\neq 0$. Moreover, in the latter case, there exists a minimal vector $z=(z_1,z_2)$ such that $v$ and $z$ are two consecutive minimal vectors and
	\[
	t=\frac12 \ln \frac{|z_2|}{|v_1|}.
	\]
	Consequently the first return map applied to $\LL$ is $g_t\LL$.
\end{lemma}
\begin{proof}
	If $v_1\neq 0$ then by Lemma \ref{lem:lexico} a  minimal element $z=(z_1,z_2)\in\LL$ for the lexicographic preoder $\prec$ in the cylinder $\overset{o}C_1(v)$ is a minimal vector. By definition  $v$ and $z$ are consecutive. So, by Lemma \ref{lem:consecutive}, $\overset{o}C(v,z)\cap\LL\setminus\{0\} =\emptyset$, hence $\overset{o}C(g_tv,g_t z)\cap g_t(\LL\setminus\{0\}) =\emptyset$, thus $g_t\LL$ is in $T'$ when $t=\frac12 \ln \frac{|z_2|}{|v_1|}$. Let $0<s<t$. We want to show that $g_s\LL\notin T'$. 
	\begin{itemize}
		\item If $r_2=e^{-s}|v_2|\leq r_1=e^s|v_1|$, since $r_1=e^s|v_1|<e^t|v_1|=e^{-t}|z_2|<e^{-s}|z_2|$, we have  $B_{\infty}(0,r_1=|g_sv|_{\infty})\subset C(g_sv, g_sz)$. Now, there is no vector in $g_s(\LL\setminus\{0\}) $ of the shape $(x_1,x_2)$ with $|x_1|<r_1$ and $|x_2|=r_1<e^{-s}|z_2|$ because such vector would be in $\overset{o}C(g_sv,g_s z)=g_{s-t}\overset{o}C(g_tv,g_t z)$. Therefore, $g_s\LL$ is not in the full transversal $T'$. 
		\item If $r_2=e^{-s}|v_2|> r_1=e^s|v_1|$, since $r_2=e^{-s}|v_2|=e^{-s}|u_1|<e^s|u_1|$, we have $B_{\infty}(0,r_2=|g_sv|_{\infty})\subset C(g_su, g_sv)$ and there is no vector in $g_s(\LL\setminus\{0\})$ of the shape $(x_1,x_2)$ with $|x_1|=r_2<e^s|u_1|$ and $|x_2|<r_2$. Therefore, $g_s\LL$ is not in the full transversal $T'$.
	\end{itemize}	
\end{proof}

Let $u=(u_1,v_2w_2)$ and $v=(u_1w_1,v_2)$ be the two consecutive minimal  vectors associated with a lattice $\LL$ that is in the full transversal $T'$.  By the above lemma, the computation of the first return map is reduced to the computation of the minimal vector $ v'\in\LL $ such that  $v$ and $v'$ are two consecutive minimal vectors. This is the purpose of Theorem \ref{thm:continuedfraction} that we recall below. To perform this calculation, we must take into account the component of the transversal which contains $\LL $.

Recall that the lexicographic preoder on $\C^2$  is defined by
\[
(x_1,x_2)\prec (y_1,y_2) 
\]
iff $|x_2|<|y_2|$ or $|x_2|=|y_2|$ and $|x_1|\leq |y_1|$.
We recall the statement of Theorem \ref{thm:continuedfraction}  for the convenience of the reader.\\

\noindent
{\bf Theorem \ref{thm:continuedfraction}.}
{\sl
	Let $u=(u_1,v_2w_2)$ and $v=(u_1w_1,v_2)$ be the two minimal consecutive vectors associated with a lattice $\LL$ that is in the full transversal $T'$. If $w_1\neq 0$ then there exists  $v'\in\LL$  a minimal vector such that $v$ and $v'$ are two consecutive minimal vectors and
	\begin{itemize}
		\item if $\det_{\C}(u,v)=1$, then $v'$ is  any vectors that is minimal for the preoder $\prec$ in the set
		\begin{align*}
			\left\{z=-au+gv: a\in\{1,1+i\},\, g\in\Z[i],\,|\tfrac{a}{w_1}-g|<1\right\}.
		\end{align*} 
		Moreover with $u'=v=(u'_1,v'_2w'_2)$ and $v'=-au+gv=(u'_1w'_1,v'_2)$ we have 
		\begin{align}
			w'_1=g-\frac{a}{w_1}, \hspace{1cm} w'_2=\frac{1}{g-aw_2}.	
		\end{align}
		\item If $\det_{\C}(u,v)=1+i$, then $v'$ is  any vectors that is minimal in the set
		\begin{align*}
			\left\{z=-\tfrac{1}{1+i}(u+v)+gv:  g\in\Z[i],\,|\tfrac1{(1+i)w_1}+\tfrac1{(1+i)}-g|<1\right\}.
		\end{align*}  
		Moreover with $u'=v=(u'_1,v'_2w'_2)$ and $v'=-au+gv=(u'_1w'_1,v'_2)$ we have 
		\begin{align}
			w'_1=g-\tfrac1{(1+i)w_1}-\tfrac1{(1+i)}, \hspace{1cm} w'_2=\frac{1}{g-\tfrac1{(1+i)}w_2-\tfrac1{(1+i)}}.	
		\end{align} 
	\end{itemize}	
}

\begin{proof}
	If $w_1\neq 0$ then $v_1=u_1w_1\neq 0$ and by Minkowski convex body theorem, the cylinder $\overset{o}C_1(|v_1|)\{(x,y)\in\C^2:|x|<|v_1|\}$ contains at least one nonzero vector of $\LL$. A vector of $\LL$ in this cylinder which is  minimal for the preorder $\prec$ is a minimal vector $v'$ that follows $v$.
	
	{\bf Case 1:} 
	$\det_{\C}(u,v)=1$. 
	Let  $L=\Z[i]v+\Z[i]v'$ be the lattice generated by $v$ and $v'$. Since $L$ has index $1$ or $2$ in $\LL$, the determinant of $v$ and $v'$ in the basis $u,v$ is a unit of $\Z[i]$ or $(1+i)$ times a unit. This implies that $v'=-au+gv$ with $g\in\Z[i]$ and $a\in\UU_4$ or $a\in(1+i)\UU_4$.  We can suppose that  $a\in \{1,1+i\} $ w.l.o.g.. 
	The condition $v'\in \overset{o}C_1(|v_1|)$ is equivalent to,
	\[
	|au_1-gu_1w_1|<|u_1w_1|
	\] 
	which in turn is equivalent to 
	\[
	|\frac{a}{w_1}-g|<1.
	\] 
	By definition $v'$ is minimal for the preorder $\prec$ among the  vectors $-au+gv$ such that the latter inequality holds. An easy calculation leads to the formula for $w'_1$ and $w'_2$. 
	
	{\bf Case 2:} $\det_{\C}(u,v)=1+i$.
	By Proposition \ref{prop:index2}, $\Z[i]v+\Z[i]v'$ has index one in $\LL$ and  by Proposition \ref{prop:index}, $\frac1{1+i}(u+v)$, $v$  is a basis of $\LL$, therefore, $z'$ is of the shape
	\[
	v'=a\tfrac{1}{1+i}(u+v)+gv
	\]
	with $g\in\Z[i]$ and $a\in\UU_4$. We can suppose that $a=-1$ w.l.o.g.
	As before, we have
	\[
	|\tfrac1{1+i}(u_1+u_1w_1)-gu_1w_1|<|u_1w_1|,
	\]
	hence,
	\[
	|\tfrac1{(1+i)w_1}+\tfrac1{(1+i)}-g|<1.
	\]
	We conclude as in the first case.
\end{proof}

\section{Parametrization of the open transversal and the first return map}
We first give a parametrization of the open transversal with coordinates $(\ttt,w_1,w_2)\in\R\times\D^2$.
Next we want to describe the open transversal with the $(\ttt,w_1,w_2)$ coordinates. For this purpose, we first write the symmetries of the transversal with the coordinates. Finally we give explicit formulas for the  first return map as a function of the coordinates $(\ttt,w_1,w_2)$.

\subsection{Parametrization of the open transversal $T$}\label{subsec: param}

\begin{proposition}
	\label{prop:para} Let $\Psi_{k}:\mathbb{R}\times\mathbb{D}^{2}\rightarrow
	\Omega_{1}$, $k=1,2$ be the maps defined by
	\begin{align*}
		&  \Psi_{1}(\theta,w_{1},w_{2})=\mathbb{Z}[i]u+\mathbb{Z}[i]v\\
		&  \Psi_{2}(\theta,w_{1},w_{2}) =\mathbb{Z}[i]u+\tfrac{1}{1+i}\mathbb{Z}%
		[i](u+v)
	\end{align*}
	where
	\begin{align*}
		&  u=u(\theta,w_{1},w_{2})=r(u_{1},v_{2}w_{2}),\\
		&  v=v(\theta,w_{1},w_{2})=r(u_{1}w_{1},v_{2}),\\
		&  r=\frac{k^{1/4}}{\sqrt{|1-w_{1}w_{2}|}},\\
		&  u_{1}=\exp i\theta,\\
		&  v_{2}=\exp i\theta^{\prime}=\exp i((k-1)\tfrac{\pi}{4}-\theta-\arg
		(1-w_{1}w_{2})).
	\end{align*}
	For $k=1,2$, let $C_k(\theta,w_1,w_2)=C(u(\theta,w_1,w_2),v(\theta,w_1,w_2)$.)
	Then for all $\Lambda$ in $T_{k}$ there exists exactly one element
	$(\theta,w_{1},w_{2})\in[0,\tfrac{\pi}{2}[\times\mathbb{D}^{2}$ such that
	$\LL=\Psi_k(\theta,w_1,w_2)$ and  $\LL\cap\overset{o}{C_k}(\theta,w_1,w_2)=\{0\}$.
\end{proposition}

\begin{proof}
	{\sc Existence.} Let $\Lambda$ be a unimodular lattice in $\mathbb{C}^{2}$ that belongs to
	$T_{k}$ and call $u$ and $v$ the two minimal vectors associated with $\Lambda
	$. Denoting $r=|u|_{\infty}$, $u$ and $v$ can be written $u=r(u_{1}%
	,v_{2}w_{2})$ and $v=r(u_{1}w_{1},v_{2})$ with
	\[
	|w_{1}|,|w_{2}|<1=|u_{1}|=|v_{2}|.
	\]
	The unimodularity implies that
	\[
	\operatorname{\det}_{\mathbb{C}}(u,v)=r^{2}u_{1}v_{2}(1-w_{1}w_{2}%
	)\in\mathbb{U}_{4} \text{ or } \in(1+i)\mathbb{U}_{4}
	\]
	according to $\Lambda\in T_{1}$ or $T_{2}$. On the one hand this implies that
	$r=\tfrac{ k^{1/4}}{\sqrt{|1-w_{1}w_{2}|}}$. On the other hand, since $u$ and
	$v$ can be changed in $\omega u$ and $\omega^{\prime}v$ with $\omega
	,\omega^{\prime}\in\mathbb{U}_{4}$, we can impose  $u_{1}=\exp i\theta$
	with $\theta\in[0,\tfrac{\pi}{2}[$ and $v_{2}=\exp i\theta^{\prime}$ where
	\[
	\theta^{\prime}=(k-1)\tfrac{\pi}{4}-\theta-\arg(1-w_{1}w_{2}).
	\]
	With our choices, the triple $(\theta,w_1,w_2)$ belongs to $U=[0,\tfrac{\pi}{2}[\times\mathbb{D}^{2}$. 
	
	If $k=1$, we have $\LL=\Psi_1(\theta,w_1,w_2)\in \Psi_1(U)$. If $k=2$, thanks to Proposition \ref{prop:index} Part 3, we have $\LL=\Psi_2(\theta,w_1,w_2)\in \Psi_2(U)$.

	{\sc Uniqueness.} Let $\LL$ be in $T_k$. Suppose that $\LL=\Psi_k(\alpha,w_1,w_2)=\Psi_k(\alpha',w'_1,w'_2)$ with $(\alpha,w_1,w_2)$ and $(\alpha',w'_1,w'_2)\in U$. Let $u=u(\alpha,w_{1},w_{2})$, $v=v(\alpha,w_{1},w_{2})$, $u^{\prime}=u(\alpha^{\prime
	},w^{\prime}_{1},w^{\prime}_{2})$ and $v^{\prime}=v(\alpha^{\prime},w^{\prime
	}_{1},w^{\prime}_{2})$. The two cylinders $\overset{o}{C_k}(\alpha,w_1,w_2)$ and $\overset{o}{C_k}(\alpha,w_1,w_2)$ must be equal, otherwise one of them would contains nonzero elements of $\LL$. Hence $r=r'$ and by Lemma \ref{lem:transversal1}, $u'\in\UU_{4}u$. 
	Since $\alpha$ and $\alpha^{\prime}$ are both in
	$[0,\tfrac{\pi}{2}[$, it follows that $u=u'$ and $\alpha=\alpha'$. Hence
	$v_{2}w_{2}=v^{\prime}_{2}w^{\prime}_{2}$. Again $v^{\prime}=\omega v$ with
	$\omega\in\mathbb{U}_{4}$, therefore $w^{\prime}_{1}u_{1}=\omega w_{1}u_{1}$
	and $v^{\prime}_{2}=\omega v_{2}$. It follows that $w^{\prime}_{1}=\omega
	w_{1}$ and $w^{\prime}_{2}\omega=w_{2}$ which in turn imply $w^{\prime}%
	_{1}w^{\prime}_{2}=w_{1}w_{2}$. Now by definition of $\Psi_{k}$,
	$\det_{\mathbb{C}}(u,v)=\det_{\mathbb{C}}(u^{\prime},v^{\prime})=1$ or $1+i$, hence
	$u_{1}v_{2}(1-w_{1}w_{2})=u^{\prime}_{1}v^{\prime}_{2}(1-w^{\prime}%
	_{1}w^{\prime}_{2})$ and taking into account the relations $u^{\prime}_{1}
	=u_{1}$ and $w^{\prime}_{1}w^{\prime}_{2}=w_{1}w_{2}$ we obtain $v_{2}
	=v^{\prime}_{2}$. Finally, this implies, $\omega=1$ and $(\alpha^{\prime
	},w^{\prime}_{1},w^{\prime}_{2})=(\alpha,w_{1},w_{2})$.
\end{proof}

\subsection{Symmetries of the transversal}

Given $(\theta,w_{1},w_{2})\in \R\times \D^2$, we would like to know what are the
condition in order that $\Psi_{1}(\theta,w_{1},w_{2})\in T_{1}$ and $\Psi
_{2}(\theta,w_{1},w_{2})\in T_{2}$. Thanks to Theorem \ref{thm:geonumber3} and to the distance formula (Lemma \ref{lem:formule}), these conditions are given by a finite set
of inequalities on $w_1$ and $w_2$ and do not depend on $\ttt$. 
As for Theorem \ref{thm:geonumber3}, the use of transversal symmetry properties will simplify the statement. In fact, it can be reduced to the case
\[
w_{1}\in\mathcal{C}=\{z\in\mathbb{D}: \arg z\in[0,\tfrac{\pi}4 ]\}
\]
and the other cases, $w_{1}\in\{z\in\mathbb{D}:\arg z\in[\tfrac{k\pi}4
,\tfrac{(k+1)\pi}4 ]\}$, $k=1,\dots,7$, will be obtained through simple transformations.

Let $T_{1}^{0}$ and $T_{2}^{0}$ be the subset of $T_{1}$ and $T_{2}$ by
\begin{align*}
	T_{1}^{0}=T_{1}\cap\{\LL=\Psi_{1}(\theta,w_1,w_2):(\theta,w_1,w_2)\in[0,\tfrac{\pi}{2}[\times\CC\times
	\DD),\, \LL\cap\overset{o} {C}_1(\theta,w_1,w_2)=\{0\}\},\\
	T_{2}^{0}=T_{2}\cap\{\LL=\Psi_{2}(\theta,w_1,w_2):(\theta,w_1,w_2)\in[0,\tfrac{\pi}{2}[\times\CC\times
	\DD),\, \LL\cap\overset{o}{ C}_1(\theta,w_1,w_2)=\{0\}\}.
\end{align*}
Recall that $\mathbb{D}%
_{8}$ is the group of isometries acting on $\mathbb{C}$ generated by the
multiplications by elements in $\mathbb{U}_{4}$ and by the conjugation.
For $\varphi\in\mathbb{D}_{8}$ consider the map $F_{k,\varphi
}:T_{k}\rightarrow\Omega_{1}$ defined by
\[
F_{k,\varphi}(\Psi_{k}(\theta,w_{1},w_{2}))=\Psi_{k}(\theta,\varphi(w_{1}%
),\tfrac{1}{\varphi(1)^{2}}\varphi(w_{2})).
\]

This map is well defined because by Proposition \ref{prop:para}, for each $\LL\in T_k$, there exists $(\ttt,w_1,w_2)$ unique in $[0,\tfrac{\pi}{2}[\times \D^2$ such that $\psi_k(\ttt,w_1,w_2)=\LL$ and $\LL\cap \overset{o}{C}_k(\theta,w_1,w_2)=\{0\}$.

Our aim is to prove:

\begin{proposition}\label{prop:sym2}
	For $k=1,2$,
	\[
	T_{k}=\bigcup_{\varphi\in\mathbb{D}_{8}}F_{k,\varphi}(T_{k}^{0}).
	\]
	
\end{proposition}

Since 
\[
\D=\cup_{\varphi\in\D_8}\varphi(\CC),
\]
the proposition is an obvious consequence of the following lemma.

\begin{lemma}
	\label{lem:sym} For $k=1,2$ and $\varphi\in\mathbb{D}_{8}$
	\[
	F_{k,\varphi}(T_{k})=T_{k}.
	\]
	
\end{lemma}

\begin{proof}
	[Proof of Lemma \ref{lem:sym}]

	1. It is enough to prove that $F_{k,\varphi
	}(T_{k})\subset T_{k}$ for all $\varphi\in\mathbb{D}_{8}$. Indeed, if so, we
	have $F_{k,\varphi^{-1}}(F_{k,\varphi}(T_{k}))\subset F_{k,\varphi^{-1}}%
	(T_{k})\subset T_{k}$ and since the elements in $\mathbb{D}_{8}$ are
	$\mathbb{R}$-linear, for $\Lambda=\Psi_{k}(\theta,w_{1},w_{2})\in T_{k}$, we
	have
	\begin{align*}
		F_{k,\varphi^{-1}}(F_{k,\varphi}(\Psi_{k}(\theta,w_{1},w_{2})))  &
		=F_{k,\varphi^{-1}}(\Psi_{k}(\theta,\varphi(w_{1}),\tfrac{1}{\varphi(1)^{2}%
		}\varphi(w_{2}))\\
		&  =\Psi_{k}(\theta,\varphi^{-1}(\varphi(w_{1})),\tfrac{1}{\varphi^{-1}%
			(1)^{2}}\varphi^{-1}(\tfrac{1}{\varphi(1)^{2}}\varphi(w_{2})))\\
		&  =\Psi_{k}(\theta,w_{1},\tfrac{1}{\varphi^{-1}(1)^{2}}\tfrac{1}%
		{\varphi(1)^{2}}w_{2})\\
		&  =\Psi_{k}(\theta,w_{1},w_{2}),
	\end{align*}
	which implies that $F_{k,\varphi^{-1}}(F_{k,\varphi}(T_{k}))=T_{k}$.\medskip
	
	2. Call $E_{1}=(\mathbb{Z}[i]\setminus\{0\})^{2}$ and $E_{2}=(\mathbb{Z}[i]\setminus\{0\})^2\cup J^{2}$. For
	each $\varphi,\psi\in\mathbb{D}_{8}$, the maps $f(a,b)=(\psi(a),\varphi(b))$
	induces a bijection of $E_{k}$ in itself. It is an immediate consequence of  $\varphi
	(\mathbb{Z}[i])=\psi(\mathbb{Z}[i])=\mathbb{Z}[i]$ and $\varphi(J)=\psi(J)=J$.
	\medskip
	
	3. Let  $\ttt\in[0,\tfrac{\pi}{2}[$, $w_1,w_2\in \D$, $w'_1=\varphi(w_1)$, and $w'_2=\tfrac{1}{\varphi(1)^2}\varphi(w_2)$ 
	\begin{align*}
		&u=u(\ttt,w_1,w_2)=(u_1,v_2w_2),\ v=v(\ttt,w_1,w_2)=(u_1w_2,v_2)\\	
		&u'=u(\ttt,w'_1,w'_2)=(u'_1,v'_2w'_2),\ v'=v(\ttt,w'_1,w'_2)=(u'_1w'_2,v'_2).
	\end{align*}
	Suppose that $\LL=\ZZ u+\ZZ v\in T_k$ and $\LL\cap C_k(\theta,w_1,w_2)=\{0\}$. Consider $\LL'=\ZZ u'+\ZZ v'$. By definition $\LL'=F_{k,\varphi}(\LL)$. We want to show that $\LL'\in T_k$.
	By Proposition \ref{prop:sym} about symmetries, for all nonzero complex numbers $a,b$, 
	\begin{align*}
		|au-bv|_{u,v}  &  =|\varphi(1)\varphi(a)u^{\prime}-\varphi(b)v^{\prime}|_{u',v'}.
	\end{align*}
	By 2, it follows that $|a^{\prime}u'-b^{\prime}v'|_{u',v'}>1$ for all
	$(a^{\prime},b^{\prime})\in E_{k} $ iff
	$|au-bv|_{u,v}>1$ for all $(a,b)\in E_{k} $. Since
	$\Lambda$ is in $T_{k}$, $|au-bv|_{u',v'}>1$ for all $(a,b)\in E_{k}
	$ which implies that $|a^{\prime}u'-b^{\prime
	}v'|_{u',v'}>1$ for all $(a^{\prime},b^{\prime})\in E_{k}
	$ and we are done.
\end{proof}

\subsection{Determination of the open transversals $T_{1}$ and $T_2$ in the $(\ttt,w_1,w_2)$-coordinates 
}\label{subsec:detopentrans}

Recall that
\begin{align*}
	\mathcal{C}=  &  \{w\in\mathbb{C}: |w|<1,\,\arg w\in[0,\tfrac{\pi}4 ]\},\\
	\mathcal{D}=  &  \{w\in\mathbb{C}:|w|<1,\, \operatorname{d}(w,1)>
	1,\,\operatorname{d}(w,1-i)> 1 \},\\
	\mathcal{T}=  &  \{w\in\mathbb{C}:|w|<1,\, \operatorname{d}(w%
	,1)>\sqrt2,\,\operatorname{d}(w,-i)> \sqrt2 \}
\end{align*}
and that the parametrizations $\Psi_1$ and $\Psi_2$ have been defined Proposition \ref{prop:para}.

Consider the following pairs of closed disks in $\mathbb{C}$: 
\begin{align*}
	&  \bar Red_{1}=D(i,1),\,\bar Red_{2}=D(-i,1)\\
	&  \bar Blue_{1}=D(\tfrac{1-i}{2},\tfrac1{\sqrt2}),\,\bar Blue_{2}=D(1+i,1)\\
	&  \bar  Green_{1}=D(1+i,1),\, \bar Green_{2}=D(\tfrac{1-i}{2},\tfrac1{\sqrt2}).
\end{align*}
see the Figure 3 in subsection \ref{subsec:constraints}.

Let us define $W_1^0$ to be the set of $(w_1,w_2)$ in $\CC\times\DD$ such that  one of the four conditions
\begin{enumerate}
	\item $w_{2}\in \bar Green_{2}$ and $w_{1}\notin \bar Red_{1}\cup \bar Green_{1}$,
	
	\item $w_{2}\in \bar Red_{2}\setminus \bar Green_{2}$ and $w_{1}\notin \bar Red_{1}$,
	
	\item $w_{2}\notin \bar Red_{2}\cup \bar Blue_{2}$ and $w_1\neq 0$,
	
	\item $w_{2}\in \bar Blue_{2}$ and $w_{1}\notin \bar Blue_{1}$,
\end{enumerate}
holds, and let
\[
W_1=\{(\varphi(w_1),\tfrac{1}{\varphi(1)^2}\varphi(w_{2})):\varphi\in\D_8,\,(w_1,w_2)\in W_1^0\}.
\]
Let us define 
\[
W_2^0=\mathcal{C}\setminus D(-i,\sqrt2)\times \mathcal T
\]
and let
\[
W_2=\{(\varphi(w_1),\tfrac{1}{\varphi(1)^2}\varphi(w_{2})):\varphi\in\D_8,\,(w_1,w_2)\in W_2^0\}.
\]

\begin{theorem}
	\label{thm:T} Let $(\ttt,w_1,w_2)$ be in $[0,\tfrac{\pi}{2}[\times \D^2$. Then	
	\begin{itemize}
		\item $\Psi_1(\ttt,w_1,w_2)\in T_{1}$ iff $(w_1,w_2)\in W_1$, 
		
		\item $\Psi_2(\ttt,w_1,w_2)\in T_{2}$ iff   $(w_{1},w_{2})\in W_2$.
	\end{itemize}
\end{theorem}

\begin{proof}
	[Abridge proof of Theorem \ref{thm:T}] With $(w'_1,w'_2)= (\varphi(w_1),\tfrac{1}{\varphi(1)^2}\varphi(w_{2}))$, $\Psi_k(\ttt,w_1,w_2)\in T_k$ iff $\Psi_k(\ttt,w'_1,w'_2)\in T_k$  by Proposition \ref{prop:sym} about  symmetries of the transversal. Then, we follow the proof of Corollary   \ref{cor:ConstraintsLarge} using the Proposition \ref{prop:geonumber} of the geometry of numbers with strict inequalities.
\end{proof}

\begin{remark}
	The conditions for $\Psi_k(\ttt,w_1,w_2)$ to be in the full transversal are similar, just replace the closed disks by open disks and take care of the particular case $w_1=0$. In this latter case $\Psi_1(\ttt,w_1,w_2)$ is in the transversal whatever $w_2$.
\end{remark}

\subsection{The first return map in the $(\ttt,w_1,w_2)$-coordinates}\label{subsec:firstreturn}
Denote by $R:T\fff T'$ be the first return map in the full transversal $T'$. We want to find the formula 
\[
R(\ttt,w_1,w_2)=(\ttt',w'_1,w'_2)
\]
in the coordinates $(\ttt,w_1,w_2)$. As the example of subsection \ref{subsec:example2}, shows the minimal vectors following one minimal vectors is not necessarily unique up to a multiplicative factor in $\UU_4$.  This makes the map $T_G$ multi-valued. In order to avoid this drawback we restrict the first return map to $T'\setminus \mathcal N=T\setminus \mathcal N$ (see the definition \ref{subsec:fulltrans} where $\mathcal N$ is defined).

Recall that the set $\mathcal N$ is negligible, is invariant by the flow and contains the lattices with nonzero vectors on the axes. Therefore, the restriction of the first return map 
\[
R:T\setminus \mathcal N\fff T\setminus \mathcal N
\] is a bijection. 

For $k=1,2$, let $W'_k$  be the set of $(\ttt,w_1,w_2)\in W_k$ such that $\Psi_k(\ttt,w_1,w_2)\notin \mathcal N$. 
Let $T_G$ be the map defined on the disjoint union of $W'_1$ and  $W'_2$ according to Theorem \ref{thm:continuedfraction}:
\begin{itemize}
	\item If $(w_1,w_2)\in W_1$, let $a\in\{1,1+i\}$ and $g\in\ZZ$ such that  $-a(1,w_2)+g(w_1,1)$ is minimal for the preoder $\prec$ in the set
	\begin{align*}
		\left\{-a(1,w_2)+g(w_1,1): a\in\{1,1+i\},\, g\in\Z[i],\,|\tfrac{a}{w_1}-g|<1\right\},
	\end{align*} 
	then
	\begin{align*}
		T_G(w_1,w_2)=\left(g-\frac{a}{w_1},\frac{1}{g-aw_2}\right)	
	\end{align*}
	and $T_G(w_1,w_2)\in W_1$ or $W_2$ according to $a=a_1(w_1,w_2)=1$ or $a=a_1(w_1,w_2)=1+i$.
	\item If $(w_1,w_2)\in W_2$, let $g\in\ZZ$ such that $-\tfrac1{1+i}(1+w_1,w_2+1)+g(w_1,1)$ is minimal for the preorder $\prec$ in the set
	\begin{align*}
		\left\{-\tfrac1{1+i}(1+w_1,w_2+1)+g(w_1,1):  g\in\Z[i],\,|\tfrac1{(1+i)w_1}+\tfrac1{(1+i)}-g|<1\right\},
	\end{align*}  
	then
	\begin{align*}
		T_G(w_1,w_2)=\left(g-\frac1{(1+i)w_1}-\frac1{(1+i)}, \frac{1}{g-\tfrac1{(1+i)}w_2-\tfrac1{(1+i)}}\right)	
	\end{align*} 
	and $T_G(w_1,w_2)$ is always in $W_1$. In that case we set $a=a_2(w_1,w_2)=\tfrac{1}{1+i}$.
\end{itemize}

Formally the map $T_G$ should be defined on $(\{1\}\times W'_1)\cup(\{2\}\times W'_2)$ with values in the same set.

Now we are able to compute the first return map in $(\ttt,w_1,w_2)$ coordinates.
\begin{theorem}\label{thm:firstreturn}
	Let $k=1$ or $2$. Let $(\ttt,w_1,w_2)$ be in $[0,\tfrac{\pi}{2}[\times W'_k$. Then 
	\[
	R\circ\Psi_k(\ttt,w_1,w_2)=\Psi_j(\ttt', \alpha^2T_G(w_1,w_2))
	\]
	where
	\begin{itemize} 
		\item $j=2$ when $a_k(w_1,w_2)=1+i$ and $j=1$ otherwise,
		\item $\alpha=\alpha(\ttt,w_1)$ is the only element in $\UU_{4}$ such that $ \ttt'=\ttt+\arg w_1+\arg \alpha\in [0,\tfrac{\pi}{2}[$.
	\end{itemize}
\end{theorem}

\begin{proof}
	Set $(w'_1,w'_2)=T_G(w_1,w_2)$.	
	By Theorem \ref{thm:continuedfraction} and the definitions of  the parametrizations $\Psi_1$ and $\Psi_2$ and of the map $T_G$, 
	\[
	R\circ \Psi_k(\ttt,w_1,w_2)=\Psi_j(\ttt+\arg w_1,w'_1,w'_2)
	\]	
	where $j=2$ iff $a_k(w_1,w_2)=1+i$.
	The only thing we have to worry about is that $\ttt+\arg w_1$ could be outside the interval $[0,\tfrac{\pi}{2}[$. In any cases there exists $\alpha \in \UU_{4}$ unique such that  $ \ttt'=\ttt+\arg w_1+\arg \alpha\in [0,\tfrac{\pi}{2}[$. So, if we change the vectors $u(\ttt+\arg w_1,w'_1,w'_2)$
	and $v(\ttt+\arg w_1,w'_1,w'_2)$ in 
	$\alpha u(\ttt+\arg w_1,w'_1,w'_2)$ and 
	$\tfrac{1}{\alpha}v(\ttt+\arg w_1,w'_1,w'_2)$ we obtain the same lattice and we do not change the determinant (see Proposition \ref{prop:para} the definition of $u(.)$ and $v(.)$).
	Now, 
	\begin{align*}
		\alpha u(\ttt+\arg w_1,w'_1,w'_2)= u(\ttt+\arg w_1+\arg\alpha,\tfrac{1}{\alpha^2}w'_1,\alpha^2 w'_2)\\
		\tfrac{1}{\alpha} v(\ttt+\arg w_1,w'_1, w'_2)=v(\ttt+\arg w_1+\arg\alpha,\tfrac{1}{\alpha^2}w'_1,\alpha^2 w'_2)	
	\end{align*}
	and since $\tfrac{1}{\alpha^2}=\alpha^2$, we are done.
\end{proof}

\begin{figure}

	\includegraphics[width=15cm]{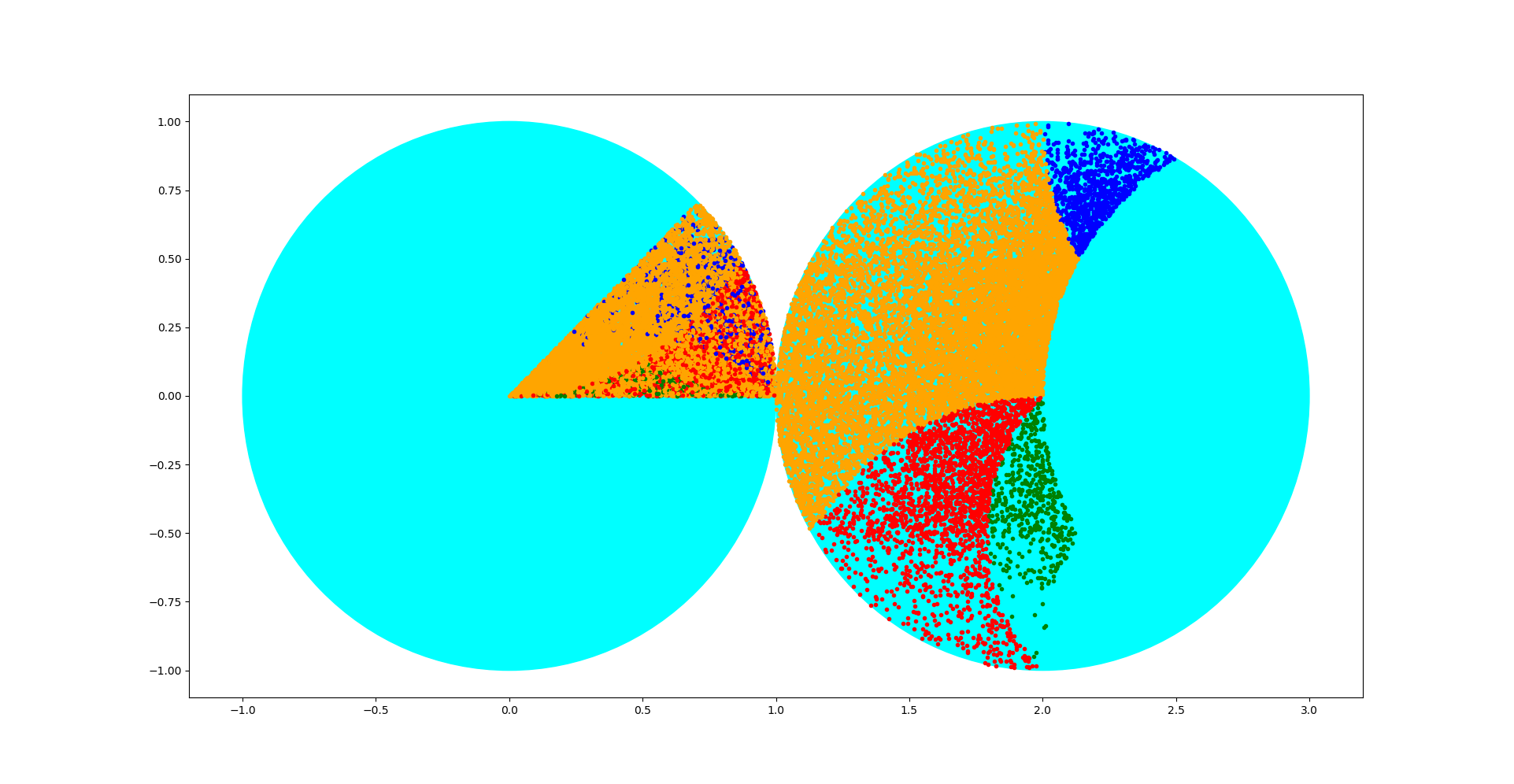}
	
	\includegraphics[width=15cm]{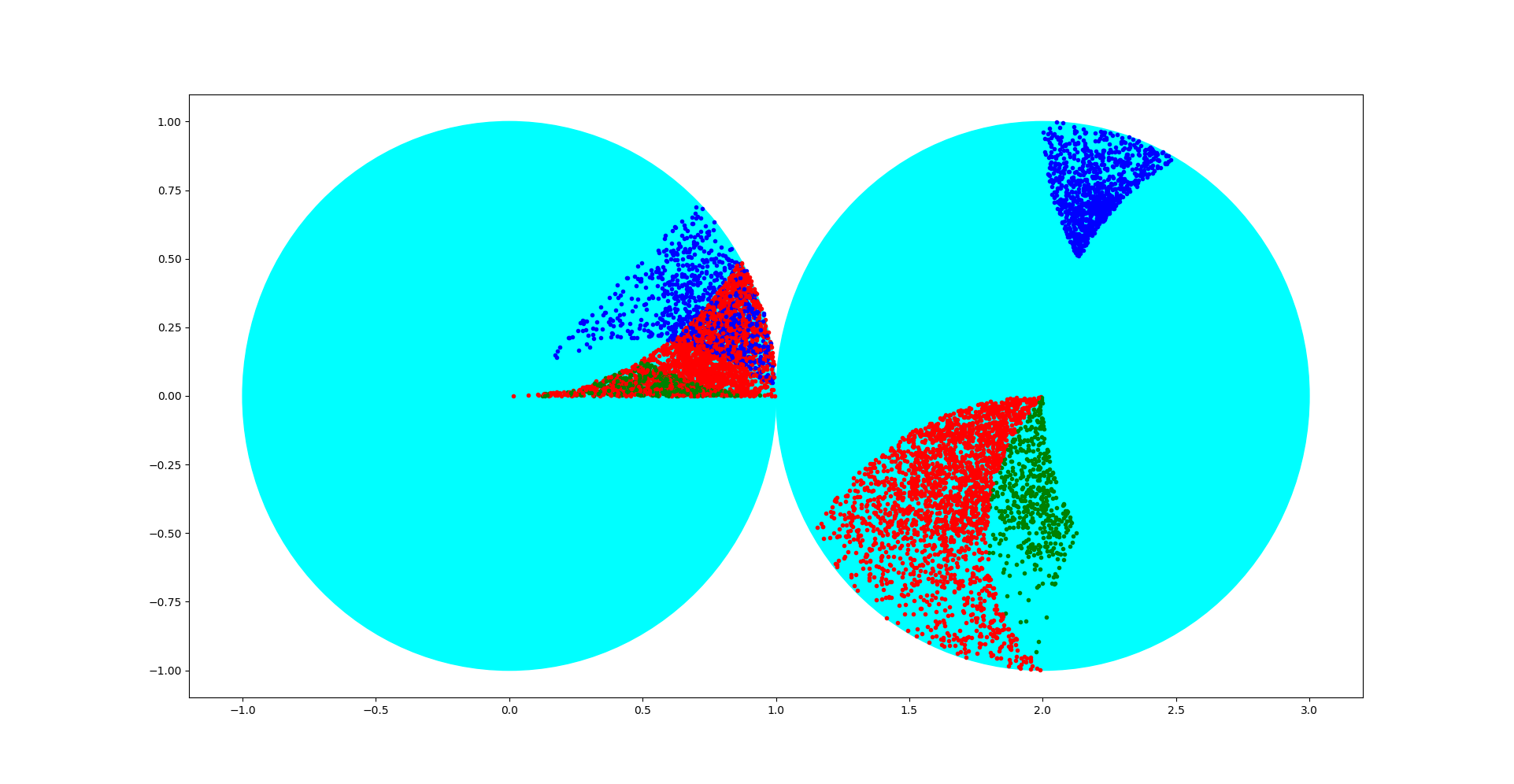}
	
	\caption{Iterates of the map $T_G$ with one initial point. Only  the couple $(w_1,w_2)$ in $T_1$ with $w_1 \in\CC$ are plotted, $w_1$ in the left disk and $w_2$ in the right disk. The color is chosen according to the regions in Corollary \ref{cor:ConstraintsLarge}.  In the second rectangle, the orange points have been suppressed.}

	\end{figure}
	
	\begin{figure}\label{fig:T2}
	
	\includegraphics[width=15cm]{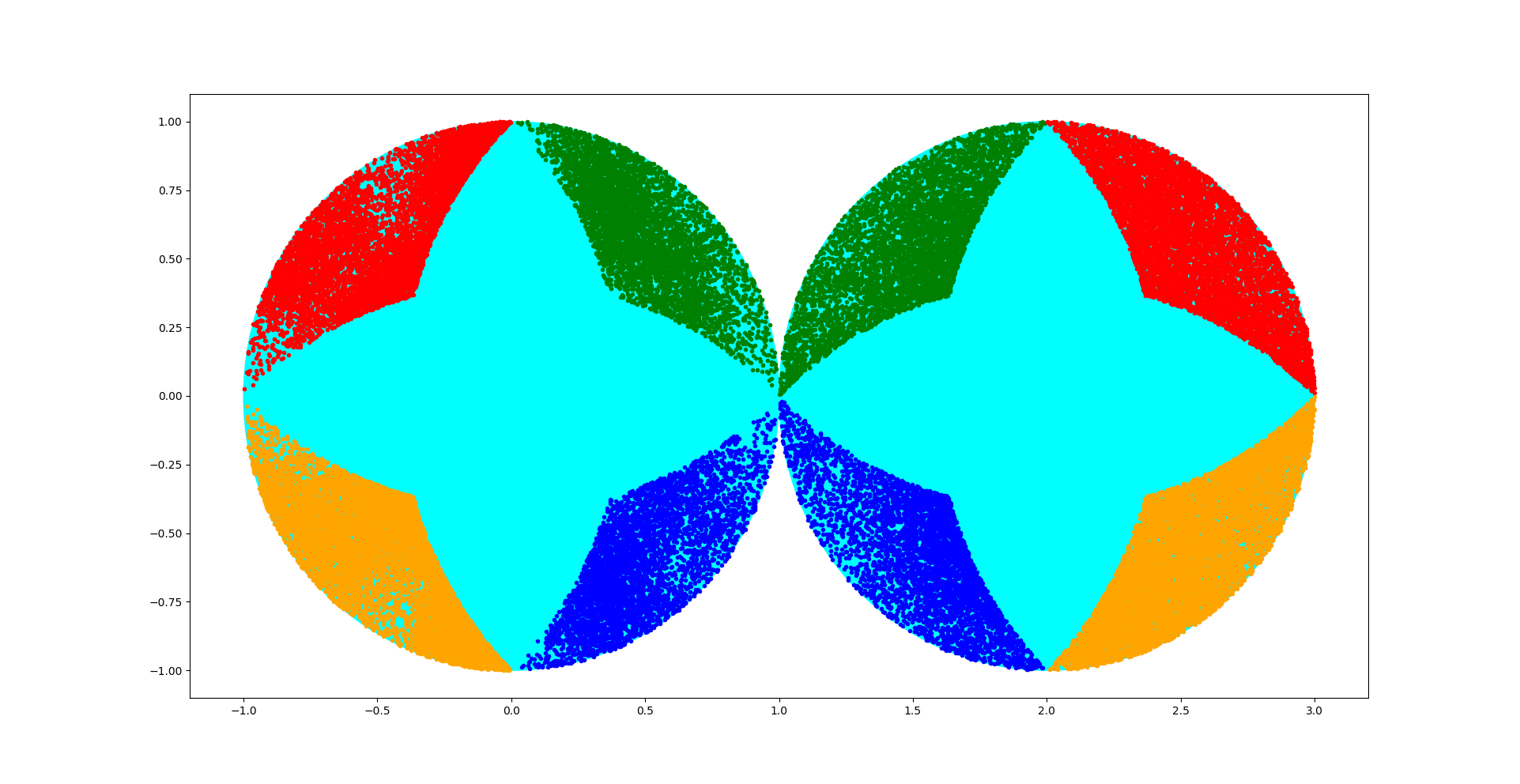}
	
	\caption{Iterates of the map $T_G$: the points in $T_2$.}
\end{figure}

\section{Invariant measures, Proof of Theorem \ref{thm:densityinvariantmeasure1}}

We want to find the measure induced by the Haar measure and the flow
$(g_{t})_{t}$ on the transversal $T$ in the coordinates systems $\Psi
_{1}(\theta,w_{1},w_{2})$ and $\Psi_{2}(\theta,w_{1},w_{2})$ (see Proposition
\ref{prop:para} the definitions of the parametrizations $\Psi_{k}$). The Haar
measure is defined up to a multiplicative constant and can be defined with an
invariant volume form $\alpha$ on $\operatorname{SL}(2,\mathbb{C})$. To take
advantage of the $\mathbb{C}$-linearity of the differential of $\Psi_{k}$ with
respect to $w_{1},w_2$, we use the following volume form. Let
$\omega$ be the differential form of degree $3$ defined on $M_{2}(\mathbb{C})$
by
\[
\omega_{M}(M_{1},M_{2},M_{3})=\operatorname{det}(M,M_{1},M_{2},M_{3})
\]
where $M_{2}(\mathbb{C})$ is identified with $\mathbb{C}^{4}$. Since for every matrix $A\in M_{2}(\mathbb{C})$,
\[
\det(AM,AM_{1},AM_{2},AM_{3})=\det(MA,M_{1}A,M_{2}A,M_{3}A)=(\det A)^{2}%
\det(M,M_{1},M_{2},M_{3}),
\]
the form $\omega$ is $\operatorname{SL}(2,\mathbb{C})$-invariant
and
\[
\alpha=-i\omega\wedge\overline{\omega}
\]
is a volume form on $\operatorname{SL}(2,\mathbb{C})$ and defines a Haar
measure $\mu_{1}$ on $\SL(2,\C)$ (see below  the definition of $\overline\omega$). We restate Theorem \ref{thm:densityinvariantmeasure1}.

\begin{theorem}\label{thm:densityinvariantmeasure2}
	Using the parametrization $\Psi_{k}$, $k=1,2$ of the transversal $T_{k}$, the
	Haar measure $\mu_1$ associated with the volume form $\alpha$ and the flow $g_{t}$
	induce a measure $\nu$ with density
	\[
	h(\theta,w_{1},w_{2})=\frac{32}{|1-w_{1}w_{2}|^{4}}
	\]
	with respect of the Lebesgue measure of $[0,\pi/2]\times\mathbb{D}^{2}$.
\end{theorem}

Before we prove the above theorem, we wish to compute the volume of the space of lattices $\operatorname{SL}%
(2,\mathbb{C})/\operatorname{SL}(2,\mathbb{Z}[i])$. This volume can be deduced from the
volume of $\operatorname{SL}(2,\mathbb{Z}[i]))\backslash\mathbb{H}_{3}$ where
$\mathbb{H}_{3}=\mathbb{C}+j\mathbb{R}_{>0}$ is the three dimensional
hyperbolic space.  Indeed,
consider the left action of $\operatorname{SL}(2,\mathbb{C})$ on the three
dimensional hyperbolic space $\mathbb{H}_{3}$ defined by
\[
M.(z+rj)=\frac{(az+b)(\bar c\bar z+\bar d)+a\bar cr^{2}+jr}{|cz+d|^{2}%
	+|c|^{2}r^{2}}
\]
for $M=%
\begin{pmatrix}
	a & b\\
	c & d
\end{pmatrix}
\in\operatorname{SL}(2,\mathbb{C})$ and $z+rj\in\mathbb{H}_{3}$. The
stabilizer of $j$ is $K=\operatorname{SU}(2,\mathbb{C})$. Choosing an
appropriate Haar measure $\mu_{2}$ on $\operatorname{SL}(2,\mathbb{C})$, we
have
\[
\operatorname{Vol}(\operatorname{SL}(2,\mathbb{C})/\operatorname{SL}%
(2,\mathbb{Z}[i]))=\operatorname{Vol}(\operatorname{SL}(2,\mathbb{Z}%
[i]))\backslash\mathbb{H}_{3})\times\operatorname{Vol}(\operatorname{SU}%
(2,\mathbb{C})).
\]
The measure $\mu_{2}$ can be chosen in order that in the above
formula,
\[
\operatorname{Vol}(\operatorname{SU}(2,\mathbb{C}))=\operatorname{Vol}%
(S_{3})=2\pi^{2}
\]
is the volume of the unit sphere $S_{3}$ with respect to the standard
Euclidean distance and $\operatorname{Vol}(\operatorname{SL}(2,\mathbb{Z}%
[i]))\backslash\mathbb{H}_{3})$ is computed with respect to the hyperbolic
metric on $\mathbb{H}_{3}$. In that case
\[
\operatorname{Vol}(\operatorname{SL}(2,\mathbb{Z}[i]))\backslash\mathbb{H}%
_{3})=\frac{|d|^{\tfrac32}}{4\pi^{2}}\zeta_{K}(2)
\]
where $d=-4$ is the discriminant of the quadratic field $K=\mathbb{Q}%
+i\mathbb{Q}=\mathbb{Q}[\sqrt{-1}]$ and $\zeta_{K}(s)=\sum\frac{1}%
{(a^{2}+b^{2})^{s}}$ where the sum is computed over all the nonzero ideals
$(a+ib)\mathbb{Z}[i]$ in $\mathbb{Z}[i]$ (see \cite{ElGrMe} p. 311-312).
Actually, $\zeta_{K}(2)=\frac{\pi^{2}}{6}C$ where $C=\sum_{n\geq0}%
\frac{(-1)^{n}}{(2n+1)^{2}}$ is the Catalan number.

So we have two Haar measures on $\operatorname{SL}(2,\mathbb{C}%
)/\operatorname{SL}(2,\mathbb{Z}[i])$, one obtained with the volume form
$\alpha$ and a second obtained from $\mathbb{H}_{3}$ and $\operatorname{SU}%
(2,\mathbb{C})$. It is possible to compute the
normalization factor between the two Haar measures $\mu_1$ and $\mu_2$, in fact 
\[
\mu_1=16 \mu_2,
\]
which leads to 
\begin{align*}
	\operatorname{Vol}_{\alpha}(\operatorname{SL}(2,\mathbb{C})/\operatorname{SL}%
	(2,\mathbb{Z}[i]))&=16\times(2\pi^2)\times(\frac{|-4|^{\tfrac32}}{4\pi^{2}}\zeta_{K}(2))\\
	&=64\zeta_{K}(2)=\frac{32\pi^{2}}{3}\sum_{n\geq0}%
	\frac{(-1)^{n}}{(2n+1)^{2}}.
\end{align*}

\begin{proof}[Proof of Theorem \ref{thm:densityinvariantmeasure2}]
	The parametrizations $\Psi_{k}$, $k=1,2$ factor in compositions $\Psi_{k}%
	=\Phi_{k}\circ F_{k}$. Indeed, let $\Phi_{k}:\mathbb{C}^{*}\times
	\mathbb{D}^{2}\rightarrow\operatorname{SL}(2,\mathbb{C})$, $k=1,2$, be the
	maps defined by
	\[
	\Phi_{1}(u_{1},w_{1},w_{2})=
	\begin{pmatrix}
		u_{1} & u_{1}w_{1}\\
		v_{2}w_{2} & v_{2}%
	\end{pmatrix}
	\]
	where $v_{2}=v_{2}(u_{1},w_{1},w_{2})=\frac{1}{u_{1}(1-w_{1}w_{2})}$ and
	\[
	\Phi_{2}(u_{1},w_{1},w_{2})=
	\begin{pmatrix}
		u_{1} & u_{1}w_{1}\\
		v^{\prime}_{2}w_{2} & v^{\prime}_{2}%
	\end{pmatrix}
	\begin{pmatrix}
		1 & \frac{1}{1+i}\\
		0 & \frac{1}{1+i}%
	\end{pmatrix}
	\]
	where $v^{\prime}_{2}=v^{\prime}_{2}(u_{1},w_{1},w_{2})=(1+i)v_{2}(u_{1}%
	,w_{1},w_{2})$. Let $F_{k}:\mathbb{R}\times\mathbb{D}^{2}\rightarrow
	\mathbb{C}^{*}\times\mathbb{D}^{2}$, $k=1,2$, be the maps defined by
	$F_{k}(\theta,w_{1},w_{2})=(u_{1}=re^{i\theta},w_{1},w_{2})$ where
	$r=\frac{k^{1/4}}{\sqrt{|1-w_{1}w_{2}|}}$. By definition, $\Psi_{k}=\Phi
	_{k}\circ F_{k}$.
	
	The first step is to compute the pull back $\Phi_{k*}\omega$. Let
	$p=(u_{1},w_{1},w_{2})$. Straightforward calculations lead to
	\begin{align*}
		(\Phi_{1*}\omega)_{p}(\tfrac{\partial}{\partial u_{1}},\tfrac{\partial
		}{\partial w_{1}},\tfrac{\partial}{\partial w_{2}})  &  = \det%
		\begin{pmatrix}
			u_{1} & 1 & 0 & 0\\
			u_{1}w_{1} & w_{1} & u_{1} & 0\\
			v_{2}w_{2} & w_{2}\frac{\partial v_{2}}{\partial u_{1}} & w_{2}\frac{\partial
				v_{2}}{\partial w_{1}} & v_{2}+w_{2}\frac{\partial v_{2}}{\partial w_{2}}\\
			v_{2} & \frac{\partial v_{2}}{\partial u_{1}} & \frac{\partial v_{2}}{\partial
				w_{1}} & \frac{\partial v_{2}}{\partial w_{2}}\\
			&  &  &
		\end{pmatrix}
		\\
		&  =\frac{-2}{u_{1}(1-w_{1}w_{2})^{2}},
	\end{align*}
	hence
	\[
	(\Phi_{1*}\omega)_{p}=\frac{-2}{u_{1}(1-w_{1}w_{2})^{2}}du_{1}\wedge dw_{1}\wedge
	dw_{2}.
	\]
	Using that $v^{\prime}_{2}=(1+i)v_{2}$, we obtain
	\begin{align*}
		(\Phi_{2*}\omega)_{p}  &  =\det%
		\begin{pmatrix}
			1 & \frac{1}{1+i}\\
			0 & \frac{1}{1+i}%
		\end{pmatrix}
		^{2}\frac{-2(1+i)^{2}}{u_{1}(1-w_{1}w_{2})^{2}}du_{1}\wedge dw_{1}\wedge
		dw_{2}\\
		&  =\frac{-2}{u_{1}(1-w_{1}w_{2})^{2}}du_{1}\wedge dw_{1}\wedge dw_{2}.
	\end{align*}
	
	Let us now use the conjugation. Consider the maps $c:\C^*\times \D^2\fff \C^*\times \D^2$ and $C:\SL(2,\C)\fff\SL(2,\C)$ defined by
	\begin{align*}
		c(u_1,w_1,w_2)&=(\overline u_1,\overline w_1,\overline w_2),\\
		C(\begin{pmatrix}
			a&b\\
			c&d
		\end{pmatrix}
		)&=
		\begin{pmatrix}
			\overline a&\overline b\\
			\overline c&\overline d
		\end{pmatrix}.
	\end{align*} 
	The form $\overline\omega$ is defined by $\overline \omega=C_*\omega$.
	Since $C\circ\Phi_k=\Phi_k\circ c$, we have
	\begin{align*}
		\Phi_{k*}\overline \omega&=(C\circ\Phi_k)_*\omega\\
		&=(\Phi_k\circ c)_*\omega\\
		&=c_*\Phi_{k*}\omega.
	\end{align*}
	With  $u_{1}=u_{11}+iu_{12},\,w_{1}=w_{11}+iw_{12}$ and $w_{2}=w_{21}+iw_{22}$, we have
	\[
	du_{1}=du_{11}+idu_{12} \text{ and } \overline{du_{1}}=c_*du_1=du_{11}-idu_{12}.
	\]
	Hence, 
	\begin{align*}
		c_*(\frac{-2}{u_{1}(1-w_{1}w_{2})^{2}}du_{1}\wedge dw_{1}\wedge
		dw_{2})&=\frac{-2}{c(u_{1})(1-c(w_{1})c(w_{2}))^{2}}c_*du_{1}\wedge c_*dw_{1}\wedge
		c_*dw_{2}\\
		&=\overline{\frac{-2}{u_{1}(1-w_{1}w_{2})^{2}}}\overline{du_{1}}\wedge \overline{dw_{1}}\wedge
		\overline{dw_{2}}
	\end{align*}
	Therefore, in coordinates $(u_{1},w_{1},w_{2})$, the Haar measure is
	associated with the differential form
	\begin{align*}
		(\Phi_{k*}\alpha)_{p}  &  =-i\frac{4}{|u_{1}|^{2}|1-w_{1}w_{2}|^{4}}%
		du_{1}\wedge dw_{1}\wedge dw_{2}\wedge\overline{du_{1}}\wedge\overline{dw_{1}%
		}\wedge\overline{dw_{2}}\\
		&  =\frac{32}{|u_{1}|^{2}|1-w_{1}w_{2}|^{4}}du_{11}\wedge du_{12}\wedge
		dw_{11}\wedge dw_{12}\wedge dw_{21}\wedge dw_{21}.
	\end{align*}

	In coordinates $(u_{1},w_{1},w_{2})$, the diagonal flow $g_{t}$ writes
	\[
	g_{t}(u_{1},w_{1},w_{2})=(e^{t}u_{1},w_{1},w_{2})
	\]
	and is associated with the vector field $X(p)=(u_{1},0,0)$.
	
	In order to compute the measure induced by the Haar measure and the flow
	$g_{t}$, it is enough to compute the Jacobian of the map
	\[
	(t,\theta,w_{1},w_{2})\rightarrow g_{t}\circ F_{k}(\theta,w_{1},w_{2}%
	)=(re^{t+i\theta},w_{1},w_{2})
	\]
	at $(0,\theta,w_{1},w_{2})$. It is the $6\times6$ determinent
	\[
	\det%
	\begin{pmatrix}
		r\cos\theta & -r\sin\theta & . & . & . & .\\
		r\sin\theta & r\cos\theta & . & . & . & .\\
		0 & 0 & 1 & 0 & 0 & 0\\
		0 & 0 & 0 & 1 & 0 & 0\\
		0 & 0 & 0 & 0 & 1 & 0\\
		0 & 0 & 0 & 0 & 0 & 1\\
	\end{pmatrix}
	=r^{2}.
	\]
	Finally, we obtain the density
	\begin{align*}
		h(\theta,w_{1},w_{2})=\frac{32 r^{2}}{|u_{1}|^{2}|1-w_{1}w_{2}|^{4}}=\frac{32
		}{|1-w_{1}w_{2}|^{4}}.
	\end{align*}
	
\end{proof}

\section{Dirichlet best constant, proof of Theorem \ref{thm:Dirichlet}}
Let $\ttt\in\C$. The Dirichlet constant associated with $\ttt$ is the infimum $C(\ttt)$ of  constants $C$ such that for all real number $Q\geq 1$ there exist $p,q\in\ZZ$ such that
\[
\left\{
\begin{array}{l}
	0<|q|< Q\\
	|q\ttt-p|\leq \frac{C}{Q}
\end{array}
\right..
\]
The best constant in Theorem \ref{thm:Dirichlet} is then $C_D=\sup\{C(\ttt):\ttt\in\C\}$.

Let $(p_n,q_n)\in\ZZ^2$, $n\in I_{\ttt}\subset\N$, be a sequence of best approximations vectors of $\ttt$ such that
\[
1=|q_0|<|q_1|<\dots<|q_n|<\dots, 
\]  
and including all the denominators: if $(p,q)$ is a best approximation vector then there is an $n$ such that $|q|=|q_n|$. The sequence is infinite, i.e. $I_{\ttt}=\N$, iff $\ttt\notin\Q[i]$.

Then, it is clear that
\[
C(\ttt)=\sup\{|q_{n+1}||q_n\ttt-p_n|:n,n+1\in I_{\ttt}\}.
\]
If we want to study the best Dirichlet constant for all large enough $Q$ when $\ttt\notin \Q[i]$, we have to use the constant 
\[
C'(\ttt)=\limsup_{n\fff\infty}|q_{n+1}||q_n\ttt-p_n|.
\]
instead of the constant $C(\ttt)$.
By Proposition \ref{prop:best}, the sequence of best approximation of $\ttt$ is the sequence of minimal vectors of the lattice
\[
\LL_{\ttt}=\begin{pmatrix}
	1&-\ttt \\
	0&1
\end{pmatrix}\Z[i]^2=M_{\ttt}\Z[i]^2.
\]
More precisely, the sequence $M_{\ttt}\begin{pmatrix}p_n\\q_n\end{pmatrix}$ contains  exactly one element equivalent to any minimal vector of $\LL_{\ttt}$ with nonzero second coordinate. It follows that the best Dirichlet constant $C_D$ is bounded above by 
\[
C_S=\sup|u_1||v_2|
\]
where the supremum is taken over all Gauss unimodular lattices $\LL$ in $\C^2$ and all the pairs  $u=(u_1,u_2), v=(v_1,v_2)$ of consecutive minimal vectors in $\LL$ with $0<|u_2|<|v_2|$.    The proof now goes in two steps :
\begin{enumerate}
	\item We prove that $C_S=\frac{1}{\sqrt{6-3\sqrt 3}}=\frac{\sqrt 2}{ 3-\sqrt 3}$,
	\item We prove that for almost all $\ttt\in\C$, $C'(\ttt)=C_S$.
\end{enumerate}
\subsection{Step 1}
\subsubsection{ A first reduction to compute $C_S$} 
Let $u=(u_1,u_2)$ and $v=(v_1,v_2)$ be two consecutive minimal vectors of a unimodular Gauss lattice $\LL$ in $\C^2$.  Then $|u_1|> |v_1|$ and $|v_2|>|u_2|$ and by Theorem \ref{thm:index}, the index of the sublattice $\ZZ u+\ZZ v$ in $\LL$, is one or two. Thus we can write
\[
\left\{
\begin{array}{ll}
	u=(u_1,v_2w_2),&|w_2|<1\\
	v=(u_1w_1,v_2),&|w_1|<1
\end{array}
\right.
\]   
and $|\det_{\C}(\ZZ u+\ZZ v)|=|u_1v_2(1-w_1w_2)|=1$ or $\sqrt 2$ according to  index $1$ or $2$. 
Set
\[
C_1=\sup u_1v_2=\sup\frac{1}{|1-w_1w_2|}
\]
where the supremum is taken over all unimodular lattice $\LL$ and all pairs of consecutive minimal vectors $u,v$ of index $1$ in $\LL$ and set \[
C_2=\sup u_1v_2=\sup\frac{\sqrt 2}{|1-w_1w_2|}
\]
where the supremum is taken over all unimodular lattice $\LL$ and all pairs of consecutive minimal vectors $u,v$ of index $2$ in $\LL$. Then 
\[
C_S=\max(C_1,C_2).
\]
Thanks to Proposition \ref{prop:sym}, using the symmetries associated with $\varphi\in \D_8$, we can suppose that $w_1\in \CC$. We can now evaluate $C_1$ and $C_2$ using Corollary \ref{cor:ConstraintsLarge} that give necessary and sufficient conditions on $w_1$ and $w_2$ in order that $u$ and $v$ are consecutive minimal vectors. 
\subsubsection{We show that $C_1=\frac{1}{\sqrt{6-3\sqrt 3}}$}
We want to bound above the function $f(w_1,w_2)=\frac{1}{|1-w_1w_2|}$.
In the particular case $w_1=0$, $f(w_1,w_2)=1$ so that $f(w_1,w_2)\leq \frac{1}{\sqrt{6-3\sqrt 3}}$. From now on, we suppose $w_1\neq 0$.

Recall the notations \begin{align*}
	&  Red_{1}=\D(i,1),\, \,Red_{2}=\D(-i,1)\\
	&  Blue_{1}=\D(\tfrac{1-i}{2},\tfrac1{\sqrt2}),\,Blue_{2}=\D(1+i,1)\\
	&  Green_{1}=\D(1+i,1),\,Green_{2}=\D(\tfrac{1-i}{2},\tfrac1{\sqrt2}).
\end{align*}
Let  $u=(u_{1},v_{2}w_{2})$ and
$v=(u_{1}w_{1},v_{2})$ be two vectors in $\C^2$ with $|u_{1}|,|v_{2}|>0$, $|w_1|,|w_2|<1$, and  $w_{1}\in\mathcal{C}\setminus\{0\}$. 

By Corollary \ref{cor:ConstraintsLarge},		zero is the only vector of $\ZZ u+\ZZ v$ that is in $\overset{o}C(u,v)$ iff $w_{2}\in\overline{\mathcal{D}}$ and one of the four conditions

\begin{enumerate}
	\item $w_{2}\in Green_{2}$ and $w_{1}\notin Red_{1}\cup Green_{1}$,
	
	\item $w_{2}\in Red_{2}\setminus Green_{2}$ and $w_{1}\notin Red_{1}$,
	
	\item $w_{2}\notin Red_{2}\cup Blue_{2}$,
	
	\item $w_{2}\in Blue_{2}$ and $w_{1}\notin Blue_{1}$,
\end{enumerate}
holds.

So we have to compute the supremum of the function $f(w_1,w_2)=\frac{1}{|1-w_1w_2|}$ over the four regions defined by (1), (2), (3) and (4).

In the following we assume the arguments of  complex numbers are in $[0,2\pi[$.\medskip

{\sc Case 1: $(w_1,w_2)\in \CC\setminus \{0\}\times \overline\DD$, $w_{2}\in Green_{2}$ and $w_{1}\notin Red_{1}\cup Green_{1}$.} 

We have to minimize the distance from $w_1w_2$ to $1$ when $(w_1,w_2)$ is in this region. 
If $w_1$ is on the circle of radius $r_1$ centered at $0$ and $w_2$ on the circle of radius $r_2$ centered at $0$, the point $w_1w_2$ is on the circle of radius $r_1r_2$ and will be closest to $1$ when the arguments of $w_1$ and $w_2$ are  maximal.
It follows that the infimum is reached when $w_1$ and $w_2$ are in the following arcs of circle (see Figure 3 in subsection \ref{subsec:constraints})
\begin{enumerate}
	\item[(a)] $w_1$ is in the arc $\cc_a$ of the circle $\cc(i,1)$ from  $z_0=0$ to $z_1=1/2+(1-\sqrt 3/2)i$ (positive orientation), 
	\item[(b)] $w_1$ is in the arc $\cc_b$ of the circle $\cc(1+i,1)$ from  $z_1$ to $z_2=1$
\end{enumerate} 
and
\begin{itemize}
	\item[(c)] $w_2$ is in the arc $\cc_c$ of the circle $\cc(1,1)$ from  $z_0$ to $z_3=-iz_1$,
	\item[(d)] $w_2$ is in the arc $\cc_d$ of the circle $\cc(1-i,1)$ from  $z_3$ to $z_4=-i$.
\end{itemize}
We are going to show that the infimum of $|1-w_1w_2|$ is 
\[r=|1-z_1z_3|=\sqrt{6-3\sqrt 3}.
\]

{\sc Sub-case $w_1\in\cc_a$ and $w_2\in\cc_c$.}
Since $|w_1w_2|\leq |z_1z_3|<1$ and $3\pi/2\leq \arg w_1w_2\leq \arg z_1z_3$. Now  $z_1z_3=1-\sqrt 3/2-i(\sqrt 3-3/2)$, $|z_1z_3|=2-\sqrt 3$ and $\arg z_1z_3=2\pi-\pi/3$, hence $w_1w_2$ is in the sector
\[
S=\{z\in \C: 3\pi/2\leq \arg w_1w_2\leq  2\pi-\pi/3,\, |z|\leq 2-\sqrt 3 \}.
\]
Since $z_1z_3\in \cc(1,r)$ and $2-\sqrt 3<1/2$, this sector doesn't intersect the open disk $\D(1,r)$, hence $|1-w_1w_2|\leq r$. \medskip

{\sc Sub-case $w_1\in\cc_a$ and $w_2\in\cc_d$.} Suppose first that $w_1\notin D(1,r)$.
It is enough to prove that $\Re( w_1w_2)\leq 1-r$. We have $\Re( w_1w_2)=|w_1w_2|\cos(\arg w_1w_2)$. As before $\arg w_1w_2\in[3\pi/2, \arg z_1z_3]$, so that 
\[
\cos(\arg w_1w_2)\leq \cos(\arg z_1z_3)=\frac12.
\] 
Let $w_0=x+iy$ be the point in $\cc_a$ on the circle $\cc(1,r)$. We have that $|w_1|\leq |w_0|$ so that 
$\Re(w_1w_2)=|w_1w_2|\cos(\arg w_1w_2)\leq |w_0|/2$. $w_0=x+iy$ can easily be computed because its real part is solution of an equation of degree $2$. We find that $x\leq 0.103...$ and $y\leq 0.0054...$ so that $|w_0|\leq 0.11$. Since $1-r\geq 0.10$ we are done.\\

When $w_1\in D(1,r)$ and in the remaining cases below, we shall use the following simple lemma. It is an easy consequence of the fact that two circles meet in  two points at most.
\begin{lemma}\label{lem:circle}
	Let $\cc_1$ and $\cc_2$ be two circles in the plane and $p_1,p_2$ and $p_3$ be three distinct points in $\cc_1$. If $p_1$ and $p_2$ are not in the interior of $\cc_2$ while $p_3$ is in the closed disk associated with $\cc_2$, then the closed arc of the circle $\cc_1$ between $p_1$ and $p_2$ that doesn't contain $p_3$, doesn't intersect the interior of $\cc_2$.	
\end{lemma}
We now use the lemma with $\cc_1=w_1\cc(1-i,1)$ and $\cc_2=\cc(1,r)$. Since $w_1\in \cc_a$ and $z_3\in\cc_c$, the products $w_1z_3$ is not in the interior of $\cc_2$. Next   $w_1z_4=-iw_1\in\cc(1,1)$ is not in the interior of $\cc_2$ while the product $w_1\times 1$ is in the closed disk associated with $\cc_2$, therefore, by Lemma \ref{lem:circle}, $w_1\cc_d$ doesn't meet the interior of $\cc_2$.\medskip 

{\sc Sub-case $w_1\in\cc_b$ and $w_2\in\cc_c$.} 
We have $w_1w_2=w'_1w'_2$ with $w'_1=iw_2\in \cc_a$ and $w'_2=-iw_1\in\cc_d$ so we are done thanks to the above case.\medskip

{\sc Sub-case $w_1\in\cc_b$ and $w_2\in\cc_d$.} We now show that $|1-w_1w_2|\geq r$ when $w_1\in\cc_b$  and $w_2\in \cc_d$.
We fix $w_2$. The points $w_2z_1$ and $w_2z_2=w_2$ are not in the interior of $\cc_2$ while $w_2\times i$ is in the interior of $\cc_2$, therefore, by Lemma \ref{lem:circle}, $w_2\cc_b$ doesn't meet the interior of $\cc_2$. \medskip

{\sc Case 2: $(w_1,w_2)\in \CC\setminus \{0\}\times \overline\DD$, $w_{2}\in Red_2\setminus Green_{2}$ and $w_{1}\notin Red_{1}$.}	

As before, if $w_1$ is on the circle of radius $r_1$ and $w_2$ on the circle of radius $r_2$, the point $w_1w_2$ is on the circle of radius $r_1r_2$ and will be closest to $1$ when the arguments of $w_1$ and $w_2$ are  maximal.
It follows that the infimum is reached when $w_1$ and $w_2$ are in the following arcs of circle (see Figure 3)
\begin{enumerate}
	\item[(a)] $w_1$ is in the arc $\in\cc_a$ of the circle $\cc(i,1)$ from the point $z_0=0$ to the point $z_1=z_1=i+e^{i\pi/3}=e^{i\pi/6}$ (positive orientation).
	\item[(b)] $w_2$ is in the arc $\cc_b$ of the circle $ \cc(\tfrac{1-i}{2},\tfrac{1}{\sqrt 2})$ from the point $z_0=0$ to the point $z_2=-i$ (positive orientation).
\end{enumerate}

The points three points $z_0=0$, $z_2=-i$ et $z_3=1-i$ are on  the circle $\cc_1=\cc(\tfrac{1-i}{2},\tfrac{1}{\sqrt 2})$. 
The products $z_1z_0$ and $z_1z_2=e^{-i\pi/3}$ are one the circle $\cc_2=\cc(1,1)$ and the product $z_1z_3=\sqrt 2 e^{-i\pi/12}$ is in the interior of $\cc_2$. Therefore, by Lemma \ref{lem:circle}, the arc $z_1\cc_b$ doesn't meet the interior of $\cc_2$ which means that
\[
|1-z_1w_2|\geq 1
\]
for all $w_2\in\cc_b$.

Next fix $w_2\in\cc_b$. The three points $z'_1=w_2z_0$, $z'_2=w_2z_1$ and $z'_3=w_2(1+i)$ are on the circle $\cc_3=w_2\cc(i,1)$. The point $z'_1$ is not in the interior of the circle $\cc_2$ and we just proved that $z'_2$ is not in the interior of the circle $\cc_2$ either, while $z'_3$ is on the circle $(1+i)\cc(\tfrac{1-i}{2},\tfrac{1}{\sqrt 2})=\cc_2$. Therefore, by Lemma \ref{lem:circle}, the arc $w_2\cc_a$ doesn't meet the interior of the circle $\cc_2$ which means that
\[
|1-w_2w_1|\geq 1
\]
for all $w_1\in\cc_a$.\medskip  

{\sc Case 3: $(w_1,w_2)\in \CC\setminus \{0\}\times \overline\DD$, $w_{2}\notin Red_2\cup Blue_{2}$.}

If $w_{2}\notin Red_2\cup Blue_{2}$ then either $\arg w_2\leq \pi/2$ or 
\[
\pi/2\leq \arg w_2 \leq  7\pi/6.
\] 
In the latter case,
since $w_1\in \CC$, we have
\begin{align*}
	w_1w_2\in&\{z:\pi/2\leq \arg z\leq 7\pi/6+\pi/4\}\\
	&\subset\{z:\Re z\leq 0\}.
\end{align*}
Therefore, $|1-w_1w_2|\geq 1$. In the former case, $\arg w_1w_2\geq \arg ((1-\sqrt 3/2)+i/2)=\pi/2-\pi/12$. Therefore, $|1-w_1w_2|\geq \cos(\pi/12)>r=\sqrt{6-3\sqrt 3}$ for all $w_1\in\CC$ and $w_2\notin Red_2\cup Blue_{2}$.	\medskip

{\sc Case 4: $(w_1,w_2)\in \CC\setminus \{0\}\times \overline\DD$, $w_2\in Blue_1$ and $w_{1}\notin Blue_{2}$.}

Again, if $w_1$ is on the circle of radius $r_1$ and $w_2$ on the circle of radius $r_2$, the point $w_1w_2$ is on the circle of radius $r_1r_2$ and will be closest to $1$ when the arguments of $w_1$ and $w_2$ are minimal.
It follows that the infimum is reached when $w_1$ and $w_2$ are in the following arcs of circle (see Figure 3)
\begin{enumerate}
	\item[(a)] $w_1$ in the arc $\cc_a$ of the circle $\cc(\tfrac{1-i}{2},\tfrac{1}{\sqrt 2})$ from the point $z_1=1$ to the point $z_0=0$.
	\item[(b)] $w_2$ in the arc $\cc_b$ of the circle $\cc(1,1)$ from the point $z_2=1/2+i\sqrt 3/2$ to the point $z_3=(1-\sqrt 3/2)+i/2$.
\end{enumerate}
Fix $w_2$ in $\cc_b$. Then the extremties of $w_2\cc_a$ are $0$ and $w_2$ and they are not inside the circle $\cc(1,1)$. While the point $w_2(-i)\in w_2\cc(\tfrac{1-i}{2},\tfrac{1}{\sqrt 2})$ is inside the circle $\cc(1,1)$. Therefore, the arc $w_2\cc_a$ is outside the circle $\cc(1,1)$ which means that $|1-w_1w_2|\geq 1$ for all $w_1\in\cc_a$
\subsubsection{We show that $C_2=\frac{1}{\sqrt{6-3\sqrt 3}}$}

We suppose that $u$ and $v$ are of index $2$. Thanks to Proposition \ref{prop:sym} we can suppose $w_1\in\CC$ and thanks to Corollary \ref{cor:ConstraintsLarge}, we know that,
zero  is the only vector of $\langle u,v\rangle_J$  in $\overset{o}C(u,v)$, iff $(w_{1},w_{2})\in(\mathcal{C}\setminus \D(-i,\sqrt2))\times
\overline{\mathcal{T}}$. We want to show that   $\sup|1-w_1w_2|= 3-\sqrt 3.$

Once again, if $w_1$ is on the circle of radius $r_1$ and $w_2$ on the circle of radius $r_2$, the point $w_1w_2$ is on the circle of radius $r_1r_2$ and will be closest to $1$ when the arguments of $w_1$ and $w_2$ are minimal.
It follows that the infimum is reached when $w_1$ and $w_2$ are in the following arcs of circle (see Figures 5 and 6)
\begin{enumerate}
	\item[(a)] $w_1$ is in the arc $\cc_a$ of the circle $\cc(-i,\sqrt 2)$ with extremities  $z_0=1$ and $z_1=\tfrac{\sqrt 3-1}{2}(1+i)$,
	\item[(b)] $w_2$ is in the arc $\cc_b$ of the circle $\cc(1,\sqrt 2)$ with extremities   and $z_2=i$ and $z_3=\tfrac{\sqrt 3-1}{2}(-1+i)$.
\end{enumerate}
\begin{figure}[H]
\includegraphics[width=15cm]{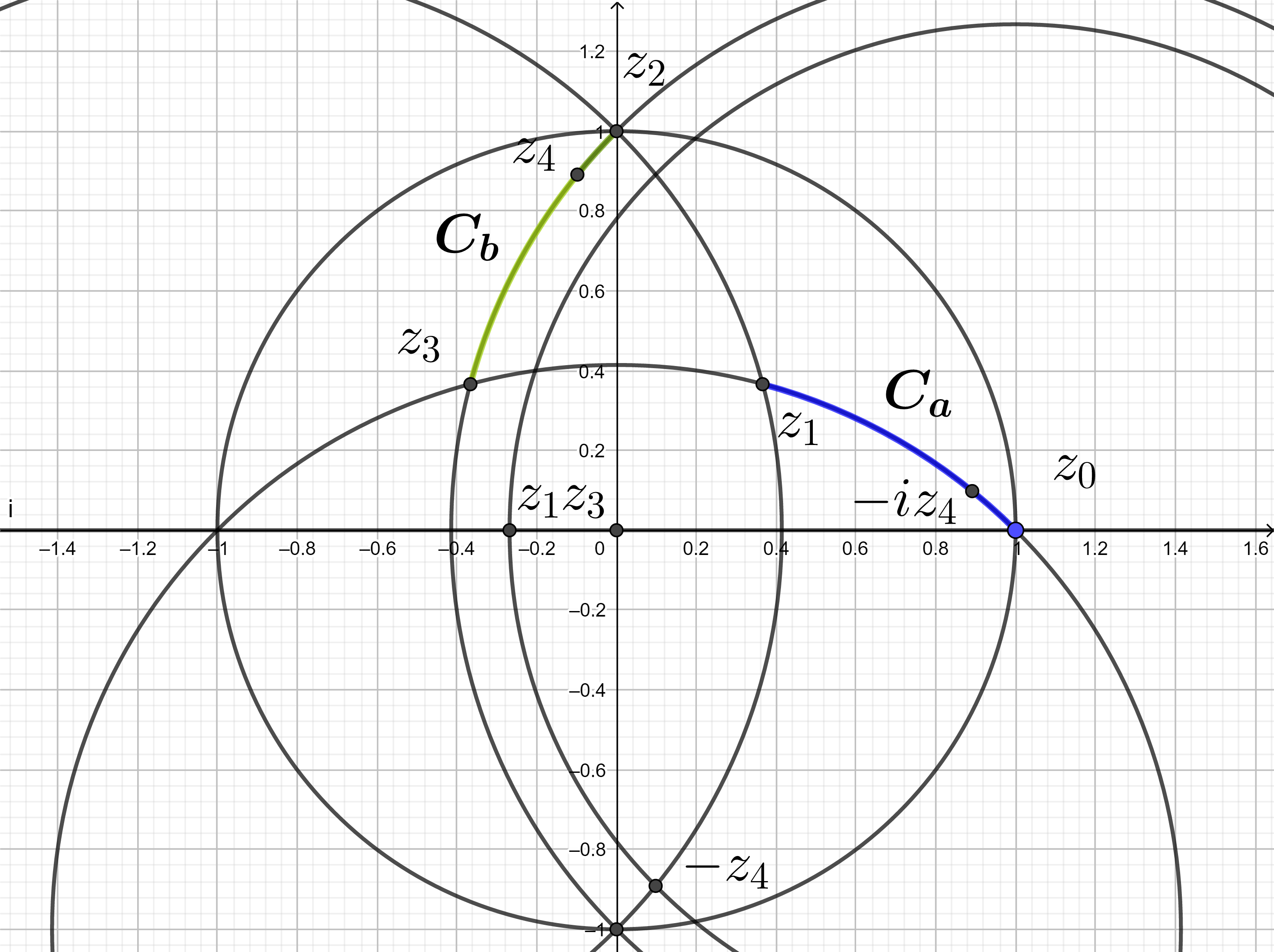}
\caption{The infimum of $|1-w_1w_2|$ on $T_2$.}
\end{figure}
A short computation shows that $\rho=|1-z_1z_3|=3-\sqrt 3=\sqrt 2 \sqrt{6-3\sqrt 3}$. 

Our objective is to show that if $w_1\in\cc_a$ and $w_2\in\cc_b$ then 
\[
|1-w_1w_2|\geq 3-\sqrt 3.
\]
It will implies that $C_2=\tfrac{\sqrt 2}{3-\sqrt 3}$.

When $w_1=z_1$, the points $w_1z_2$ and $w_1z_3$ are not inside  the circle $\cc=\cc(1,\rho)$ and the point $w_1(-i)$ is inside $\cc$. Therefore, by Lemma \ref{lem:circle}, the arc $w_1\cc_b$ doesn't meet the interior of $\cc$.

Fix $w_2\in\cc_b$. The points $w_2z_0=w_2$  and  $w_2z_1$ 
are not inside the circle $\cc$. If  the point $w_2(-1)$ were inside $\cc$, by Lemma \ref{lem:circle}, $w_2\cc_a$ would be  outside $\cc$ which means that 
\[
|1-w_1w_2|\geq 3-\sqrt 3.
\]
However, $w_2(-1)$ could be outside the circle $\cc$. Let us determine the points $w_2=x+iy\in\cc_b$ such that $w_2(-1)$ is inside the circle $\cc$: $w_2(-1)$ is inside $\cc$ iff
\begin{align*}
	&|1-w_2(-1)|^2<\rho^2\\
	\Leftrightarrow&\, (1+x)^2+y^2<(3-\sqrt 3)^2\\
	\Leftrightarrow&\,x^2+y^2+2x+1<12-6\sqrt 3.
\end{align*}
Now since $w_2\in \cc(1,\sqrt 2)$, $x^2+y^2=1+2x$, hence $w_2(-1)$ is  inside $\cc$ iff $x<\tfrac{5-3\sqrt 3}{2}$. 
So by Lemma \ref{lem:circle}, if $x\leq   \tfrac{5-3\sqrt 3}{2}$ then $|1-w_1w_2|\geq 3-\sqrt 3$. 

Call $x_0=\tfrac{5-3\sqrt 3}{2}$, $y_0=\sqrt{2-(x_0-1)^2}$ and $z_4=x_0+iy_0$.
Let $\cc'_b$ and $\cc''_b$ be the portions of the arc $\cc_b$ from $z_2=i$ to $z_4$ and from $z_4$ to $z_3$. 
Let $\cc'_a$ and $\cc''_a$ be the portions of the arc $\cc_a$ from $z_0=1$ to $(-i)z_4$ and from $(-i)z_4$ to $z_1=(-i)z_3$.

We already know that if $w_1\in \cc_a$ and $w_2\in \cc''_b$ then $|1-w_1w_2|\geq 3-\sqrt 3$. Since $w_2\in\cc_b$ and $w_1\in\cc''_a$ implies $(-i)w_2\in\cc_a$ and $iw_1\in\cc''_b$, it follows that if $w_2\in\cc_b$ and $w_1\in\cc''_b$ then
\[
|1-w_1w_2|=|1-iw_1(-1)w_2|\geq 3-\sqrt 3.
\]

So we are left with the case $w_1\in\cc'_a$ and $w_2\in\cc'_b$.
Since $|1-iw_1|=\sqrt 2$,
\begin{align*}
	|1-w_1w_2|^2&=|1-w_1(w_2-i+i)|^2=|(1-iw_1)-w_1(w_2-i)|^2\\	
	&=2+|w_1(w_2-i)|^2-2\Re((1-iw_1)\overline{w_1(w_2-i)}).	
\end{align*}
Furthermore, when $w_1\in\cc'_a$ and $w_2\in\cc'_b$,
\begin{align*}
	7\pi/4\leq \arg(1-iw_1)&\leq \arg(1-z_4)\\
	2\pi-\arg(-iz_4)\leq \arg \overline w_1&\leq  2\pi\\
	2\pi-\arg(z_4-i)\leq \arg(\overline{w_2-i})&\leq 2\pi-5\pi/4, 
\end{align*}
and   $\arg(z_4-i)=5\pi/4+\eps_1$, $\arg(-iz_4)=\eps_2$ and $\arg(1-z_4)=7\pi/4+\eps_3$ where $\eps_1,\eps_2$ and $\eps_3$ are positive and  small.
It follows that, modulo $2\pi$, we have
\begin{align*}
	7\pi/4-\arg(-iz_4)-\arg(z_4-i)&\leq \arg((1-iw_1)\overline{w_1(w_2-i)})\leq \arg(1-z_4)-5\pi/4\\
	\pi/2-\eps_1-\eps_2&\leq \arg((1-iw_1)\overline{w_1(w_2-i)})\leq \pi/2+\eps_3.
\end{align*}
Now $\eps_1=\arg(z_4-i)-5\pi/4=0.0.0518\dots\leq 0.06$, and $\eps_2=\arg(-iz_4)=0.109\dots\leq 0.11$
therefore,
\begin{align*}
	2\Re((1-iw_1)\overline{w_1(w_2-i)})&\leq 2\sqrt{2}|w_2-i|\cos(\pi/2-\eps_1-\eps_2)\\
	&\leq2\sqrt{2}\times0.15 \times 0.17 \\
	&\leq 0.08.
\end{align*}
Finally, we obtain
\[
|1-w_1w_2|\geq \sqrt{2-0.08}\geq 1.38\geq 3-\sqrt{3}
\]
and we are done.

\subsection{Step 2, $C(\ttt)=C_S$ for almost all $\ttt\in\C$.}
Consider the lattice $\LL_0=\ZZ u+\ZZ v$ defined by the vectors $u=(u_1,u_2)=r_0(1,e^{i\alpha}w_2)$ and $v=(v_1,v_1)=r_0(w_1,e^{i\alpha})$ where 
\begin{align*}
	w_1&=1-\tfrac{\sqrt3}{2}+\tfrac12 i,\\
	w_2&=(-i)w_1,\\
	r_0&=\frac{1}{\sqrt{|1-w_1w_2|}},\\
	\alpha&=-\arg(1-w_1w_2).	
\end{align*} 
The lattice $\LL_0$ is unimodular and by Theorem \ref{thm:geonumber1} (more precisely by Corollary \ref{cor:ConstraintsLarge}), the open cylinder $\overset{o}C(u,v)=\overset{o}B_{\infty}(0,r_0)$. The vectors $u$ and $v$ has been chosen in order that 
\[
|u_1v_2|=r_0^2=C_S=\frac{\sqrt 2}{3-\sqrt 3}.
\]

\begin{lemma}\label{lem:dirichletproof}
	Let $\LL$ be a lattice in $\C^2$ and let $r$ be a positive real number. Suppose that for some real number $t$, $g_t\LL\cap B_{\infty}(0,r)=\{0\}$. Then there exist two consecutive minimal vectors $u=(u_1,u_2)$ and $v=(u_2,v_2)$ in $\LL$ with $|v_2|>|u_2|$ and such that $|u_1v_2|\geq r^2$ and $|v_2|\geq re^{t}$. 
\end{lemma}
\begin{proof}
	Let $u=(u_1,u_2)$ be a minimal vector in $g_t\LL$ with $|u_1|\geq r$ and $|u_1|$ minimal. Such a minimal vector exists because by Lemma \ref{lem:lexico}, there exist minimal vectors in $\overset{o}C_2(r)$ and such minimal vectors  have a first coordinate with modulus $\geq r$ because $g_t\LL\cap B_{\infty}(0,r)=\{0\}$. 
	Let $v=(v_1,v_2)$ be a minimal element for the lexicographic preorder $\prec$ in $g_t\LL\cap \overset{o}C_1(r)$. By Lemma \ref{lem:lexico}, $v$ is a minimal vector and again $|v_2|> r$ because $g_t\LL\cap B_{\infty}(0,r)=\{0\}$.
	If $w=(w_1,w_2)$ is a minimal vector in $g_t\LL$ with $|w_2|>|u_2|$ then $|w_1|<|u_1|$. By definition of $u$ this implies $|w_1|<r$ which implies that $w\in \overset {o}C_1(r)$ which in turn implies $v\prec w$ and therefore $|w_2|\geq |v_2|$. It follows that $u$ and $v$ are consecutive minimal vectors in $g_t\LL$  and that $B_{\infty}(0,r)\subset C(u,v)$.   It follows that  $g_{-t}u$ and $g_{-t}v$ are consecutive minimal vectors in $\LL$  and since $|v_2|>r$, we have $|e^tv_2|>re^t$. Moreover $|e^{-t}u_1e^tv_2|=|u_1||v_2|\geq r\times r$.
\end{proof}

\begin{lemma}\label{lem:open}
	Let $r$ be positive real number. The set $F$ of unimodular lattices $\LL$ in $\C^2$ such that $\LL\cap B_{\infty}(0,r)$ contains a nonzero vector, is closed in $\SL(2,\C)/\SL(2,\ZZ)$. 
\end{lemma}
\begin{proof}
	Let $(\LL_n=M_n\ZZ^2)_n$ be a sequence of lattices in $F$. Suppose that the sequence converges to a lattice $\LL=M\ZZ^2$. We want to show that $\LL\in F$. We can suppose that the sequence of matrices $(M_n)_n$ converges to $M$ w.l.o.g..  For each $n$, there exists a nonzero vector $X_n\in\ZZ^2$ such that $Y_n=M_nX_n\in B_{\infty}(0,r)$. Changing $X_n$ into $2^kX_n$ for some non-negative integer $k$, we can suppose that $r/2\leq |Y_n|_{\infty}\leq r$. Since the matrices $M_n$ are all invertible and since the sequence $(M_n)$ is convergent, there exists $\delta>0$ such that $\|M_n\|\geq \delta>0$ for all $n$, where $\|A\|$ is the operator norm of the matrix $A$ associated with the sup norm on $\C^2$. Therefore, $|X_n|_{\infty}\leq \frac{r}{\delta}$ for $n$ large enough. Thus there exist a vector $X\in\ZZ^2$ and an increasing sequence of integers $n_k$ such that  $X_{n_k}=X$ for all $k$. 
	Since $|M_{n_k}X_{n_k}|_{\infty}=|Y_{n_k}|_{\infty}\geq r/2$, $MX=\lim_{k\fff\infty}M_{n_k}X=\lim_{k\fff\infty}Y_{n_k}$ is a nonzero vector of $\LL=M\ZZ^2$ in the ball $B_{\infty}(0,r)$, which means that $\LL\in F$.  
\end{proof}

\begin{proof}[End of proof of Theorem \ref{thm:Dirichlet} and Theorem 5 bis] In order to prove that $C'_S(\ttt)\geq r_0$ for almost all $\ttt$, by Lemma \ref{lem:dirichletproof}, it is enough to prove that the set
	\[
	\{\ttt\in [0,1]+i[0,1]:\forall T\geq 0,\forall \eps>0, \exists t\geq T, g_t\LL_{\ttt}\cap B_{\infty}(0,r_0-\eps)=\{0\}\}
	\]
	has full Lebesgue measure in $[0,1]+i[0,1]$. Suppose on the contrary, that the set
	\[
	\{\ttt\in [0,1]+i[0,1]:\exists T\geq 0,\exists \eps>0, \forall t\geq T, g_t\LL_{\ttt}\cap B_{\infty}(0,r_0-\eps)\neq\{0\}\}
	\]
	has positive Lebesgue measure. Then there exist $T\geq 0$ and $\eps>0$ such that the set 
	\[
	N=\{\LL_{\ttt}:\ttt\in [0,1]+i[0,1] \text{ and } \forall t\geq T, g_t\LL_{\ttt}\cap B_{\infty}(0,r_0-\eps)\neq\{0\}\}
	\]
	has positive measure. By definition of $N$, for all $\LL_{\ttt}\in N$ and all $t\geq T$, there exists a nonzero vector $X(\ttt,t)\in \ZZ^{2}$ such that 
	\[
	Y(\ttt,t)=g_tM_{\ttt}X(\ttt,t)\in B_{\infty}(0,r_0-\eps).
	\]
	Let  $\mathcal{H}_{\leq}$ be the subgroup of $\SL(d+1,\C)$
	defined by%
	\[
	\mathcal{H}_{\leq}=\left\{h=\left(
	\begin{array}
		[c]{cc}%
		a & 0\\
		b & a^{-1}
	\end{array}
	\right)\in\SL(2,\C): a\in\C^*,\ b\in\C \right\}.
	\]
	There exists $\delta>0$ such that for all $A\in\mathcal H_{\leq}$
	\[
	|A-Id|_{\infty}\leq \delta\Rightarrow \forall t\geq 0,\|g_tAg_t^{-1}-Id\|\leq \frac{\eps}{2r_0} 
	\]
	where $|M|_{\infty}$ is the sup norm of the matrix $M$ and $\|M\|$ is its operator norm associated with the sup norm. For all $\LL_{\ttt}\in N$,  all $t\geq T$ and all $A\in \mathcal H_{\leq}$ with $|A-Id|_{\infty}\leq \delta$, we have
	\begin{align*}
		|g_tAM_{\ttt}X(\ttt,t)|_{\infty}&=|g_tAg_t^{-1}g_tM_{\ttt}X(\ttt,t)|_{\infty}\\
		&=|(g_tAg_t^{-1}-Id)g_tM_{\ttt}X(\ttt,t)+g_tM_{\ttt}X(\ttt,t)|_{\infty}\\
		&\leq|(g_tAg_t^{-1}-Id)g_tM_{\ttt}X(\ttt,t)|_{\infty}+|g_tM_{\ttt}X(\ttt,t)|_{\infty}\\
		&\leq (\|g_tAg_t^{-1}-Id\|+1)|g_tM_{\ttt}X(\ttt,t)|_{\infty}\\
		&\leq (\frac{\eps}{2r_0}+1)(r_0-\eps)\\
		&\leq r_0-\frac{\eps}{2}.
	\end{align*}
	Therefore,
	\[
	g_t\LL\cap B_{\infty}(0,r_0-\tfrac{\eps}{2})\neq \{0\}
	\]
	for all $\LL\in B_{\mathcal H_{\leq}}(Id,\delta)N$ where $B_{\mathcal H_{\leq}}(Id,\delta)$ is the set of matrices $M$ in the subgroup $\mathcal H_{\leq}$ such that $|M-Id|_{\infty}\leq \delta$ and $B_{\mathcal H_{\leq}}(Id,\delta)N$ is the set of lattices of the shape $A\LL$ with $A\in B_{\mathcal H_{\leq}}(Id,\delta)$ and $\LL\in N$.
	
	Let $U$ be the set of unimodular lattices $\LL$ such that $\LL\cap B_{\infty}(0,r_0-\tfrac{\eps}{2})=\{0\}$. By the choice of $r_0$, the lattice $\LL_0$ is in $U$. By Lemma \ref{lem:open}, $U$ is open. So that $U$ is a nonempty open set and has a positive Haar measure  in $\SL(2,\C)/\SL(2,\ZZ)$.
	The action of the flow $g_t$, $t\in\R$, on $\SL(2,\C)/\SL(2,\ZZ)$ is ergodic, see \cite{BeMa} page 90 (it  is also a direct consequence of  Mautner's lemma and of the fact that $\SL(2,\C)$ is generated by the matrices of the form $\begin{pmatrix} 1&*\\0&1\end{pmatrix}$ and $\begin{pmatrix} 1&0\\ *&1\end{pmatrix}$).
	It follows, by Birkhoff ergodic theorem applied to the flow $g_t$ and to the function $f=1_U$ that for almost all lattices $\LL$, there exist arbitrarily large $t$ such that $g_t\LL\in U$. Now, the set $B_{\mathcal H}(Id,\delta)N$ has positive Haar measure and by construction for all lattice $\LL$ in this set and all $t\geq T$ 
	\[
	g_t\LL\cap B(r_0-\tfrac{\eps}{2})\neq \{0\},
	\]
	and therefore $g_t\LL\notin U$ for all $t\geq T$, a contradiction.
\end{proof}

\section{Search of minimal vectors in a Gauss lattice in $\C^2$}
In this subsection we address the problem of finding two consecutive minimal vectors in a Gauss lattice in $\C^2$. The first step is to find one minimal vector and the second step the next one.

Thanks to the Gauss reduction algorithm it can be done very efficiently.
\subsection{The Gauss reduction algorithm}
Given Gauss lattice in $\C^2$ we want to find a minimal vector in this lattice. This can be done with a Gauss reduction algorithm and the following observation:

\textsl{ If $\LL$ is Gauss lattice in $\C^2$ and if $u$ is a shortest vector of $\LL$ for the standard Hermitian norm then $u$ is minimal vector in $\LL$.}

Indeed, if $u=(u_1,u_2)$ is a shortest vector for the standard Hermitian norm then any vector $v=(v_1,v_2)$ in the cylinder $C(u)$ such that $|v_1|>|u_1|$ or $|v_2|<|u_2|$ has a strictly smaller Hermitian norm. 

Given a basis of a lattice in a two dimensional (real) Euclidean vector space, the Gauss reduction
algorithm provides a reduced basis of the lattice. This algorithm can be
adapted to the case of Gauss lattices in two dimensional $\mathbb{C}$-vector
space equipped with an Hermitian norm. See \cite{LPC} where Gauss reduction algorithm is
proved to work for lattices in $\mathbb{C}^{2}$ over an Euclidean ring of
integers of an imaginary quadratic field. We state their result for lattices over $\ZZ$ without proof.

\begin{definition}
	Let $E$ be a two-dimensional $\mathbb{C}$-vector space equipped with a norm
	$\|.\|$. A basis $(u,v)$ of a Gauss lattice $\Lambda=\mathbb{Z}[i]u+\mathbb{Z}%
	[i]v$ is reduced with respect to the norm $\|.\|$ if $\|u\|=\lambda
	_{1}(\Lambda,\|.\|,\C)$ and $\|v\|=\lambda_{2}(\Lambda
	,\|.\|,\C)$.
\end{definition}

Let $E$ be a two dimensional $\mathbb{C}$-vector space equipped
with an Hermitian norm $|.|_{E}$.

The Gauss reduction algorithm proceed as follows. \newline\textbf{Input:} A basis
$(u,v)$ of a Gauss lattice $\Lambda$ in $E$.

\begin{enumerate}
	\item If $|v|_{E}<|u|_{E}$, exchange $u\leftrightarrow v$.
	
	\item $A:=\operatorname{False}$
	
	\item Main loop: while $A=\operatorname{False}$
	
	\begin{enumerate}
		\item Compute $w=(a+ib)u$ the orthogonal projection of $v$ on the line
		$\mathbb{C }u$.
		
		\item Find the Gaussian integer $p$ closest to $a+ib$ and replace $v$ with $v-pu$.
		
		\item If $|u|_{E}\leq|v|_{E}$, $A:=\operatorname{True}$, else exchange
		$u\leftrightarrow v$.
	\end{enumerate}
\end{enumerate}

\textbf{Output} A reduced basis of $\Lambda$.

\begin{proposition}
	\label{prop:gauss} The above algorithm find a reduced basis of $\Lambda
	=\mathbb{Z}[i]u+\mathbb{Z}[i]v$ for the norm $|.|_{E}$ in finitely many steps.
\end{proposition}

\subsection{Algorithm finding consecutive minimal vectors}

Let $\Lambda$ be a Gauss lattice in $\mathbb{C}^{2}$ and let $u$ be a minimal
vector in $\Lambda$. How do we find a minimal vector $v$ in
$\Lambda$ such that $u$ and $v$ are consecutive ?

This problem can be solved with next proposition we state without proof.

\textbf{Notation.} For a positive real number $t$ denote $|.|_{t}$ the
Hermitian norm on $\mathbb{C}^{2}$ defined by
\[
|(z_{1},z_{2})|_{t}^{2}= |tz_{1}|^{2}+|\tfrac1tz_{2}|^{2}.
\]

\begin{proposition}
	\label{prop:fouille} Let $\Lambda$ be a Gauss lattice in $\mathbb{C}^{2}$ and let $u=(u_{1},u_{2})$ be a minimal vector in $\LL$. Set
	\[
	s=\sqrt{\frac{4}{\pi}|\operatorname{det}_{\mathbb{C}}(\Lambda)|}\, \text{ and
	} t=\frac{s}{|u_{1}|}.
	\]
	Let $(w,w^{\prime})$ be a reduced basis of $\Lambda$ with respect to the norm
	$|.|_{t}$. Then the minimal vectors $v$ such that $u$ and $v$ are consecutive minimal vectors, belong to the set of vectors $zw+z^{\prime}w^{\prime}$ with $z,z'\in\Z[i]$ and $(|z|^{2}+|z^{\prime2}|)< 23$.
\end{proposition}
\begin{proof}
Let $v=(v_{1},v_{2})$ be a minimal
vector in $\Lambda$ such that $u$ and $v$ are consecutive minimal vectors, let $s=\sqrt
{a|\operatorname{det}_{\mathbb{C}}(\Lambda)|}$ where $a$ is a positive
constant and let $t=\frac{s}{|u_1|}$. We will make the choice $a=\frac{4}{\pi}$ only at the end of the proof.
It is enough to prove that $a$ can be chosen so that for all $z,z'\in\C$, $zw+z'w'\in C(u,v)$ implies $|z|^2+|z'|^2<23$.

The sup norm defined by
\[
\|(z_{1},z_{2})\|_{t}=\max(|tz_{1}|,|\tfrac1tz_{2}|)
\]
is bounded below by $\tfrac1{\sqrt{2}}|(z_{1},z_{2})|_{t}$. By Lemma \ref{lem:pigeon}, $\tfrac12 |\det_{\mathbb{C}}(\Lambda)|\leq|u_{1}||v_{2}%
|\leq\tfrac{4}{\pi}|\det_{\mathbb{C}}(\Lambda)|=C|\det_{\mathbb{C}}(\Lambda
)|$. Since,
\begin{align*}
	\lambda_{2}(\Lambda,|.|_{t},\C)  &  \leq\sqrt2\,\lambda_{2}%
	(\Lambda,\|.\|_{t},\C)\\
	&  \leq\sqrt2\max(\|u\|_{t},\|v\|_{t})\\
	&  = \sqrt2\max(s,\frac1{s}|u_{1}||u_{2}|,s\tfrac{|v_{1}|}{|u_{1}|},\frac
	1{s}|u_{1}||v_{2}|)\\
	&  = \sqrt2\max(s,\frac1{s}|u_{1}||v_{2}|),
\end{align*}
with $s=\sqrt{a|\operatorname{det}_{\mathbb{C}}(\Lambda)|}$, we obtain
\begin{align*}
	|w^{\prime}|_{t}=\lambda_{2}(\Lambda,|.|_{t},\C)  &  \leq\sqrt{2}
	\max(s,\frac1{s}|u_{1}||v_{2}|)\\
	&  \leq\sqrt{2}\max(s,\frac{C}{s}|\operatorname{det}_{\mathbb{C}}(\Lambda)|)\\
	&  =\sqrt{2}\max(1,\frac{C}{a})s=bs.
\end{align*}
Now by Hadamard inequality,
\[
|\operatorname{det}_{\mathbb{C}}(\Lambda)|=|\det(w,w^{\prime}%
)|=|\operatorname{det}_{|.|_{t}}(w,w^{\prime})|\leq|w|_{t}|w^{\prime}|_{t}
\]
where $\operatorname{det}_{|.|_{t}}(w,w^{\prime})$ is the determinant computed
in a $|.|_{t}$-orthonormal basis, hence
\[
|w|_{t}\geq\frac{|\operatorname{det}_{\mathbb{C}}(\Lambda)|}{|w^{\prime}|_{t}%
}\geq\frac{s}{ab}.
\]

Again, since $|u_{1}||v_{2}|\leq C|\det_{\mathbb{C}}(\Lambda)|$ and
$s=\sqrt{a|\operatorname{det}_{\mathbb{C}}(\Lambda)|}$, the cylinder $C(|u_{1}%
|,|v_{2}|)$ is included in the closed ball of radius $\max(1,\frac{C}{a})s$
associated with this sup norm $\|.\|_{t}$. Therefore, it is enough to find a constant  $A$ such
that $|z|^{2}+|z^{\prime2}|> A$ implies $|zw+zw^{\prime}|_{t}>\sqrt{2}%
\max(1,\frac{C}{a})s=bs$. Now, since the basis $(w,w^{\prime})$ is reduced, $|w\pm w'|_t^2\geq |w'|_t^2$, which implies $|\Re\langle w,w'\rangle_t|\leq \tfrac1{2} |w|_t^2$, and $|w\pm iw'|_t^2\geq |w'|_t^2$ implies $|\Im\langle w,w'\rangle_t|\leq \tfrac1{2} |w|_t^2$ as well. Hence $|\langle w,w'\rangle_t|\leq \tfrac1{\sqrt2} |w|_t^2$, and 
\begin{align*}
	|zw+z^{\prime}w^{\prime}|_{t}^{2}  &  \geq|z|^{2}|w|_{t}^{2}+|z'|^2|w^{\prime}|_{t}^{2}-\sqrt{2}|z||z^{\prime}| |w|_{t}^{2}\\
	&  \geq(1-\tfrac{1}{\sqrt{2}})(|z|^{2}+|z'|^{2})|w|_{t}^{2}\\
	&  \geq(\tfrac{\sqrt{2}-1}{\sqrt2})(|z|^{2}+|z'|^{2})(\tfrac{s}{ab})^{2}.  
\end{align*}
Therefore, if $|z|^{2}+|z^{\prime2}|> A$, then $|zw+z^{\prime}w^{\prime}|_{t}^{2}\geq(\tfrac{\sqrt{2}-1}{\sqrt2})(\tfrac{1}{ab})^{2}s^{2}A.$
So we are done if $A\geq\tfrac{\sqrt{2}}{\sqrt2-1}a^{2}b^{4}$. The value $a=C$
minimizes $a^{2}b^{4}$ and gives (recall that $C=\tfrac4{\pi}$),
\[
\tfrac{\sqrt{2}}{\sqrt2-1}a^{2}b^{4}=\tfrac{4\sqrt{2}C^{2}}{\sqrt2-1}
=\tfrac{64\sqrt{2}}{(\sqrt2-1)\pi^{2}}=22.139...
\]
\end{proof}

Now, the algorithm to find two consecutive minimal vectors in a Gauss lattice $\LL\subset \C^2$ goes as follows:
\begin{itemize}
	\item Use the Gauss reduction algorithm to find a reduced base $(u,u')$ in $\LL$ with respect to the standard Hermitian norm. $u$ is the first minimal vector of the pair.
	\item Use once again the Gauss reduction algorithm to find a reduced base $(w,w')$ with respect to the Hermitian norm $|.|_t$ where $t$ is the parameter associated with $u$ defined in the proposition. 
	\item Find the minimal element for the lexicographic order among the vectors  $zw+z^{\prime}w^{\prime}$ with $z,z'\in\Z[i]$ and $(|z|^{2}+|z^{\prime2}|)< 23$ that are in the infinite cylinder $C_1(u)$.
\end{itemize}

\section{Miscellaneous questions and comments}
\subsection{The Hurwitz's algorithm in the space of bases of $\C^2$}
Consider a basis
$u=(u_{1},u_{2}),$ $v=(v_{1},v_{2})$ of the vector space $\C^2$. We want to define $H$ a map that associates a new basis $u'=(u'_1,u'_2)$, $v'=(v'_1,v'_2)$ to each basis $u$, $v$. This map is defined only when $u_1\neq 0$ and $v_1\neq 0$. The first vector $u'$ is defined by $u'=v$ and $v'$ is defined as follows.  Set $w_1=\tfrac{v_1}{u_1}$. We define $v^{\prime}=(v_{1}^{\prime},v_{2}%
^{\prime})=u-gv$  where $g$ is the Gaussian integer such that%
\[
\frac{v_{1}^{\prime}}{v_{1}}=\frac{u_{1}-gw_{1}u_{1}}{w_{1}u_{1}}=\frac
{1}{w_{1}}-a\in S=[-\tfrac{1}{2},\tfrac{1}{2}[+[-\tfrac{1}{2},\tfrac{1}{2}[i.
\]
For the new basis $u'$, $v'$, we have
\[
\frac{v_{1}^{\prime}}{u_{1}^{\prime}}=\frac{v_{1}^{\prime}}{v_{1}}=\frac
{1}{w_{1}}-g=w_{1}^{\prime}\in S
\]
and therefore $v_{1}^{\prime}=w_{1}^{\prime}u_{1}^{\prime}$ with
$w_{1}^{\prime}\in S$. We recognize the Hurwitz continued fraction algorithm  applied to $w_{1}$. The map $H$ is defined on the set of  pairs of independent vectors $(u,v)$ such that first coordinates of $u$ and $v$ are nonzero. Observe that $\ZZ u'+\ZZ v'=\ZZ u+\ZZ v$ and that $\det_{\C}H(u,v)=-\det_{\C}(u,v)$.

\begin{remark}In Theorem \ref{thm:continuedfraction}, the map $T_G$ was defined with a good choice of the Gaussian integer $g$ and of $a\in\{1,1+i\}$ such that 
	$|\tfrac{a}{w_1}-g|<1$ while in the Hurwitz algorithm there is a unique Gaussian integer $g$ such $\frac
	{1}{w_{1}}-g=w_{1}^{\prime}\in S$.
\end{remark}

There are two  simple questions: 
\begin{itemize}
	\item If $(u,v)$ is a pair of consecutive minimal vectors in a Gauss lattice $\LL\subset \C^2$ and $(u',v')=H(u,v)$ is defined, is it true that $v'$ is a minimal vector in $\LL$ ?\\
	\item Is it possible to continue the process : if $v'$ is still a minimal vector and $H(u',v')=(u'',v'')$, is $v''$ minimal  ?
\end{itemize}

\subsection{Ergodic theory and the first return map}
Let $\ttt$ be in $\C$ and let $X_n(\LL)=(x_n(\ttt),y_n(\ttt))$, $n\in\N$ be the sequence of minimal vectors of the lattice $\LL_{\ttt}$. We can ask several questions about the quantities $x_n(\ttt)$ and $y_n(\ttt)$.
\begin{itemize}
	\item (Levy-Khintchin theorem) Show that for almost all $\ttt\in\C$
	\[
	\lim_{n\fff\infty}\frac1{n}\ln|y_n(\ttt)|=C
	\]
	where $C$ is a constant that can be computed with the Haar measure of $\SL(2,\C)/\SL(2,\ZZ)$ and the induced measure $\nu$ (see Theorem \ref{thm:densityinvariantmeasure2}).
	\item (Bosma-Jager-Wiedijk theorem)
	Show that for almost all $\ttt\in\C$, the sequence of probabilities
	\[
	\frac1{n}\sum_{k=0}^{n-1}\delta_{x_n(\ttt)y_n(\ttt)}
	\]
	converges in measure to a probability $\lambda$ in $\C$. Show that $\lambda$ has a density with respect to the Lebesgue measure and compute this density. The question can be studied with $y_{n+1}(\ttt)x_n(\ttt)$ instead of $x_n(\ttt)y_n(\ttt)$.  
\end{itemize}
For these two questions the method in \cite{Che} should lead to the existence of the limit almost everywhere. The explicit computations of the limits $C$ and $\lambda$ could be more difficult.\\

As we have seen the core of the first return map is the map $T_G$ (see subsection \ref{subsec:firstreturn}). In the definition of the map $T_G$ two coefficients appear : $a\in\{1,1+i\}$ and $g\in\ZZ$.  By Proposition \ref{prop:index2}, if $a=1+i$ for some iterate of $T_G$, the next iterate should be with $a=1$.   
\begin{itemize}
	\item More generally, find the succession laws for the coefficients.
	\item What is the almost-sure frequency of $a=1$ when computing the sequence of iterates of $T_G$?
	\item Is there a Borel-Bernstein theorem for the coefficient $g$?
\end{itemize}

\section{Appendix 1, Gauss lattices}

\begin{definition}
	Let $E$ be a finite dimensional $\mathbb{C}$-vector space. A subset
	$\Lambda$ in $E$ is a Gauss lattice if it is a $\mathbb{Z}[i]$-submodule of $E$, if it is a discrete subset of
	$E$ and if it generates the vector space $E$.
\end{definition}

\begin{lemma}
	Let $E$ be a $\mathbb{C}$-vector space of dimension $n$ and let $\Lambda$ be a
	Gauss lattice in $E$. Then there exists a basis $u_{1},\dots,u_{n}$
	of $E$ such that
	\[
	\Lambda=\oplus_{j=1}^{n}\mathbb{Z}[i]u_{j}.
	\]
	
\end{lemma}

\begin{proof}
	Denote $\|.\|$ a Hermitian norm in $E$ (an Hermitian structure is used only for convenience). We proceed by induction. If $n=1$, $E=\mathbb{C }u$ and $\Lambda
	=\mathbb{Z}[i]\lambda u$ where $\lambda u$ is a shortest vector in $\Lambda$.
	Indeed for all $zu\in\Lambda$ there exits $p\in\mathbb{Z}[i]$ such that
	$|\tfrac{z}{\lambda}-p|<1$, hence $\|zu-p\lambda u\|=|z-p\lambda|\|u\|<|\lambda
	|\|u\|$ and therefore $z=p\lambda$.
	
	Suppose the result holds for all $n-1$-dimensional vector space. Let $E$ be a
	$\mathbb{C}$-vector space with $\dim_{\mathbb{C}}E=n$ and let $\Lambda$ be a Gauss
	lattice in $E$. Since $\LL$ generates the vector space $E$, there is a basis $u_{1},\dots,u_{n}$ of $E$
	with $u_{1},\dots,u_{n}\in\Lambda$. Let $F$ be the vector space spanned by
	$u_{1},\dots,u_{n-1}$. By induction hypothesis there exists a basis
	$v_{1},\dots,v_{n-1}$ of $F$ such that $F\cap\Lambda=\oplus_{j=1}%
	^{n-1}\mathbb{Z}[i]v_{j}$. The orthogonal projection $\Lambda^{\prime}$ of
	$\Lambda$ on the line $D$ orthogonal to $F$ is discrete. Indeed suppose there
	is a sequence $w_{n}\in\Lambda^{\prime}$ of nonzero vectors which converges to
	zero. We can suppose
	that the vectors $w_{n}$, $n\in\mathbb{N}$, are distinct. For each $n$, let $w^{\prime}_{n}\in\Lambda$ be a vector whose  projection 
	is $w_{n}$. The vectors $w^{\prime}_{n}$ can be chosen in order that
	their orthogonal projection on $F$ is in the bounded set $\{\sum_{j=1}%
	^{n-1}z_{j}v_{j}\in F:(\Re z_{j},\Im z_{j})\in [0,1]^2  \}$ so that the $w^{\prime}%
	_{n}$ are distinct and in a bounded set, a contradiction. Therefore $\LL'$ is discrete. Let $v'_n$ be shortest vector nonzero vector in $\LL'$. Since $\LL'$ is a Gauss lattice, the step $n=1$ of the induction implies that $\LL'=\ZZ v'_n$. Finally choose any vector $v_n\in \LL$ whose projection on $\LL'$ is $v'_n$. If $v\in\LL$ then its projection $v'$ on $\LL'$ is in $\ZZ v'_n$. It follows that $v'=gv'_n$ for some $g\in\ZZ$. Therefore, the projection of $v-gv_n$ is $0$ which implies that $v-gv_n\in F\cap\LL$. We conclude that $v_1,\dots,v_n$ is a $\ZZ$ basis of $\LL$.
\end{proof}

A direct adaptation of Theorem I page 11 of Cassels' book, \cite{Ca},
An introduction to the geometry of numbers, shows that

\begin{theorem}\label{thm:base}
	Let $E$ be a $n$-dimensional $\C$-vector space, let $\Lambda$ be a Gauss lattice in $E$ and let $L\subset\Lambda$ be a lattice in $E$. 
	\newline A. To
	every basis $b_{1},\dots,b_{n}$ of $\Lambda$, there can be found a basis 
	$a_{1},\dots,a_{n}$ of $L$ of the shape
	\begin{align*}
		\left\{
		\begin{array}
			[l]{l}%
			a_{1}=z_{11}b_{1}\\
			a_{2}=z_{21}b_{1}+z_{21}b_{2}\\
			\hspace{0.6cm}\vdots\\
			a_{n}=z_{n1}b_{1}+\cdots+z_{nn}b_{n}%
		\end{array}
		\right.
	\end{align*}
	where the $z_{ij}$ are in $\mathbb{Z}[i]$ and $z_{ii}$ are nonzero for all
	$i$.\newline B. Conversely, to every basis $a_{1},\dots,a_{n}$ of $L$, there
	exists a basis $b_{1},\dots,b_{n}$ of $\Lambda$ such the above system holds.
\end{theorem}

\begin{proof}
	[Proof of A]Pick one basis of $\Lambda$ and one basis of $L$ and call $D$ the
	determinant of the second in the first. The determinant $D$ is a Gaussian integer and Cramer
	formula shows that $D\Lambda\subset L$.
	
	For each $i\in\{1,\dots,n\}$ there exist points $a_{i}$ in $L$ of the shape
	\[
	a_{i}=z_{i1}b_{1}+\dots+z_{ii}b_{i}
	\]
	where the $z_{ij}$ are Gaussian integers and $z_{ii}\neq 0$ for $Db_{i}\in L$. We choose for $a_{i}$
	such an element in $L$ for which $|z_{ii}|$ is as small as possible but zero.
	We are going to show that $a_{1},\dots,a_{n}$ is a basis of $L$. Since
	$a_{1},\dots,a_{n}$ are in $L$, so is every vector $w=w_{1}a_{1}+\dots
	+w_{n}a_{n}$ where $w_{1},\dots,w_{n}$ are Gaussian integers. Suppose by
	contradiction that there exists a vector $c$ of $L$ not of the latter shape.
	Since $c$ is $\Lambda$, $c=t_{1}b_{1}+\dots+t_{k}b_{k}$ where $1\leq k\leq n$,
	$t_{k}\neq0$ and $t_{1},\dots,t_{k}$ are Gaussian integers. If there are several
	such $c$, then we choose one for which $k$ is minimal. Now since $z_{kk}\neq
	0$, we may choose a Gaussian integer $s$ such that
	\[
	|t_{k}-sz_{kk}|<|z_{kk}|.
	\]
	The vector
	\[
	c-sa_{k}=(t_{k}-sz_{k1})b_{1}+\dots+(t_{k}-sz_{kk})b_{k}
	\]
	is in $L$ since $a_{k}$ and $c$ are; but it is not of the shape $w_{1}%
	a_{1}+\dots+w_{n}a_{n}$ since $c$ is not. Hence $t_{k}-sz_{kk}$ cannot be zero
	by assumption that $k$ was minimal. But then $|t_{k}-sz_{kk}|<|z_{kk}|$
	contradicts the assumption that the nonzero $|z_{kk}|$ was minimal.
\end{proof}

\begin{proof}
	[Proof of B]Let $a_{1},\dots,a_{n}$ be some basis of $L$. Since $D\Lambda
	\subset L$, by part A, there exists a basis $Db_{1},\dots,Db_{n}$ of
	$D\Lambda$ such that
	\begin{align*}
		\left\{
		\begin{array}
			[l]{l}%
			Db_{1}=z_{11}a_{1}\\
			Db_{2}=z_{21}a_{1}+z_{21}a_{2}\\
			\hspace{0.6cm}\vdots\\
			Db_{n}=z_{n1}a_{1}+\cdots+a_{nn}b_{n}%
		\end{array}
		\right.
	\end{align*}
	with $z_{i,j}$ Gaussian intergers and $z_{ii}\neq0$. Solving the above system we
	can express $a_{1},\dots,a_{n}$ in the basis $b_{1},\dots,b_{n}$, we obtain a
	triangular system with coefficients in the fields $\mathbb{Q}(i)$. But
	$b_{1},\dots,b_{n}$ is basis of $\Lambda$ and the $a_{i}$ are in $\Lambda$ so
	the coefficients must be Gaussian integers.
\end{proof}

The norm $\|.\|$ we consider over $\mathbb{C}$-vector space are suppose to
verify
\[
\|\lambda u\|=|\lambda|\| u\|
\]
for all vector $u$ of the vector space and all complex number $\lambda$.

\begin{definition}\label{def:minima}
	Let $E$ be a finite dimensional $\mathbb{C}$-vector space equipped with a norm
	$\|.\|$ and let $\Lambda$ be a discrete subset in $E$. For $\mathbb K=\R$ or $\C$ and for $1\leq j \leq \dim_{\mathbb K}E$, the $j$-th minimum of $\Lambda$ with respect to the norm $\|.\|$ and to the field $\mathbb K$ is the
	infimum of all the real numbers $\lambda$ such that there exist $j$ $\mathbb K$-linearly independent vectors in $\LL$ with norms $\leq \lambda$.
	It is denoted $\lambda_{j}(\Lambda,\|.\|,\mathbb K)$ or simply
	$\lambda_{j}(\Lambda,\mathbb{K})$ or even $\lambda_{j}$ when there is no ambiguity.
\end{definition}

\begin{lemma}
	Let $E$ be a $\mathbb{C}$-vector space of dimension $n$ equipped with a norm
	$\|.\|$ and let $\Lambda$ a Gauss lattice in $E$. Then for
	$j=1,\dots,n$
	\[
	\lambda_{j}(\Lambda,\|.\|,\C)=\lambda_{2j-1}(\Lambda
	,\|.\|,\mathbb{R})=\lambda_{2j}(\Lambda,\|.\|,\mathbb{R}).
	\]
	
\end{lemma}

\begin{proof}
	If $u_1,\dots,u_j$ are $j$ $\C$-linearly independent vectors with norms $\leq \lambda$ then $u_1,iu_1,\dots,u_j,iu_j$ are $2j$ $\R$-linearly independent with norms $\leq \lambda$, therefore $\lambda_{j}(\Lambda,\|.\|,\C)\geq\lambda_{2j}%
	(\Lambda,\|.\|,\mathbb{R})$. Since a $\C$-vector space of $\C$-dimension $\leq j-1$ has a real dimension $\leq 2j-2$, $2j-1$ $\mathbb{R}$-linearly independent
	vectors $u_{1},\dots,u_{2j-1}$ in $\Lambda$ generate a $\mathbb{C}$-vector
	space of dimension $>j-1$, Therefore, $\lambda_{j}(\Lambda,\|.\|,\C)\leq\lambda_{2j-1}(\Lambda,\|.\|,\mathbb{R})$.
\end{proof}

\section{Appendix 2, computing the distance to $\DD$}\label{sec:distance}
There is a simple algorithm that calculate the distances from a complex number $z$ to the regions $\DD$, $\CC$ and $\mathcal T$. We explain it for the distance to the region  $	\DD=  \{z\in\C:|z|<1,\, \dd(z,1)>1,\,\dd(z,1-i)> 1 \}$. The distances to $\CC$ and to $\mathcal T$ can be calculated the same way. The complex plane is the union of seven regions $\mathcal D_0,\dots,\mathcal D_6$. For each of these regions, there is a simple formula giving the distance $\dd(z,\DD)$:
\begin{enumerate}
	\item If $z\in\mathcal D_0=\overline{\DD}$ then $\dd(z,\DD)=0$.
	\item If $z\in\mathcal D_1=\{z\in D(1,1):\arg (z-1)\in [\tfrac{2\pi}{3},\tfrac{7\pi}{6}]\}$ then $\dd(z,\DD)=1-|z-1|$.
	\item If $z\in\mathcal D_2=\{z\in \C:\arg z\leq \tfrac{\pi}{3},\,\arg(z-1)\in[-\tfrac{\pi}{12},\tfrac{2\pi}{3}]\}$ then $\dd(z,\DD)=\dd(z,z_2)$ where $z_2=\tfrac12+\tfrac{\sqrt 3}{2}i$. 
	\item If $z\in\mathcal D_3=\{z\in\C:|z|\geq 1,\arg z\in[\tfrac{\pi}{3},\tfrac{3\pi}{2}]\}$ then $\dd(z,\DD)=|z|-1$.
	\item If $z\in \mathcal D_4=\{z\in \C: \Re z\geq 0, \arg(z-1+i)\in[\pi,\tfrac{23}{12}\pi]\}$ then $\dd(z,\DD)=\dd(z,-i)$.
	\item If $z\in\mathcal D_5=\{z\in D(1-i,1):\arg(z-1+i)\in[\tfrac{5}{6}\pi,\pi]\}$ then $\dd(z,\DD)=1-|z-1+i|$.
	\item If $z\in\mathcal D_6=\{z\in\C:\arg(z-1)\in[\tfrac{7\pi}{6},2\pi],\arg(z-1+i)\in[-\tfrac{\pi}{12},\tfrac{5\pi}{6}]\}$ then $\dd(z,\DD)=\dd(z,z_1)$ where $z_1=1-\tfrac{\sqrt 3}{2}-\tfrac12 i$.
\end{enumerate}

\begin{figure}[H]
	\includegraphics[width=10cm]{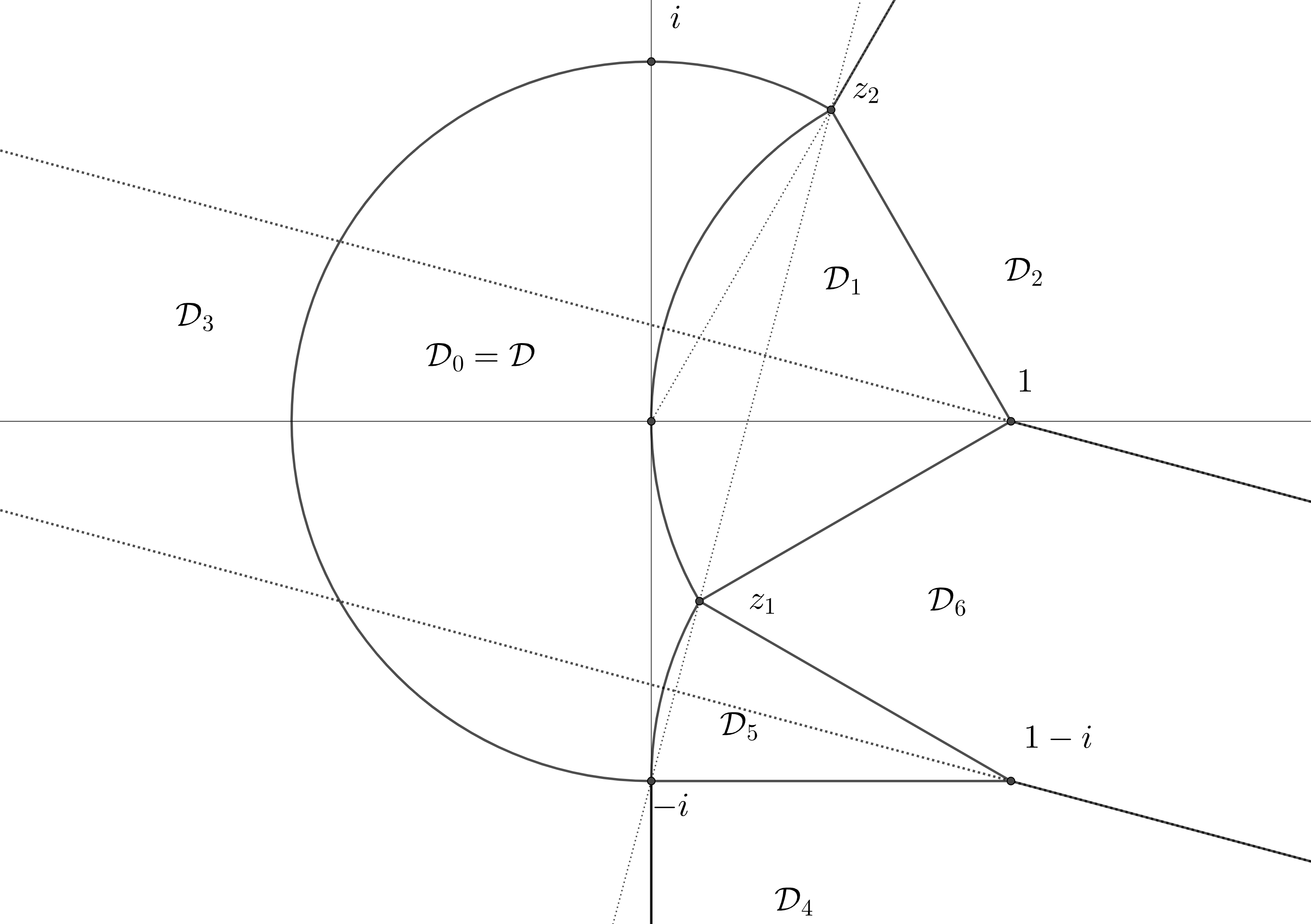}
	\caption{Distance to $\DD$.}
\end{figure}
Furthermore, it is easy to check whether a point $z$ belongs a region $\mathcal D_j$. For instance $z\in \mathcal D_1$ if and only if
\[
|z-1|\leq 1 \text{ and } \Im(\frac{z-1}{z_2-1})\geq 0 \text{ and } \Im(\frac{z-1}{z_1-1})\leq 0. 
\]

\end{document}